\documentclass[12 pt]{amsart}

\usepackage{fullpage} 
\usepackage{diagbox}

\usepackage{hyperref}
\usepackage{etex}
\usepackage[shortlabels]{enumitem}
\usepackage{amsmath}
\usepackage{amsxtra}
\usepackage{amscd}
\usepackage{amsthm}
\usepackage{amsfonts}
\usepackage{amssymb}
\usepackage{eucal}
\usepackage{cleveref}
\usepackage[all]{xy}
\usepackage{graphicx}
\usepackage{tikz}
\usetikzlibrary{cd}
\usetikzlibrary{fit, patterns}
\usepackage{mathrsfs}
\usepackage{subfiles}
\usepackage{mathpazo}
\usepackage[colorinlistoftodos, textsize=tiny]{todonotes}
\usepackage{morefloats}
\usepackage{pdfpages}
\usepackage{thm-restate}
\usepackage[utf8]{inputenc}
\usepackage[T1]{fontenc}
\usepackage{epigraph}
\usepackage{csquotes}
\usepackage{moreverb}
\usepackage[margin=1in]{geometry}
\usepackage{adjustbox}
\usepackage{accents}
\usepackage{xparse}
\usepackage{wrapfig}

\graphicspath{ {images/} }

\RequirePackage{color}
\definecolor{myred}{rgb}{0.75,0,0}
\definecolor{mygreen}{rgb}{0,0.5,0}
\definecolor{myblue}{rgb}{0,0,0.65}

\usepackage{hyperref}
\hypersetup{citecolor=blue}
\usepackage{tikz}
\usetikzlibrary{matrix,arrows,decorations.pathmorphing}

\theoremstyle{plain}
\newtheorem{theorem}[subsubsection]{Theorem}
\newtheorem{proposition}[theorem]{Proposition}
\newtheorem{lemma}[theorem]{Lemma}
\newtheorem{corollary}[theorem]{Corollary}
\theoremstyle{definition}
\newtheorem{definition}[theorem]{Definition}
\newtheorem{construction}[theorem]{Construction}
\newtheorem{remark}[theorem]{Remark}
\newtheorem{example}[theorem]{Example}

\newtheorem{convention}[theorem]{Convention}

\newtheorem{conjecture}[theorem]{Conjecture}

\newtheorem{notation}[theorem]{Notation}

\theoremstyle{remark}
\numberwithin{equation}{section}

\newcommand\nc{\newcommand}
\nc\on{\operatorname}
\nc\renc{\renewcommand}

\newcommand*{\shom}{\mathscr{H}\kern -.5pt om}
\newcommand*{\stor}{\mathscr{T}\kern -.5pt or}
\newcommand*{\sext}{\mathscr{E}\kern -.5pt xt}

\newcommand \cC{{\mathcal C}}

\makeatletter
\providecommand\@dotsep{5}
\renewcommand{\listoftodos}[1][\@todonotes@todolistname]{%
\@starttoc{tdo}{#1}}
\makeatother

\makeatletter
\def\Ddots{\mathinner{\mkern1mu\raise\p@
		\vbox{\kern7\p@\hbox{.}}\mkern2mu
		\raise4\p@\hbox{.}\mkern2mu\raise7\p@\hbox{.}\mkern1mu}}
\makeatother

\newtoggle{draft}
 \toggletrue{draft}

\iftoggle{draft} {
	\newcommand{\NB}[1]{\todo[color=gray!40]{#1}}
	\newcommand{\TODO}[1]{\todo[color=red]{#1}}
}{ 
	\newcommand{\NB}[1]{}
	\newcommand{\TODO}[1]{}
	\renewcommand{\todo}[1]{}
	\renewcommand{\todo}[1]{}
}

\newcommand{\customlabel}[2]{\protected@write \@auxout {}{\string \newlabel {#1}{{#2}{\thepage}{#2}{#1}{}} }\hypertarget{#1}{#2}}

\DeclareMathOperator\id{id}

\DeclareMathOperator\Alg{Alg}
\DeclareMathOperator\Sp{Sp}
\DeclareMathOperator\Fun{Fun}
\DeclareMathOperator\Spc{Spc}

\DeclareMathOperator\cof{cof}

\DeclareMathOperator\unit{\mathbb{1}}

\renewcommand\hom{\mathrm{Hom}}

\DeclareMathOperator\Ind{Ind}

\DeclareMathOperator\spec{Spec}

\DeclareMathOperator\Set{Set}
\DeclareMathOperator\Mod{Mod}

\DeclareMathOperator\im{im}

\DeclareMathOperator\map{map}

\newcommand\Top{\mathrm{Top}}

\DeclareMathOperator\gal{Gal}
\DeclareMathOperator\Spec{Spec}
\DeclareMathOperator\tr{tr}
\DeclareMathOperator\frob{Frob}

\DeclareMathOperator\cl{Cl}

\DeclareMathOperator\Hur{Hur}
\DeclareMathOperator\Conf{Conf}

\DeclareMathOperator\aut{Aut}

\newcommand\QQ{\mathbb{Q}}
\newcommand\ZZ{\mathbb{Z}}

\newcommand\RR{\mathbb{R}}
\newcommand\NN{\mathbb{N}}

\newcommand\EE{\mathbb{E}}

\DeclareMathOperator\ord{ord}

\DeclareMathOperator\fib{fib}

\DeclareMathOperator\conf{Conf}

\DeclareMathOperator\colim{colim}

\DeclareFontFamily{U}{wncy}{}
\DeclareFontShape{U}{wncy}{m}{n}{<->wncyr10}{}
\DeclareSymbolFont{mcy}{U}{wncy}{m}{n}
\DeclareMathSymbol{\Sha}{\mathord}{mcy}{"58}

\newcommand{\hur}[4]{[\operatorname{Hur}^{#1, #3}_{#2,#4}/#1]}
\newcommand{\phur}[4]{\operatorname{Hur}^{#1, #3}_{#2,#4}}
\newcommand{\chur}[4]{[\operatorname{CHur}^{#1, #3}_{#2,#4}/#1]}
\newcommand{\cphur}[4]{\operatorname{CHur}^{#1, #3}_{#2, #4}}
\newcommand{\boundarycphur}[5]{\operatorname{CHur}^{#1, #3, #5}_{#2, #4}}

\newcommand{\hurc}[3]{[\operatorname{Hur}^{#1, #3}_{#2}/#1]}
\newcommand{\phurc}[3]{\operatorname{Hur}^{#1, #3}_{#2}}
\newcommand{\churc}[3]{[\operatorname{CHur}^{#1, #3}_{#2}/#1]}
\newcommand{\cphurc}[3]{\operatorname{CHur}^{#1, #3}_{#2}}

\DeclareMathOperator\surj{Surj}

\def\listtodoname{List of Todos}
\def\listoftodos{\@starttoc{tdo}\listtodoname}
\title{The Cohen--Lenstra moments over function fields via the stable homology of non-splitting Hurwitz spaces}

\author{Aaron Landesman}
\author{Ishan Levy}

\usepackage{microtype}
\begin{document}

\date{October 29, 2024}
\begin{abstract}
	We compute the average number of surjections from class groups of
	quadratic function fields over $\mathbb F_q(t)$ onto finite odd order
	groups $H$, once $q$ is sufficiently large. These yield the first known
	moments of these class groups, as predicted by the Cohen--Lenstra heuristics, apart from the case
	$H = \mathbb Z/3\mathbb Z$. The key input to this result is a
	topological one, where we compute the stable rational homology groups of Hurwitz
	spaces associated to non-splitting conjugacy classes.
\end{abstract}

\maketitle
\tableofcontents

\section{Introduction}

The Cohen--Lenstra heuristics, introduced by Cohen and Lenstra in
\cite{cohen1984heuristics}, predict the distribution of the odd part
of class groups of quadratic number fields, and have been one of the driving conjectures in
arithmetic statistics over the last four decades.
If $K$ is a quadratic extension of $\mathbb Q$, and $\mathscr O_K$ is its ring
of integers, the Cohen--Lenstra heuristics predict the distribution of the odd
part of $\on{Cl}(\mathscr O_K)$.
Next, let $\mathbb F_q$ be a finite field with odd characteristic,
$K$ be a quadratic extension of $\mathbb F_q(t)$,
and $\mathscr O_K$ denote the normalization of $\mathbb F_q[t]$
in $K$. The natural extension of the Cohen--Lenstra heuristics to global function
fields similarly predicts the distribution of the odd
part of $\on{Cl}(\mathscr O_K)$.

One of the primary approaches to determining this distribution is to compute its
moments, by which we mean the average number of surjections $\left|
\surj(\cl(\mathscr O_K), H) \right|$, for $H$ a finite abelian group of odd
order, as $K$ ranges over imaginary quadratic fields or real quadratic fields.
As further motivation for considering these moments, it was shown in 
\cite[Theorem 1.3]{wangW:moments-and-interpretations}
that the Cohen--Lenstra distribution for imaginary
quadratic fields or real quadratic fields is determined by its associated
moments. Hence, if one were
able to compute all conjectured moments, one would prove the Cohen--Lenstra
heuristics.

Although the problem of computing the average size of the $\ell$-torsion
(corresponding to the $\mathbb Z/\ell \mathbb Z$ moment) in the
class groups of quadratic fields may appear substantially more tractable than
determining the entire distribution of these class groups, the only $\ell$ for
which this has been carried out is $\ell = 3$.
This was computed
over $\mathbb Q$ by Davenport and Heilbronn \cite{davenportH:density-discriminants-cubica}
and the analog over function fields has been verified by Datskovsky and Wright
in 
\cite{datskovsky-wright:density-of-discriminants-of-cubic-extensions}.
To the best of our knowledge, since 1988, no additional odd order moments of class groups
of quadratic fields have been
computed.
We note that the original Cohen--Lenstra conjectures have been extended from
odd order abelian groups to all finite abelian groups by Gerth, and the 
limiting distribution
associated to $2^\infty$-part of the class group in this sense has been computed by Smith
\cite[Theorem 1.9]{smith:the-distribution-of-selmer-groups-1}.

Over function fields of the form $\mathbb F_q(t)$, significant progress was made
when a weaker version of these
moments was computed.
Namely, in
\cite[Theorem 3.1]{Achter:cohenQuadratic}, building on
\cite{Achter:distributionClassGroups},
Achter computes these moments where one first takes a large $q$ limit.
Using techniques from homological stability, Ellenberg--Venkatesh--Westerland
were able to strengthen Achter's results by computing a version of these moments where one first takes a large
height limit, and only after takes a large $q$ limit
\cite{EllenbergVW:cohenLenstra}. But even via this approach, no moment (other
than the case $H = \mathbb Z/3 \mathbb Z$) over a fixed $\mathbb
F_q(t)$ has been computed, or even shown to exist.

In this paper, we
compute the moment associated to {\em any} odd order abelian group
associated to class groups of quadratic extensions of $\mathbb F_q(t)$ over
suitably large finite fields $\mathbb F_q$ of suitable characteristic.
Our main new input which lets us accomplish this is a computation of the stable rational
homology of certain Hurwitz spaces.

\subsection{Results toward the Cohen--Lenstra heuristics}

Let $q$ be an odd prime power and 
let $\mathcal{MH}_{n,q}$ denote the set of function fields $K$
of monic smooth hyperelliptic curves of the form $y^2 = f(x)$, where $f(x)$ is a monic
squarefree degree $n$ polynomial with coefficients in $\mathbb F_q$.
(We consider monic smooth hyperelliptic curves to be the analog of quadratic number fields in
the function field case.)
Let $\mathscr O_K$ denote the normalization of $\mathbb F_q[t]$
in the quadratic extension $K$ and let $\cl(\mathscr O_K)$ denote its class group.

\begin{theorem}
	\label{theorem:moments}
	Suppose $H$ is a finite abelian group of odd order.
	Let $q$ be an odd prime power with
	$\gcd( |H|, q(q-1)) = 1$.
	There is an integer $C$, depending only on $H$, so that
	if $q > C$ and $i \in \{0,1\}$,
	\begin{align}
		\label{equation:moments}
		\lim_{\substack{n \to \infty \\ n \equiv i \bmod 2 }}
		\frac{\sum_{K \in \mathcal{MH}_{n,q}}
	\left|\surj(\cl(\mathscr O_K), H) \right|} {\sum_{K \in \mathcal{MH}_{n,q}} 1} =
	\begin{cases}
		1 & \text{ if } i =1 \\
		\frac{1}{|H|} & \text{ if } i = 0.
	\end{cases}
	\end{align}
\end{theorem}

\autoref{theorem:moments} is the special case of
\autoref{theorem:moments-roots-of-unity}, stated below, where we take the number
$h$, defined there, to be $1$.

\begin{remark}
\label{remark:}
Even though for any given $H$, we compute the $H$-moment of the class group of
quadratic fields over $\mathbb F_q(t)$ for sufficiently large $q$, we
do not prove the Cohen--Lenstra heuristics. Although it is true that knowledge of
all moments do determine the Cohen--Lenstra distribution, for a fixed value of
$q$, we are only able to
compute moments associated to sufficiently small groups $H$ relative to $q$.
We do so by computing the {\em stable} homology of related Hurwitz spaces.
If one were also able to obtain a better understanding of the
{\em unstable} homology of these Hurwitz spaces, together with the eigenvalues
of the Frobenius
action,
one might be able to
verify the entire predictions of the Cohen--Lenstra heuristics over $\mathbb
F_q(t)$.
\end{remark}

\begin{remark}
	It is possible to make the constant $C$ in \autoref{theorem:moments}
	explicit: see the discussion in \autoref{remark:explicitconstant}.
\end{remark}

\begin{remark}
	\label{remark:}
	Since it was geometrically more natural, we took a limit in
	\eqref{equation:moments} of the $H$-moments of the class group as $K \in
	\mathcal{MH}_{n,q}$, where $K$ has fixed log discriminant degree $n$. However,
	from this we can easily also deduce
	\begin{align}
		\label{equation:moments-sum}
		\lim_{\substack{n \to \infty \\ n \equiv i \bmod 2 }}
		\frac{\sum_{K \in \cup_{j \leq n, j \equiv n \bmod 2} \mathcal{MH}_{j,q}}
\left|\surj(\cl(\mathscr O_K), H) \right|} {\sum_{K \in \cup_{j \leq n, j \equiv n
	\bmod 2}\mathcal{MH}_{j,q}} 1} =
	\begin{cases}
		1 & \text{ if } i =1 \\
		\frac{1}{|H|} & \text{ if } i = 0
	\end{cases}
	\end{align}
	by summing the terms in \eqref{equation:moments} over all
	integers up to $n$ with the same parity as $n$.
\end{remark}
\begin{remark}
\label{remark:non-monic}
One may wonder whether the analog of \autoref{theorem:moments} would hold if one
averaged over all smooth hyperelliptic curves instead of just monic ones.
We explain in \autoref{remark:odd-degree-independence} why the resulting
averages will be the same in the case $n$ is odd if one takes non-monic smooth hyperelliptic curves ramified
over $\infty$ in place
of monic ones ramified over $\infty$.
In the case that $n$ is even, the statistics of non-monic smooth hyperelliptic curves
seem somewhat more subtle.
We believe it would be quite interesting and doable to work this case out in detail.
\end{remark}

\subsection{Results toward the Cohen--Lenstra heuristics with roots of unity}
We note that \autoref{theorem:moments} assumes $\gcd(q-1,|H|) = 1$, so that
there are no roots of unity in the base field of order dividing $|H|$. 
There has been a significant amount of work toward stating and proving the correct version
of the Cohen--Lenstra heuristics in the presence of such roots of unity, see
\cite{
malle2008cohen,
Malle:distribution,
adamM:a-class-group-heuristic,
garton:random-matrices-the-cohen-lenstra-heuristics-and-roots-of-unity,
lipnowskiST:cohen-lenstra-roots-of-unity,sawinW:conjectures-for-distributions-containing-roots-of-unity}. To this
end, we prove the following generalization, which removes the condition that
$\gcd(q-1,|H|) = 1$.
For $G$ a finite abelian group and $h \in \mathbb Z_{>0}$, we use $G[h]$ to denote
the $h$-torsion subgroup of $G$.

\begin{theorem}
		\label{theorem:moments-roots-of-unity}
	Suppose $H$ is a finite abelian group of odd order. 
	Fix $q$ an odd prime power with 
	$\gcd( |H|, q) = 1$ 
	and let 
	$h := \gcd(|H|, q - 1)$.
	There is an integer $C$ depending only on $H$ so that,
	if $q > C$ and $i \in \{0,1\}$,
	\begin{align}
		\label{equation:moments-with-roots-of-unity}
		\lim_{\substack{n \to \infty \\ n \equiv i \bmod 2 }}
		\frac{\sum_{K \in \mathcal{MH}_{n,q}}
		\left|\surj(\cl(\mathscr O_K), H)\right|} {\sum_{K \in \mathcal{MH}_{n,q}} 1} =
	\begin{cases}
		|\wedge^2 H[h]| & \text{ if } i =1 \\
		\frac{|\wedge^2 H[h]|}{|H|} & \text{ if } i = 0.
	\end{cases}
	\end{align}
\end{theorem}

\autoref{theorem:moments-roots-of-unity} is an immediate consequence of
\autoref{theorem:moments-with-error}, 
which proves a refined statement about the rate of convergence of the limit.

\subsection{Results toward the non-abelian Cohen--Lenstra heuristics}

Having focused on the abelian Cohen--Lenstra heuristics above, 
in recent years, there has been significant progress toward generalizing these
heuristics to the non-abelian case. This involves replacing the class group (the
maximal {\em abelian} unramified extension) by the maximal unramified extension.

For quadratic extensions, these conjectures for pro-$p$ groups were formulated in
\cite{bostonBH:heuristics-imaginary} and \cite{bostonBH:heuristics-real},
and function field versions for pro-odd groups were proven in a suitable large $q$ limit in
\cite{bostonW:non-abelian-cohen-lenstra-heuristics} when there are no additional
roots of unity in the base field.
This conjecture was further generalized to certain even order groups and the
case when there are roots of unity in the
base field (with an additional data of a lifting invariant to keep track of the the fact
that the moments may no longer equal $1$) in
\cite{wood:nonabelian-cohen-lenstra-moments}.
We will next prove many cases of \cite[Conjecture
5.1]{wood:nonabelian-cohen-lenstra-moments},
the nonabelian Cohen--Lenstra heuristics when there are roots of unity in the
base field, keeping track of the additional data of a lifting invariant.

We note that the above statements were all about quadratic extensions, and they
have since been generalized to conjectures about the maximal unramified
extension of a $\Gamma$ extension, for $\Gamma$ an arbitrary finite group.
Namely, this generalization was made in
\cite{liuWZB:a-predicted-distribution}
for the case that there are few roots of unity in the base field and it was made
in
\cite{liu:non-abelian-cohen-lenstra-in-the-presence-of-roots}
when there are additional roots of unity in the base field, where the author
also keeps track of the additional data of a lifting invariant.

To state our results proving many 
cases of \cite[Conjecture
5.1]{wood:nonabelian-cohen-lenstra-moments},
we will introduce some notation.
These cases we prove are also cases of
\cite[Conjecture 1.2]{liu:non-abelian-cohen-lenstra-in-the-presence-of-roots}.
Because the relevant notation in 
\cite{liu:non-abelian-cohen-lenstra-in-the-presence-of-roots}
appears to be a bit simpler for the cases we will prove, we follow the notation
there.
\begin{notation}
	\label{notation:non-abelian}
	Let $K$ be a quadratic extension of $\mathbb F_q(t)$, let
$K^\sharp_\emptyset/K$ denote the maximal unramified extension of $K$ that is
split completely over $\infty$, in the sense of
\autoref{definition:split-completely},
and has degree prime to $2q$.
Define
$G^\sharp_\emptyset(K) := \gal(K^\sharp_\emptyset/K)$.
Then, $G^\sharp_\emptyset(K)$ has an action of $\mathbb Z/2 \mathbb Z$ coming
from the $\mathbb Z/2 \mathbb Z = \gal(K/\mathbb F_q(t))$ action on $K$.
This action is well-defined using the Schur--Zassenhaus lemma; for further
explanation, see the paragraph prior to \cite[Definition
2.1]{liuWZB:a-predicted-distribution}.
For $H$ a group with a $\mathbb Z/2 \mathbb Z$ action,
we use $\surj_{\mathbb Z/2 \mathbb Z}(G^\sharp_\emptyset(K), H)$ to denote the
set of $\mathbb Z/2 \mathbb Z$ equivariant surjections $G^\sharp_\emptyset(K)
\to H$.
We use
$\hat{\mathbb Z}(1)_{(2q)'}$ to denote the pro-prime to $2q$ completion of
for $\widehat{\mathbb Z}(1) := \lim_n \mu_n(\overline{\mathbb F}_q)$.
For $L/K$ an extension of fields, there is a certain map
$\omega_{L/K} : \hat{\mathbb Z}(1)_{(2q)'} \to H_2(\gal(L/\mathbb F_q(t)),
\mathbb Z)_{(2q)'}$ defined in \cite[Definition
2.13]{liu:non-abelian-cohen-lenstra-in-the-presence-of-roots}.
For $\pi \in \surj_{\mathbb Z/2 \mathbb Z}(G^\sharp_\emptyset(K), H)$, we use
$\pi_* : H_2( \gal(K^\sharp_\emptyset/\mathbb F_q(t)), \mathbb Z)_{(2q)'} \to
H_2(H \rtimes \mathbb Z/2 \mathbb Z, \mathbb Z)_{(2q)'}$ to denote the corresponding map
induced by $\pi$.
We say that $H$ is {\em admissible} for the action of $\mathbb Z/2 \mathbb Z$
if $H$ has odd order and is generated 
by elements of the form $h^{-1} \cdot \iota(h)$ for $h \in H$ and
$\iota$ the nontrivial element of $\mathbb Z/2 \mathbb Z$.
\end{notation}

\begin{theorem}
		\label{theorem:non-abelian-moments-roots-of-unity}
	Using notation from \autoref{notation:non-abelian},
	suppose $H$ is an admissible group with a $\mathbb Z/2 \mathbb Z$ action.
	Fix $q$ an odd prime power with 
	$\gcd( |H|, q) = 1$.
Let $\delta: \hat{\mathbb Z}(1)_{(2q)'} \to H_2(H \rtimes \mathbb Z/2 \mathbb Z, \mathbb
Z)_{(2q)'}$ be a group homomorphism with $\on{ord}(\im \delta) \mid q-1.$
There is an integer $C$, depending only on $H$ so that,
if $q > C$,
\begin{equation}
	\begin{aligned}
		\label{equation:non-abelian-moments-with-roots-of-unity}
		\lim_{\substack{n \to \infty \\ n \equiv 0 \bmod 2 }}
		\frac{\sum_{K \in \mathcal{MH}_{n,q}}
		\left| \{ \pi \in \surj_{\mathbb Z/2 \mathbb
			Z}\left(G^\sharp_\emptyset(K), H \right): \pi_* \circ
	\omega_{K^\sharp/K} = \delta \}\right|} {\sum_{K \in \mathcal{MH}_{n,q}} 1} =
	\frac{1}{[H:H^{\mathbb Z/2 \mathbb Z}]},
	\end{aligned}
\end{equation}
where $[H: H^{\mathbb Z/2 \mathbb Z}]$ denotes the index of the $\mathbb Z/2
\mathbb Z$ invariants of $H$ in $H$.
\end{theorem}
\autoref{theorem:non-abelian-moments-roots-of-unity} is an immediate consequence of
\autoref{theorem:non-abelian-moments-with-error}, 
which proves a refined statement about the rate of convergence of the limit.

\subsubsection{}
\label{subsubsection:homology-to-cohen-lenstra}
In order to prove \autoref{theorem:moments},
\autoref{theorem:moments-roots-of-unity}, and
\autoref{theorem:non-abelian-moments-roots-of-unity},
we adapt the general strategy outlined in
\cite{ellenbergVWhomologicalII}.\footnote{Before Ellenberg, Venkatesh, and Westerland published their preprint
\cite{ellenbergVWhomologicalII},
Oscar Randal-Williams
discovered a serious error in 
\cite{ellenbergVWhomologicalII},
which the authors were unable to address. See
\cite{ellenberg:homological-stability-withdrawn} for further details. The main
result of this paper can be viewed as a resolution of that error.
}
Namely, one constructs certain Hurwitz spaces parameterizing covers of $\mathbb
P^1$ branched at $n$ points over $\mathbb A^1$. Via the
Grothendieck--Lefschetz trace formula,
computing the above moments
amounts to obtaining a precise count of the number of $\mathbb F_q$ points of
these Hurwitz spaces.
It was shown in \cite{EllenbergVW:cohenLenstra} that the homology groups of
these spaces stabilize in a linear range. The authors in loc. cit. used this stability
to prove a weaker version of our results, which included a large $q$ limit.
In order to compute the desired moments, it suffices to compute the stable
value of these homology groups.
Therefore, after the reductions mentioned above, the key new
result of this paper is to compute the stable homology groups of certain
Hurwitz spaces, which we discuss next.

\subsection{Stable homology of Hurwitz spaces}
\label{subsection:stable-homology-intro}

	There are a number of situations in algebraic geometry where one may be
	interested in determining the stable homology or cohomology groups of a sequence of
	spaces. One natural example is the sequence of spaces $\{\mathscr
	M_g\}_{g \geq 2}$. In
	\cite{Harer1985}, Harer proved that the $i$th cohomology of $\mathscr M_g$
	stabilizes as $g$ grows, see also 
	\cite{wahl:homological-stability-for-the-mapping-class-groups} and
	\cite{Wahl:stability-handbook}.
	Later, Madsen and Weiss
	\cite{madsen-weiss:the-stable-moduli-space-of-riemann-surfaces}
	computed the stable values of these cohomology groups.

	As a companion to the moduli spaces of curves, Hurwitz stacks also form
	prominent objects of study in algebraic geometry. 
	There are a number of equivalent viewpoints on how to think about
	Hurwitz stacks.
	If one fixes a finite group $G$, a conjugacy class
	$c$, and an integer $n$, the associated Hurwitz stack parameterizes
	$G$ covers of $\mathbb A^1$ branched at a degree $n$ divisor with
	inertia at each point of the divisor lying
	in $c$. 
	
	It is natural to ask whether the $i$th homology of these Hurwitz
	stacks also stabilizes as $n$ grows.
	Ellenberg, Venkatesh, and Westerland \cite{EllenbergVW:cohenLenstra}
	showed that this is indeed the case for certain {\em non-splitting}
	$(G,c)$, 
	namely those $(G,c)$ so that $c \cap G'$ consists of at most one conjugacy class for every subgroup $G' \subset G$.
	An important example of non-splitting $(G,c)$ is the following, which is relevant for the
	Cohen--Lenstra heuristics:
	$G = H \rtimes \mathbb Z/2 \mathbb Z$, for
	$H$ an odd order abelian group, and $c$ the conjugacy class of order
	$2$ elements.
	However, the value of these stable homology groups remained open.
	In this paper, we will compute this stable value of the homology of
	Hurwitz spaces associated to non-splitting $(G,c)$.

	We will consider a slight variant of the above notion of
	Hurwitz stack, which we refer to as the pointed Hurwitz space $\phurc G n c,$
	which we define precisely in \autoref{definition:pointed-hurwitz-space}
	and \autoref{notation:complex-hurwitz}.
	This variant is a
	finite \'etale cover of the Hurwitz stack described above, and involves marking
	a point over $\infty$. We briefly mention four equivalent descriptions of this
	space. The equivalence between the descriptions given below in
	\autoref{subsubsection:hurwitz-disc},
	\autoref{subsubsection:hurwitz-quotient}, and
	\autoref{subsubsection:hurwitz-ag} is shown in
	\cite[\S2.3]{EllenbergVW:cohenLenstra}, while the equivalence between
	these and the description in
	\autoref{subsubsection:hurwitz-e2} is given in \cite[Remark
	6.4]{randal-williams:homology-of-hurwitz-spaces}.

\subsubsection{In topology as covers of the disc}
\label{subsubsection:hurwitz-disc}
	One description of our pointed Hurwitz spaces, $\phurc G n c,$ is that they
parameterize pointed branched covers of the disc, branched at $n$ points, where the
inertia type of each branch point lies in a conjugacy class $c \subset G$. 
\subsubsection{In homotopy theory as an orbit space}
\label{subsubsection:hurwitz-quotient}
The pointed Hurwitz space $\Hur^{G,c}_n$ can also be described as the homotopy quotient 
$(c^n)_{hB_n}$, where $B_n$ is the braid
group on $n$ strands. 
\subsubsection{In homotopy theory as a free $\EE_2$ algebra}
\label{subsubsection:hurwitz-e2}
Conjugation by $G$ gives a $G$-action on $\Hur_n^{G,c}$, and
$\coprod_{n\geq0}\Hur_n^{G,c}$ can also be described as the free $\EE_2$-algebra
generated by $c$ in the category of $G$-crossed spaces, see \cite[Remark 6.4]{randal-williams:homology-of-hurwitz-spaces}.
\subsubsection{In algebraic geometry as a moduli space}
\label{subsubsection:hurwitz-ag}
Algebro-geometrically, we can also think of complex valued points of $\phurc G n
c$ as
parameterizing $G$-covers $X \to \mathbb P^1_{\mathbb C}$, branched at
a degree $n$ divisor in $\mathbb A^1_{\mathbb C}$ with inertia at these branch
points lying in $c \subset G$ and with a marked point over $\infty \in \mathbb
P^1_{\mathbb C}$.\footnote{When the monodromy at $\infty$ is nontrivial of order
	$r$, we need
to replace $\mathbb P^1_{\mathbb C}$ with a root stack of order $r$ at
$\infty$ to make this technically correct.}
See \autoref{definition:pointed-hurwitz-space} for a formal algebro-geometric
definition of these pointed Hurwitz spaces.
We call them pointed because they involve making a choice of marked point of the
cover $X$ over
$\infty$.

\begin{notation}
	\label{notation:configuration-space}
	Let $\conf_{n}$ denote the configuration space of $n$
unordered, distinct points in $\mathbb A^1_{\mathbb C}$.
There is a natural map $\phurc G n c \to \conf_{n}$ which
sends a branched cover $X \to \mathbb P^1$ to the intersection of its branch
locus with $\mathbb A^1$.
\end{notation}

As setup for our main result computing the stable homology of Hurwitz spaces, following
\cite[Definition 3.1]{EllenbergVW:cohenLenstra},
we recall the notion of a non-splitting conjugacy class.
\begin{definition}
	\label{definition:nonsplitting}
	Let $G$ be a finite group and $c \subset G$ be a conjugacy
	class. We say $(G, c)$ is {\em non-splitting} if $c$ generates
$G$ and, for every subgroup $G' \subset G$, $c \cap G$ is either empty or
consists of a single conjugacy class.
\end{definition}

That is, $(G, c)$ is non-splitting if $c$ does not split into multiple conjugacy classes upon intersecting
with subgroups.
See \autoref{subsection:nonsplitting-examples} for some examples relevant to the
Cohen--Lenstra heuristics.
Our main result on the stable homology of Hurwitz spaces is the following.

\begin{theorem}
	\label{theorem:stable-homology-intro}
	Let $G$ be a finite group, $c \subset G$ a conjugacy class, and suppose
	$(G,c)$ is non-splitting.
	There are constants $I$ and $J$ depending only on $G$ so that for 
	any $i \geq 0$ and
	$n > iI + J$ and any connected component $Z \subset \phur G n c {\mathbb
	C}$, the map $H_i(Z; \mathbb Q) \to H_i(\conf_{n}; \mathbb Q)$
	is an isomorphism.
\end{theorem}
We prove
\autoref{theorem:stable-homology-intro} in
\autoref{subsubsection:proof-stable-homology}.

\begin{remark}
	\label{remark:}
	\autoref{theorem:stable-homology-intro} is known to hold for certain special
	components $Z$ by work of Bianchi and Miller
	\cite[Corollary C']{bianchiM:polynomial-stability} and the text
	following it.
	Via personal communication, we learned this argument was also known to
	Ellenberg--Venkatesh--Westerland.
	In particular, this applies when $G$ is the dihedral group of order
	$2n$, for $n$ odd, to components $Z$ whose boundary monodromy
	(see \autoref{definition:boundary-monodromy}) is a generator of the
	index $2$ abelian subgroup of $G$.
	However, it does not apply, for example, when $G$ is the dihedral group
	of order $2n$ and the boundary monodromy of $Z$ is trivial.
	See also \cite[\S6]{tietz:thesis} for additional computations of the
	homology of Hurwitz spaces in other special cases similar to the above, via essentially
	the same strategy.

	Additionally, we mention that the results of 
	\cite[Corollary C']{bianchiM:polynomial-stability} do not require 
	a non-splitting hypothesis on $(G,c)$. As a trade off, without the
	non-splitting hypothesis, they obtain bounds on the stability range
	which are much worse than linear in the homological degree;
	see also \cite[\S1.2]{bianchiM:polynomial-stability}.
\end{remark}

\begin{remark}
	\label{remark:orw}
	The subtlety in understanding homological stability for $\Hur_{n}^{G,c}$ (which includes points corresponding to disconnected
	covers) is that it does not stabilize with respect to the 
	operator $V\in H_0(\phur G {\ord(c) \cdot |c|} c {\mathbb
		C};
	\mathbb Q)$ corresponding to $\prod_{g \in c} [g]^{\ord(g)}$. Rather,
	\cite{EllenbergVW:cohenLenstra} show that the operator $\sum_{g \in c}
	[g]^{\ord(g)}$ stabilizes the homology, which is not an operation that exists at the level of spaces. 
	Oscar Randal-Williams mentions in his Bourbaki survey article that
	``homological stability of the spaces $\mathrm{CHur}_{n}^{G,c}$ with respect to the maps
$V$ should serve as a guiding problem for mathematicians working in this
subject.''
\cite[p. 24]{randal-williams:homology-of-hurwitz-spaces}.

On the subspace $\mathrm{CHur}_{n}^{G,c}\subset \Hur_{n, \mathbb
	C}^{G,c}$ corresponding to $G$-covers that are connected, it follows from \autoref{theorem:stable-homology-intro}
	that for $(G,c)$ non-splitting,
	$V$ does stabilize the homology. Moreover, we show the stable value on each component agrees with that of $\Conf_n$.
\end{remark}

\subsection{The stable homology for $S_3$ and degree $3$ covers}
Before proceeding further, we pause to highlight a seemingly elementary case of
\autoref{theorem:stable-homology-intro} which, surprisingly, was previously unknown.
Consider the special case that $G = S_3$, and the conjugacy class $c \subset G$
consists of transpositions. The pointed Hurwitz space $\cphurc G n c $ 
associated to this parameterizes
connected degree $3$ simply branched covers of $\mathbb P^1$,
branched over a degree $n$ divisor in $\mathbb A^1$,
see \autoref{notation:connected-hurwitz-spaces} for a more formal definition.
Simply branched triple covers have been studied extensively in algebraic
geometry, and it is natural to ask about the homology of the space of such
covers as the number of branch points grows.

\begin{corollary}
	\label{corollary:s3-homology}
	There are constants $I$ and $J$ so that for any $i \geq 0$ and $n > I i + J$
	\begin{align*}
		h^i(\churc {S_3} n { \{(12),(13),(23)\}} ; \mathbb Q) =
		\begin{cases}
			1 & \text{ if } i \in \{0,1\} \\
			0 & \text{ otherwise. }  \\
		\end{cases}
	\end{align*}
\end{corollary}

\begin{remark}
	\label{remark:}
	If one considers the space of all smooth trigonal curves, instead
of just simply branched ones, the stable cohomology of such spaces has been
computed in \cite{zheng:stable-cohomology-trigonal} using Vassiliev's method.
See \cite{tomassi:rational-cohomology-of-the-moduli-space},
	\cite{gorinov:real-cohomology}, and
\cite{vassiliev:how-to-calculuate} for references on Vassiliev's method.
However, via personal communication, we learned that it is
unclear whether a similar method can be used to compute the stable cohomology of
the space of simply branched trigonal curves.
\end{remark}
\begin{remark}
	\label{remark:}
	There is a cycle class map from the Chow ring of these Hurwitz spaces to
	the cohomology of these Hurwitz spaces. The same explicit
	parameterizations which allowed Davenport and Heilbronn
	\cite{davenportH:density-discriminants-cubica} to compute the
	number of cubic fields allowed Patel and Vakil to compute the stable
	Chow rings of these Hurwitz spaces
	\cite[Theorem C]{patel2015chow}, see also
	\cite{canningL:chow-rings-of-low-degree-hurwitz-spaces}.
\end{remark}

\subsection{Conjectures on stable homology of Hurwitz spaces}
Before proceeding to describe the new ideas in the proof of
\autoref{theorem:stable-homology-intro},
we give a conjecture for what we think the stable homology of Hurwitz spaces
should look like in general.

\begin{notation}
	\label{notation:conjecture-setup}
	To set up notation, let $G$ be a finite group and $c \subset G$ denote a union of conjugacy
	classes. That is, $c = c_1 \cup \cdots c_k$, where each $c_i \subset G$ is a conjugacy class.
	Let $\cphurc G {n_1, \ldots, n_k} {c_1,
\ldots, c_k}$ denote the Hurwitz space parameterizing connected $G$ covers of
$\mathbb A^1_{\mathbb C}$ with a trivialization at $\infty$ and $n_i$ branch
points over $\mathbb A^1_{\mathbb C}$ with inertia in the conjugacy class of any element
of $c_i$. 
We also let $\conf_{n_1, \ldots, n_k}$ denote the multi-colored
configuration space in $\mathbb A^1_{\mathbb C}$ parameterizing $n_i$ points of
color $i$ so that all $n_1 + \cdots + n_k$ points in $\mathbb A^1_{\mathbb C}$ are distinct.
That is, 
$\conf_{n_1, \ldots, n_k}$ is the quotient of ordered configuration space of
$n_1 + \cdots + n_k$ points by the action of $S_{n_1} \times \cdots \times
S_{n_k}$, where $S_{n_i}$ permutes points $n_1 + \cdots + n_{i-1} + 1, \ldots,
n_1 + \cdots + n_i$.
\end{notation}
\begin{remark}
\label{remark:union-of-conjugacy-classes}
One might require a weaker notion that $c$ is only conjugation invariant, in the
sense that for any 
$x, y \in c$, we also have $x^{-1}yx \in c$.
However, if $c$ generates $G$, such a subset must
actually be a union of conjugacy classes.
In general, without this assumption,
such a conjugation invariant subset can be viewed as a union of conjugacy
classes in the subgroup it generates.
\end{remark}

\begin{remark}
	\label{remark:}
	Because there will not be any {\em connected} $G$ covers with inertia in $c$
	unless $c$ generates
	$G$, using notation from \autoref{notation:conjecture-setup},
$\cphurc G {n_1, \ldots, n_k} {c_1,
\ldots, c_k} $ will be the empty set unless the conjugacy classes $c_1, \ldots, c_k$
jointly generate $G$.
\end{remark}

Our main conjecture on the stable homology of Hurwitz spaces is the following,
which predicts that the only stable homology is the ``obvious'' stable homology.

\begin{conjecture}
	\label{conjecture:homology}
	Fix an integer $i$. With notation as in
\autoref{notation:conjecture-setup}, suppose
	all $n_1, \ldots, n_k$ are sufficiently large (with how large they must
	be depending on $i$) and let $Z
	\subset \cphurc G {n_1, \ldots, n_k} {c_1,\ldots, c_k} $ denote
	a connected component. Then, the natural map sending a cover to its
	branch locus $Z \to \conf_{n_1, \ldots,
	n_k}$ induces an isomorphism $
H_i(Z; \mathbb Q) \to
	H_i(\conf_{n_1, \ldots,
	n_k}; \mathbb Q).$
\end{conjecture}

The special case of the above conjecture where $G = S_d$, $k = 1$, and $c_1$ is the 
conjugacy class of transpositions is a conjecture of
Ellenberg-Venkatesh-Westerland \cite[Conjecture
1.5]{EllenbergVW:cohenLenstra}.
Even in this special case, our conjecture is slightly stronger because
Ellenberg-Venkatesh-Westerland do not conjecture that the stability
occurs in a linear range.
The above conjecture was suggested to us by Melanie Wood, who in turn pointed
out to us that a statement very closely related to this conjecture was formulated in
\cite[Corollary 5.8]{ellenbergVWhomologicalII}.

\begin{remark}
	\label{remark:}
	Admittedly, the evidence for \autoref{conjecture:homology} in the literature is somewhat limited.
The case that $G$ is abelian is fairly trivial, since in that case, the relevant
Hurwitz space of
$G$-covers is precisely identified with the corresponding configuration space:
indeed, in the case $G$ is abelian, all elements correspond to distinct conjugacy
classes, so one can read off the total monodromy of the cover from the conjugacy
classes of the inertia and the locations of the branch points, and therefore recover the cover.
To our knowledge, the only nonabelian cases of this conjecture which are known
are those in the present paper, namely those verified in
\autoref{theorem:stable-homology-intro}.
However, this conjecture is largely motivated by Malle's conjecture on counting
$G$ extensions of number fields.
See \cite{ellenbergV:statistics-of-number-fields-and-function-fields} for a
detailed description of this connection.
\end{remark}


\begin{remark}
	\label{remark:}
	Note that we make \autoref{conjecture:homology} for arbitrary $G,c_1,
	\ldots, c_k$
	even though it is not
known whether the homology of Hurwitz spaces stabilize at all, except in the
case that $k = 1$ and $c_1$
is non-splitting.
\end{remark}

Perhaps more ambitiously,
We ask whether the analogous statement could possibly also hold for 
Chow groups of Hurwitz spaces.
\begin{conjecture}
	\label{conjecture:chow-stabilize}
	Fix an integer $i$. With notation as in
	\autoref{notation:conjecture-setup}, suppose
	all $n_1, \ldots, n_k$ are sufficiently large (with how large they must
	be depending on $i$).
	Then, the $i$th rational Chow groups of 
	$\cphurc G {n_1, \ldots, n_k} {c_1,\ldots, c_k} $ stabilize in the sense
	that they are independent of
	$n_1, \ldots, n_k$.
\end{conjecture}
Even more ambitiously, we conjecture that the Chow rings stabilize to those of
configuration space.
\begin{conjecture}
\label{conjecture:chow-stable-value}
In the setting of \autoref{conjecture:chow-stabilize},
let $Z \subset \cphurc G {n_1, \ldots, n_k} {c_1,\ldots, c_k}$ denote
	a connected component.
	Then, the natural map $Z \to \conf_{n_1, \ldots,
	n_k}$ sending a cover to its branch locus
	induces an isomorphism on $i$th rational Chow groups.
\end{conjecture}

\begin{remark}
	\label{remark:}
	As pointed out to us by Dan Petersen,
	the rational Chow group of 
	$\conf_{n_1, \ldots, n_k}$
	is $0$ in positive degrees and $\mathbb Q$ in degree $0$,
	so if \autoref{conjecture:chow-stable-value} were true,
	this would give an especially simple description of the stable Chow
	groups of Hurwitz spaces.
\end{remark}

\subsection{Idea of the Proof}
\label{subsection:proof-outline}

We now describe the new ideas going into our proof of \autoref{theorem:moments},
and more generally \autoref{theorem:moments-roots-of-unity},
the computation of the moments
predicted by the Cohen--Lenstra conjectures over function fields.
A standard reduction also outlined in
\autoref{subsubsection:homology-to-cohen-lenstra}, which we carry out in
\autoref{section:cohen-lenstra},
reduces us to proving
\autoref{theorem:stable-homology-intro}.

We next explain the idea of the proof of \autoref{theorem:stable-homology-intro},
which
computes the stable value of the homology for certain Hurwitz spaces.
We fix a finite group $G$ and a conjugacy class $c \subset G$ so that $(G,c)$ non-splitting.
We use $\Hur^{G,c}$ to denote the union over $n$ of the Hurwitz spaces
$\Hur^{G,c}_{n}$ and $\Conf$ to denote the union over $n$ of $\Conf_n$. We use $\Conf^{G,c}$ to denote $\Conf\times_{\pi_0\Conf}\pi_0\Hur^{G,c}$; 
in other words, it has the same
set of components as $\Hur^{G,c}$, but we replace each component of
$\Hur^{G,c}_{n}$ with $\Conf_n$.
There is a projection $v: \Hur^{G,c} \to \Conf^{G,c}$, which can be thought of
as sending a cover to
its branch locus, together with the data of which component the cover came from.
If we use $U$ to denote the stabilization map $\Hur^{G,c}$ (the map under which
	the homology of Hurwitz space was shown to stabilize in \cite[Theorem 6.1]{EllenbergVW:cohenLenstra})
we can rephrase our goal as showing $v[U^{-1}] : H_*(\Hur^{G,c};\QQ)[U^{-1}] \to
H_*(\Conf^{G,c};\QQ)[U^{-1}]$
is an equivalence. 

Considering $\Hur^{G,c}$ and $\Conf^{G,c}$ as associative monoids (or
$\EE_1$-algebras), our goal is equivalent to showing that the map $v[U^{-1}]:C_*(\Hur^{G,c};\QQ)[U^{-1}] \to
C_*(\Conf^{G,c};\QQ)[U^{-1}]$ is an equivalence of associative ring spectra\footnote{Since we are working over $\QQ$, these associative ring spectra can be viewed as the dga computing the homology with its multiplication given via the Pontryagin product.}, where $C_*(X;\QQ)$ denotes the chains of a space $X$ with $\QQ$-coefficients, viewed as either a spectrum, or in the derived category of $\QQ$-vector spaces.

We can view the center $Z(R)$ of $R:=H_0(\Hur^{G,c};\QQ)$ as a commutative ring
acting on the source and target of the map $v$, and we study the map $v[U^{-1}]$
via decomposing it geometrically along open subsets covering $\Spec(Z(R))$. It
then suffices to show that locally on each of these subsets the map $v[U^{-1}]$
becomes an equivalence. Algebraically, the way we implement restricting to an
open subset of $\Spec(Z(R))$ is by using localizations in the setting of ring
spectra (such as in \cite[Section 7.2.3]{HA}), and completions along a collection of elements,
which we show in our case happens to also be a localization. Combining this
strategy with a group theoretic argument shows that $v[U^{-1}]$ is an
equivalence if $v$ is an equivalence after inverting 
$\langle c' \rangle$, and then completing at $c - c'$,
for every subset $c' \subset c$ closed under conjugation
(i.e., $x, y \in c' \implies x^{-1}yx \in c'$).

When $c' = c$, this localized map was already known to be an equivalence. This is because inverting all elements of $c$ gives the group completion of $\Hur^{G,c}$, whose homology can be understood via its classifying space $B\Hur^{G,c}$. The homology of the classifying space of is the rack homology of the rack $c$, which was
computed in \cite{etingof2003rack}, see \cite[Corollary 5.4]{randal-williams:homology-of-hurwitz-spaces}.

When $c'$ is a proper subset of $c$, the key notion helping us show the map
becomes an equivalence is the notion of a homological epimorphism. A map $R \to
S$ of $\EE_1$-rings is called a homological epimorphism if the multiplication
map $S\otimes_RS \to S$ is an equivalence in the $\infty$-category of spectra. A
key result we use is a rigidity property of homological epimorphisms, which is
that a homological epimorphism of connective associative ring spectra that is an
isomorphism on $\pi_0$ is an equivalence. 
Applying the above results to the subgroup $\langle c' \rangle \subset G$
generated by $c'$, we can reduce to showing that a certain `restriction map' between the localized homology of Hurwitz spaces is a homological epimorphism. This fact is something that can be verified at the level of pointed spaces, by showing that a certain map between two sided bar constructions is a homotopy equivalence. We use a convenient topological model for these two sided bar constructions, the details of which are carried out in \autoref{appendix:topological-model}. Using this topological model, we write down an explicit homotopy to prove this homotopy equivalence;
see \autoref{figure:nullhomotopy} for a pictorial depiction.

\begin{remark}
	\label{remark:otherproof}
	It is also possible to prove \autoref{theorem:stable-homology-intro}
	without the
machinery of higher algebra. We originally came up with a more elementary argument but found that by translating things to higher
algebra, the argument became significantly cleaner and more conceptual. 
A version of our elementary argument will appear in
\cite{landesmanL:an-alternate-computation}.
\end{remark}

\subsection{Outline}
The sections of this paper are organized as follows.
First, we set up our notation for Hurwitz spaces in
\autoref{section:hurwitz-space-notation}.
In \autoref{section:algebra-background}, we introduce notation from higher
algebra relevant to our paper, and prove some facts about homological
epimorphisms, which will be crucial to our main results.
In \autoref{section:computing-stable-homology}, we prove our main topological result,
\autoref{theorem:stable-homology-intro},
computing the stable homology of non-splitting Hurwitz spaces. We in fact
compute the stable homology more generally of Hurwitz spaces associated to
non-splitting racks.
In \autoref{section:cohen-lenstra}, we prove our main result
toward the Cohen--Lenstra heuristics,
\autoref{theorem:moments-roots-of-unity},
by deducing it from our topological
results about the stable homology of Hurwitz spaces.
We conclude with \autoref{appendix:topological-model},
which uses various scanning arguments to construct convenient topological models for certain bar constructions on Hurwitz spaces.

\subsection{Acknowledgements}

We thank Jordan Ellenberg and Craig Westerland for sharing a number of
the ideas they tried while working on this problem.
The ideas in paper build heavily upon ideas in the retracted
preprint
\cite{ellenbergVWhomologicalII}, and we were able to draw much inspiration from
the ideas presented there.
We thank Melanie Wood for 
suggesting \autoref{conjecture:homology}, listening to a detailed explanation of
the argument, and for helpful comments on the paper.
We thank Andrea Bianchi for an extremely helpful suggestion which allowed us to
remove an extraneous condition we had on \autoref{theorem:mainhom}.
We thank Yuan Liu for explaining to us the correct version of the Cohen--Lenstra
heuristics in the function field case, and for help understanding the
non-abelian Cohen--Lenstra heuristics with roots of unity.
We also thank 
Samir Canning,
Kevin Chang, 
Elia Gorokhovsky,
Ben Knudsen, 
Jef Laga,
Jeremy Miller,
Sun Woo Park,
Anand Patel, 
Sam Payne,
Dan Petersen,
Eric Rains,
Oscar Randal-Williams,
Will Sawin,
Mark Shusterman,
Agata Smoktunowicz,
Nathalie Wahl,
and Angelina Zheng
for helpful conversations.
Landesman 
was supported by the National Science
Foundation 
under Award No.
DMS 2102955.
Levy was supported by the NSF Graduate Research
Fellowship under Grant No. 1745302, and by the Clay Research Fellowship.

\section{Notation for Hurwitz spaces}
\label{section:hurwitz-space-notation}

\subsection{The definition of Hurwitz spaces}
\label{subsection:hurwitz-definition}

We begin by carefully defining the Hurwitz spaces we will work with. We start
by defining the Hurwitz stack
$\hur G n c B$.
This will be a quotient by an action of $G$ of the pointed Hurwitz space
$\phur G n c B$, which we will define next, and this explains why there is a
quotient by $G$ in the notation for 
$\hur G n c B$.

\begin{definition}
	\label{definition:hurwitz-stacks}
	Let $G$ be a finite group, $c\subset G$ be a union of conjugacy classes.
	Let $B$ be a scheme with $|G|$ invertible on $B$.
	Assume either that $c$ is closed under invertible powering ($g \in c
		\implies g^t
	\in c$ for any $t$ prime to $|G|$,) or that $B$ is henselian with residue
	field $\spec \mathbb F_q$ and
	$c$ is closed under under $q$th powering ($g \in c \implies g^q \in c$).
	Let $\hur G n c B$ denote the {\em Hurwitz stack}
	whose $T$ points $\hur G n c B(T)$ is the groupoid
	\begin{align*}
		\left( D, i: D \to \mathbb P^1_T, X, h: X \to \mathbb P^1_T  \right)
	\end{align*}
		satisfying the following conditions:
		\begin{enumerate}
			\item $D$ is a finite \'etale cover of $T$ of degree $n$.
	        \item $i$ is a closed immersion $i: D \subset \mathbb A^1_T
			\subset \mathbb P^1_T$.
		\item $X$ is a smooth proper relative curve over $T,$ not
		necessarily having geometrically
		connected fibers.
		\item $h: X \to \mathbb P^1_T$ is a finite locally free Galois $G$-cover,
			(meaning that $G$ acts simply transitively on the
			geometric generic fiber of $h$,) which is
			\'etale away from $\infty_T \cup i(D) \subset \mathbb
			P^1_T$.
		\item The inertia of $X \to \mathbb P^1_T$ over any geometric
			point of $i(D)$ lies in $c \subset
			G$.
	\item The morphisms between two points $(D_i, i_i, X_i, h_i)$ for $i \in
		\{1,2\}$ are given by $(\phi_D, \psi_X)$ where $\phi_D: D_1
		\simeq D_2$ is an isomorphism so that $i_2 \circ \phi_D = i_1$
		and $\psi_X:X_1 \simeq X_2$ is an isomorphism such $h_2 \circ
		\psi_X = h_1$ and $\psi_X
		= g^{-1} \psi_X g$ for every $g \in G$.
	\end{enumerate}
\end{definition}

\begin{remark}
	\label{remark:}
	When $c = G - \id$, the definition of Hurwitz spaces we give here is a special case of the
definition in \cite[Definition
2.4.2]{ellenbergL:homological-stability-for-generalized-hurwitz-spaces}. 
See \cite[Remark
2.4.3]{ellenbergL:homological-stability-for-generalized-hurwitz-spaces} for an
explanation of why these Hurwitz stacks are indeed stacks. It essentially
follows from
\cite[\S1.3.2 and Appendix B]{abramovichCV:twisted-bundles}.

The above paragraph shows that
$\hur G n {G-\id} B$ is an algebraic stack, but does not yet justify the
existence of the stacks
$\hur G n c B$ under the suitable powering hypotheses, depending on $B$.
First, if $|G|$ is invertible on $B$, and $c$ is closed under
invertible powering,
then
$\hur G n c B$ is a union of components of
the Hurwitz space
$\hur G n {G-\id} B$
by
\cite[Corollary
12.2]{liuWZB:a-predicted-distribution}.
Second, we explain why
$\hur G n c B$ is a union of components of
the Hurwitz space
$\hur G n {G-\id} B$
when $B$ is a Henselian dvr with finite
residue field.
We can first reduce to the case that $B = \mathbb F_q$ using that $B$ is Henselian 
and 
$\hur G n {G-\id} B$ has a smooth proper compactification by
\cite[Corollary B.1.4]{ellenbergL:homological-stability-for-generalized-hurwitz-spaces}.
Since $c$ is closed under $q$th powering, 
the result then holds using by
\cite[Theorem 12.1(2)]{liuWZB:a-predicted-distribution}.
\end{remark}

One important idea in this paper is to work with pointed Hurwitz spaces, which deal with the case of
``real quadratic'' or ``ramified at $\infty$'' quadratic fields. 
Variants of these were also used in \cite{ellenbergL:homological-stability-for-generalized-hurwitz-spaces},
but prior work in the context of the Cohen--Lenstra heuristics appears to
primarily focus on the case over covers unramified at $\infty$.
These Hurwitz spaces parameterize covers of a stacky $\mathbb P^1$, with a root
stack at $\infty$ of order $2$.
We now define these pointed Hurwitz spaces.

\begin{definition}
	\label{definition:pointed-hurwitz-space}
	Fix an integer $w$ and define $\mathscr P^w$ to be the root stack of order $w$ along
	$\infty$ of $\mathbb P^1$. (See \cite[Definition
		2.2.4]{cadman:using-stacks-to-impose-tangency}, where this is
	notated as $\mathbb P^1_{(\infty, w)}$.) 
	Using \cite[Theorem 10.3.10(ii)]{olsson2016algebraic}, 
	the fiber of this root
	stack over $\infty$ is the stack quotient 
$[\left(\spec_B \mathscr O_B[x]/(x^r)\right)/ \mu_r]$ of the relative spectrum $\spec_B \mathscr
O_B[x]/(x^r)$ by $\mu_r$. Let $\widetilde{\infty} : B \to \mathscr P^w$ denote the section over $\sigma$ corresponding to map $B
	\to [\left(\spec_B \mathscr
	O_B[x]/(x^r) \right)/ \mu_r]$ given by the trivial $\mu_r$ torsor over
	$B$, $\mu_{r} \to
	B$, and the $\mu_r$ equivariant map $\mu_{r} \to B \to \spec_B \mathscr
	O_B[x]/(x^r)$.
	We use notation as in \autoref{definition:hurwitz-stacks},

	Define the {\em $w$-pointed Hurwitz space}, $\left(\phur G n c B
	\right)^w$, to be the algebraic space
	whose $T$ points are the set parameterizing data of the form
	\begin{align*}
		\left( D, h': X \to \mathscr
		P^w_T, t: T \to X
		\times_{h',\mathscr P^w_T, \widetilde{\infty}_T} T, i:
	D \to \mathbb P^1_T, X, h: X \to \mathbb P^1_T  \right),
	\end{align*}
	where $D, i,X,$ and $h$ satisfy the properties listed in
\autoref{definition:hurwitz-stacks}. We also assume the order of inertia of $h$
	along $\infty$ is $w$ and define
	$\widetilde{\infty}_T$ to be the base change of the section
	$\widetilde{\infty}$ defined above to $T$.
	We additionally impose the condition that
	$h'$ is a finite locally free $G$-cover, \'etale over
	$\widetilde{\infty}_T$, such that the composition of $h': X \to \mathscr
		P^w_T$ with the coarse
		space map $\mathscr P^w_T \to \mathbb P^1_T$ is $h$,
		and $t: T \to X \times_{h',\mathscr P^w_T,
		\widetilde{\infty}_T} T$
		is a section of $h'$ over $\widetilde{\infty}_T$.

		In general, we define the {\em pointed Hurwitz space} as $\phur G
		n c B := \coprod_{w \geq 1} \left(\phur G n c B\right)^w.$
\end{definition}
\begin{remark}
	\label{remark:}
	We note that 
	$\phur G n c B$, defined as an algebraic space, is in fact a scheme. 
	Indeed, $\phur G n c B$ is a finite
	\'etale cover of the configuration space of $n$ unordered points in $\mathbb A^1_B$, as
	can be verified in an analogous fashion to \cite[Proposition
	11.4]{liuWZB:a-predicted-distribution}.
\end{remark}

\begin{remark}
	\label{remark:section-independent}
	In the definition of pointed Hurwitz space,
	\autoref{definition:pointed-hurwitz-space}, we choose a particular section
	$\widetilde{\infty} :B \to \mathscr P^w$ over $\infty: B \to \mathscr
	P^1_B$. However, one can show
	that if we chose a different section $\widetilde{\infty}' :B \to \mathscr
	P^w$, the resulting pointed Hurwitz space would be isomorphic.
	The idea is that the sections $\widetilde{\infty}$ and
	$\widetilde{\infty}'$ factor respectively through maps $\alpha: B \to [B/
	\mu_w]$ and $\alpha' : B \to [B/\mu_w]$, and so it suffices to produce a
	map $\beta: [B/\mu_w] \to [B/\mu_w]$ so that $\alpha \circ \beta =
	\alpha'$, as this map will induce a map taking $\widetilde{\infty}$ to
	$\widetilde{\infty}'$, which then induces an isomorphism of these two
	Hurwitz spaces.
	The desired map $\beta$ is simply given as the composition of the
	structure map $[B/\mu_w] \to B$ with $\alpha'$.
\end{remark}

\begin{remark}
	\label{remark:odd-degree-independence}
	Note that monic smooth hyperelliptic curves $y^2 = f(x)$ with $f$ of odd degree correspond to the choice of
	marked section given in \autoref{definition:pointed-hurwitz-space} while
	non-monic smooth hyperelliptic curves $y^2 = f(x)$ with $f$ of with the same odd degree
	correspond to a different choice of marked section, see
	\autoref{lemma:hyperelliptics-split-over-infinity}.
	Hence, the two relevant
	Hurwitz spaces will be isomorphic by
	\autoref{remark:section-independent}. From this, one can obtain 
	the fact that the distributions of class groups of monic smooth hyperelliptic
	curves of odd degree are isomorphic to the distribution of class groups of non-monic
	smooth
	hyperelliptic curves of odd degree. This was mentioned in
	\autoref{remark:non-monic}.

	There is also a more elementary way to see this equivalence,  as we learned from 
\cite[Corollary
6.5]{bostonW:non-abelian-cohen-lenstra-heuristics}.
This can be viewed as a reformulation of the argument in the previous paragraph.
Namely, there is a bijection between monic smooth hyperelliptic curves and
non-monic smooth
hyperelliptic curves which additionally preserves their isomorphism type, and
hence their class groups. 
For any $\alpha \in \mathbb F_q -(\mathbb F_q)^2$
every non-monic smooth hyperelliptic curve can be written as
$y^2 = f(x)$, where $f(x)$ has leading coefficient $\alpha$.
Using this representation of such non-monic smooth hyperelliptic curves, the map is given by 
sending a quadratic extension corresponding to a monic
smooth
hyperelliptic curve
$K$ over $\mathbb F_q(t)$ to the quadratic extension corresponding to a
non-monic smooth hyperelliptic curve $K
\otimes_{\mathbb F_q(t)} \mathbb F_q(s)$ where $t \mapsto \alpha s$.
\end{remark}

\begin{notation}
	\label{notation:connected-hurwitz-spaces}
We use $\cphur G n c B$ to denote the open and closed
subscheme of $\phur G n c B$ parameterizing covers $X \to \mathbb P^1_T$ so that
the geometric fibers of $X$ over $T$ are connected.
We use $\chur G n c B$ to denote the open and closed
substack of $\hur G n c B$ parameterizing covers $X \to \mathbb P^1_T$ so that
the geometric fibers of $X$ over $T$ are connected.
\end{notation}

\begin{notation}
	\label{notation:complex-hurwitz}
	We will often work over the complex numbers, and hence it will be convenient to
assume $B = \spec \mathbb C$. Therefore, we define
$\hurc G n c := \hur G n c {\spec \mathbb C},
\phurc G n c := \phur G n c {\spec \mathbb C},
\churc G n c := \chur G n c {\spec \mathbb C},
\cphurc G n c := \cphur G n c {\spec \mathbb C}$.
When $R$ is a ring, we often use use
$\phur G n c R$ in place of $\phur G n c {\spec R},$ and use analogous notation for other variants of this Hurwitz space.
\end{notation}

\subsection{Notation for components of Hurwitz spaces}

We now introduce some notation for the Hurwitz spaces we will work with.
First, we describe the notion of boundary monodromy of a component.
If we think of the component as parameterizing
certain covers $X \to \mathbb P^1$, the boundary monodromy corresponds to the
inertia of the cover over $\infty$.

\begin{definition}
	\label{definition:boundary-monodromy}
	With notation as in \autoref{definition:pointed-hurwitz-space},
	there is a map from the set of components of $\phur G n c B$ to $G$
	defined as follows. We can represent connected components of $\phur G n
	c B$ as tuples $(g_1, \ldots, g_n)$ modulo the action of the braid group
	$B_n$, and we associate to this the element $g_1 \cdots g_n \in G$.
	For such a component mapping to $g \in G$, we say the {\em boundary
	monodromy}
	of this component is $g$.
	We define $\boundarycphur G n c B g$ denote the open and closed
	subscheme of $\cphur G n c B$ consisting of those components 
	whose boundary monodromy is $g$.
\end{definition}

\subsection{Covers split completely at infinity}
\begin{definition}
\label{definition:split-completely}
Let $X \to \mathbb P^1_{\mathbb F_q}$ be a Galois $G$ cover 
with branch divisor of degree $n$ and
inertia over
$\mathbb A^1_{\mathbb F_q}$ all lying in a
union of conjugacy classes $c \subset G$.
Hence, this cover corresponds to a point $\spec \mathbb F_q \to \hur G n c{\mathbb F_q}$.
We say this cover 
is 
{\em completely split over $\infty$} if it is in the image of the map $\phur G n
c {\mathbb F_q} (\mathbb F_q) \to \hur G n c{\mathbb F_q}(\mathbb F_q)$.
\end{definition}

Next, we wish to verify that monic smooth hyperelliptic curves precisely correspond to
smooth covers split over infinity. This is fairly straightforward in the case $n$ is
even. It is a bit more involved when $n$ is odd, but we are able to verify this
via an explicit blow-up computation.
For the next lemma, recall that a monic smooth hyperelliptic curve over $\mathbb F_q$ is by definition a hyperelliptic curve defined affine locally 
by an equation of the form $y^2 =f(x)$ with 
$f(x) =  x^n + a_{n-1}x^{n-1} + \cdots + a_1x + a_0$ for $a_i \in \mathbb F_q$.
Explicitly, we view the hyperelliptic curve as a cover of $\mathbb P^1_{\mathbb
F_q}$ with $x$ as above the coordinate on $\mathbb A^1_{\mathbb F_q}$, viewed as a
subscheme of $\mathbb P^1 = \mathbb P (\mathbb F_q[X,Y])$ with $x = X/Y$.
\begin{lemma}
\label{lemma:hyperelliptics-split-over-infinity}
Take $G = \mathbb Z/2 \mathbb Z$ and $c$ to be the nontrivial element of
$\mathbb Z/2 \mathbb Z$ in \autoref{definition:split-completely}.
The $\mathbb Z/2 \mathbb Z$ smooth covers of $\mathbb P^1_{\mathbb F_q}$ split completely over the
point $\infty \in \mathbb P^1_{\mathbb F_q}$
precisely correspond to 
monic smooth hyperelliptic curves $y^2 = f(x)$ over
$\mathbb F_q$.
\end{lemma}
\begin{proof}
First, we describe the case that $n$ is even.
Any hyperelliptic curve in characteristic not $2$ which is unramified over
$\infty$ can be uniquely written in the form
$f(x) =  \alpha x^n + a_{n-1}x^{n-1} + \cdots + a_1x + a_0$
where either $\alpha = 1$ or $\alpha \in \mathbb F_q^\times$ is a fixed
non-square.
Such a curve is split completely over infinity if and only if its
fiber over $\infty \in \mathbb P^1_{\mathbb F_q}(\mathbb F_q)$ consists of two
$\spec \mathbb F_q$ points.
The fiber over $\infty$ of this curve can be identified with the fiber over $0$
of the curve
$z^2 =  \alpha + a_{n-1}u + \cdots + a_1u^{n-1} + a_0u^n$, where we have set $u
= 1/x$ and then $z = y/x^{n/2}$.
This is the same as $\spec \mathbb F_q[y]/(y^2 - \alpha)$, which will be two
copies of $\mathbb F_q$ when $\alpha$ is a square and a single copy of $\mathbb
F_q^2$ when $\alpha$ is a non-square.
Therefore, the curves where $\alpha$ is a square (which are all isomorphic to
such curves with $\alpha = 1$) precisely correspond to those completely split over
$\infty$.

We next consider the case that $n$ is odd.
As above, we can still write our hyperelliptic curve in the form
$f(x) =  \alpha x^n + a_{n-1}x^{n-1} + \cdots + a_1x + a_0$
with $\alpha$ either $1$ or a fixed non-square.
Changing coordinates to $u = 1/x$, we can write the curve in
the form
$z^2 =  u(\alpha + a_{n-1}u + \cdots + a_1u^{n-1} + a_0u^n)$
where $u = 1/x$ and $z = y/x^{n+1/2}$.
Let $X$ denote the proper regular curve corresponding to the above equation.
We can identify the stack quotient $[X/(\mathbb Z/2 \mathbb Z)]$ Zariski locally
around the residual gerbe at $\infty$
with a Zariski open of the root stack $\mathscr P$, which is a
root stack of $\mathbb P^1$ over the divisor $\infty$, so $\mathscr P$ has a single stacky point of order
$2$.

Let $\widetilde{\infty} : \spec \mathbb F_q \to \mathscr P$ denote the map which
factors through the residual gerbe $B\mu_2$ at infinity and corresponds to the
map $\spec \mathbb F_q \to B\mu_2$ associated to the trivial double cover $\spec
\mathbb F_q \coprod \spec \mathbb F_q \to \spec \mathbb F_q$.
We wish to show that the fiber product
$X \times_{\mathscr P, \infty} \spec \mathbb F_q$ consists of two copies of
$\spec \mathbb
F_q$.
The fiber product above is identified with the fiber product
$X \times_{\mathscr P} \mathbb P^1 \times_{\mathbb P^1, \infty} \spec \mathbb F_q$, where the
map $\mathbb P^1 \to \mathscr P$ is the double cover branched over $0$. (In
	particular, on
coarse spaces $\mathbb P^1 \to \mathscr P$ corresponds to the map given in local
coordinates by $x \mapsto x^2$.)
The above fiber product $X \times_{\mathscr P} \mathbb P^1$ is a normal curve in
a neighborhood over
$\infty$, and hence
it is, locally around $\infty$, the normalization of 
$X \times_{\mathbb P^1_u} \mathbb P^1_t$, where again the map $\mathbb P^1_t \to
\mathbb P^1_u$ is induced by the double cover $u \mapsto t^2$.
In a neighborhood of $\infty$, we can express the projection map
$X \times_{\mathbb P^1_u} \mathbb P^1_t \to X$ as
\begin{align*}
	&\spec k[u,z,t]/(t^2 - u,z^2 -  u(\alpha + a_{n-1}u + \cdots + a_1u^{n-1} +
a_0u^n) \\
&\to \spec k[u,z]/(z^2 -  u(\alpha + a_{n-1}u + \cdots + a_1u^{n-1} +
a_0u^n).
\end{align*}
The source is then isomorphic to 
\begin{align*}
	\spec k[t,z]/(z^2 -  t^2(\alpha + a_{n-1}t^2 + \cdots + a_1t^{2(n-1)} +
a_0t^{2n}).
\end{align*}
This is not normal, but its blow up at the origin is normal in a
neighborhood of the origin, and its blow up is isomorphic to
\begin{align*}
\spec k[Z/U, t]/( (Z/U)^2 - (\alpha + a_{n-1}t^2 + \cdots + a_1t^{2(n-1)} +
a_0t^{2n}).
\end{align*}
The fiber of this over the point $z = t = 0$ is then simply
$k[Z/U]/( (Z/U)^2 - \alpha),$ which is a copy of $\mathbb F_{q^2}$ if $\alpha$
is a non-square, while it is two copies of $\mathbb F_q$ if $\alpha$ is a
square.
This fiber product is precisely the fiber product
$X \times_{\mathscr P, \infty} \spec \mathbb F_q$ we wished to compute, and so
we are done.
\end{proof}

\section{Algebra setup}
\label{section:algebra-background}
The goal of this section is to recall and set up some of the algebraic machinery
that we use. After introducing some of our notation that we use from higher
category theory in \autoref{subsection:algebra-notations}, we study the notion
of a homological epimorphism in \autoref{subsection:hom-epi}, which is key to our study of stable homology.
Throughout the rest of the section, we collect a number of basic facts in higher
algebra, primarily related to
homological epimorphisms. In 
\autoref{subsection:localization} we check localization and certain completions
are homological epimorphisms.
In \autoref{subsection:tools} we record a rigidity property for homological
epimorphisms, which shows in many cases that they are actually equivalences.
Finally, in \autoref{subsection:nilpotence}
we prove some well known lemmas about nilpotence of endomorphisms.

An alternate, more elementary perspective
on the calculations in this section and the next section is given in
\cite{landesmanL:an-alternate-computation} in the case that $G$ is a dihedral
group of order twice an odd prime, and $c$ is the conjugacy class of order $2$
elements.
This may be helpful for the reader to consult, especially if they are less
familiar with higher algebra.

\subsection{Algebra notations and recollections}
\label{subsection:algebra-notations}

Here, we recall some basic notions from higher algebra that we use here to study the rational homology of Hurwitz spaces. A convenient foundation for studying objects in homotopy theory is that of $\infty$-categories, as developed by Joyal and Lurie (see \cite{HTT}, \cite{HA} or \cite{kerodon}). If the reader prefers, they may also interpret the results below as about differential graded algebras: the use of spectra is not necessary to what we are doing here, and working with the derived category of the rational numbers suffices.

\subsubsection{The infinity category of spaces}
\label{subsubsection:spaces}
We use $\Spc$ to denote the $\infty$-category of spaces, which we view as symmetric monoidal with respect to the product. We use $\Spc_*$ to denote the $\infty$-category of pointed spaces, which is symmetric monoidal with respect to the smash product $\wedge$. Recall that for $(X,x)$ and $(Y,y)$ two pointed spaces, $X \wedge Y$ is defined as the quotient of $X \times Y$ by collapsing the subspace
$x \times Y \cup y \times X$. We have a symmetric monoidal functor $(-)_+:\Spc \to \Spc_*$ given by adding a disjoint base point to a space.

There is a functor $L:\Top \to \Spc$ from the category of topological spaces to the $\infty$-category of spaces, obtained by universally inverting all morphisms that are weak homotopy equivalences.

\subsubsection{$\EE_1$ algebras}
Given a symmetric monoidal $\infty$-category $\cC$, we can form the symmetric
monoidal $\infty$-category $\Alg(\cC)$
of associative (or $\EE_1$) 
algebras in $\cC$. 

\subsubsection{Stable $\infty$ categories}
We recall that an $\infty$-category $\cC$ is {\em stable} if it admits finite limits and finite colimits, a zero object, and pullback squares and pushout squares coincide. 
We use $\Sp$ to denote the symmetric monoidal stable $\infty$-category of spectra. Given a stable $\infty$-category $\cC$ and objects $X,Y \in \cC$, we use $\map_{\cC}(X,Y)$ to denote the spectrum of maps between $X$ and $Y$.

\subsubsection{$\EE_1$ rings}
We refer to an $\EE_1$-algebra in spectra as an $\EE_1$-ring. An ordinary associative ring $R$ can be viewed as an $\EE_1$-ring via the lax symmetric monoidal inclusion of abelian groups into spectra. All of the $\EE_1$-rings that we will be interested in are actually $\EE_1$-$\QQ$-algebras, so the reader may interpret the phrase $\EE_1$-ring to mean differential graded algebra up to quasi-isomorphism.

Given an $\EE_1$-ring, we can form the stable $\infty$-categories $\Mod_{\Sp}(R)$ and
$\Mod_{\Sp}(R^{\on{op}})$ 
of left and right $R$-modules in spectra respectively. 
For example, when $R$ is an ordinary associative ring, then $\Mod_{\Sp}(R)$ agrees with the derived $\infty$-category of $R$-modules, so that objects can be presented as (unbounded) chain complexes of $R$-modules. More generally given a DGA $R$, objects of $\Mod_{\Sp}(R)$ can be presented as differential graded modules over $R$ up to quasi-isomorphism.

Given $M \in \Mod_{\Sp}(R)$, we use either $\pi_iM$ or $H_iM$ depending on the context to refer to $\pi_i$ of the underlying spectrum of $M$. In the case $R$ is an ordinary ring and $M$ is presented as a chain complex of $R$-modules, this agrees with the $i^{\mathrm{th}}$ homology of the chain complex.

\subsubsection{$\EE_2$-central elements}
Let $R$ be an $\EE_1$ ring.
We say $r \in \pi_iR$ is {\em $\EE_2$-central} if there is an $R$-bimodule map $\Sigma^iR \to R$ sending $1$ to $r$. 
In particular, if $r$ is $\EE_2$-central, it is also graded commutative in the homotopy ring $\pi_*R$ of $R$. 
Given an $\EE_2$-central element $r$, we implicitly associate to it a choice of $R$-bimodule map refining $r$.

\subsubsection{Chains functor}
We have a symmetric monoidal functor $\tilde{C}_*(-;\QQ): \Spc_*\to \Sp$, 
taking $A$ to its reduced chains with coefficients in $\QQ$, with the comultiplication
coming from the diagonal map, which is dual to the cup product. Composing
$\tilde{C}_*(-;\QQ)$ with the symmetric monoidal functor $(-)_+$, we get the symmetric monoidal functor $C_*(-;\QQ): \Spc \to \Sp$. 
For $X$ a space, $H_i(X;\QQ) = \pi_iC_*(X;\QQ)$. 

\subsection{Homological epimorphisms}
\label{subsection:hom-epi}

Our goal is to study the algebra of chains on Hurwitz spaces, which forms an $\EE_1$-ring. We will study this $\EE_1$-ring by forming certain completions and localizations of it. An important property of these operations that makes them well behaved in this setting is that they are \textit{homological epimorphisms}. To define this notion, we recall the following lemma:
\begin{lemma}[{\cite[Lemma 2.9.8]{sheavesonmanifolds}}]\label{lemma:homepi}
	For a map $R \to S$ of $\EE_1$-rings with fiber $I$, the following conditions are equivalent:
\end{lemma}
\begin{enumerate}
	\item The forgetful functor $\Mod_{\Sp}(S) \to \Mod_{\Sp}(R)$ is fully faithful.
	\item The multiplication map $S \otimes_RS \to S$ is an equivalence.
	\item $I\otimes_RS =0$ 
	\item The multiplication map induces an equivalence $I\otimes_RI\cong I$.
\end{enumerate}

\begin{definition}\label{definition:homepi}
	A map $R \to S$ of $\EE_1$-rings is
	called a {\em homological epimorphism} (or hom. epi) 
	if any of the equivalent conditions of \autoref{lemma:homepi} are satisfied.
\end{definition}

\subsection{Localization and completion}
\label{subsection:localization}

Two sources of hom. epis for us will be certain localization and completion
maps. We recall below that given an $\EE_1$-ring $R$, some $i \in \mathbb Z$, and $r \in \pi_iR$,
we can form the completion $M^{\wedge}_r$, and localization $M[r^{-1}]$, of an
$M \in \Mod_{\Sp}R$, and that this forms an $\EE_1$-ring when $M=R$.

$\Mod_{\Sp}(R[r^{-1}])$ is defined as the presentable localization (or Bousfield localization) of the $\infty$-category $\Mod_{\Sp}R$ away from $R/r = \cof(\Sigma^iR\xrightarrow{r}R)$. In other words, it is the full subcategory of $\Mod_{\Sp}(R)$ such that $\map(R/r,-)$ vanishes, and the left adjoint of the inclusion of this subcategory sends $M$ to $M[r^{-1}]$. 

\begin{example}
	\label{example:localization-hom-epi}
	The map $R \to R[r^{-1}]$ is always a hom. epi \cite[Example 2.9.14]{sheavesonmanifolds}. More generally for a set of elements $S$ in the homotopy groups of $R$, we can form $M[S^{-1}]$, which can be described as the localization away from $R/r, r \in S$, and when $S$ is a finite set, agrees with iteratively localizing away from each element of $S$ in any order.
\end{example}

\begin{remark}
	\label{remark:ore-implies-colimit}
	When $r\in \pi_iR$ satisfies the left Ore condition, as defined in \cite[Definition
	7.2.3.1]{HA},
	the underlying spectrum of $M[r^{-1}]$ can be computed as the colimit of $M$ along the left multiplication by $r$ \cite[Proposition 7.2.3.20]{HA}. 
\end{remark}

We will also want to show certain completions are homological epimorphisms. We
next define completions, and later, in 
\autoref{lemma:completionhomepi}, we will show certain completions are
homological epimorphisms by showing they are localizations.
\begin{definition}
	\label{definition:} For $r \in \pi_iR$, the {\em $r$-completion}
of the $\infty$-category of $R$-modules is defined as the presentable localization of $\Mod_{\Sp}R$ away from $R[r^{-1}]$. In other words, $\Mod_{\Sp}(R)^{\wedge}_r$ is the 
full
subcategory of $R$-modules such that $\map(R[r^{-1}],-)$ vanishes. The inclusion of this subcategory 
$\Mod_{\Sp}(R)^{\wedge}_r \subset \Mod_{\Sp}R$
has a left adjoint sending $M$ to $M^{\wedge}_r$. \end{definition}

There are natural $\EE_1$-ring maps $R\to R^{\wedge}_r$, $R\to R[r^{-1}]$ induced by taking the map induced on endomorphism rings of the $R$-module $R$ via the completion and localization functor. Note also that $R[r^{-1}]=R[(r^{n})^{-1}]$ and $R^{\wedge}_r=R^{\wedge}_{r^n}$ for each $n\geq1$.

The following lemma is a standard manipulation of Bousfield localizations. See
\cite[Lemma 1.34]{ravenel1984localization} for a closely related statement.
\begin{lemma}\label{lemma:invertandcomplete}
	Let $R$ be an $\EE_1$-ring.	
	Suppose that $f:M \to N$ is a map of $R$-modules and $r \in \pi_iR$ is an element. The following are equivalent:
	
	\begin{enumerate}
		\item $f$ is an equivalence.
		\item $f[\frac 1 r]$ and $f^{\wedge}_r$ are equivalences.
		\item $f[\frac 1 r]$ and $R/r\otimes_Rf$ are equivalences.
	\end{enumerate}
\end{lemma}

\begin{proof}
	We have $(1) \implies (2)$ because localizations and completions preserve
	equivalences. Next,  $(2)\iff (3)$ follows since $M$ is an $R$ module with $r$
	acting invertibly if and only if $(R/r)\otimes_RM=0$. 
	
	Lastly, we show $(3) \implies (1)$. If $(3)$ is true, then
	$\map_{\Mod_{\Sp}(R)}(R/r,M) \to \map_{\Mod_{\Sp}(R)}(R/r,N)$ is an equivalence,
	as is $\map_{\Mod_{\Sp}(R)}(R[r^{-1}],M) \to \map_{\Mod_{\Sp}(R)}(R[r^{-1}],N)$.
	Hence, $R$ is an extension of $R[r^{-1}]$ by $\fib(R \to R[r^{-1}])$, which is
	generated under colimits by $R/r$, by definition of the Bousfield localization $(-)[r^{-1}]$.
	So, it follows that $M\cong \map_{\Mod_{\Sp}(R)}(R,M) \cong \map_{\Mod_{\Sp}(R)}(R,N) \cong N$ is an equivalence.
\end{proof}

The following lemma gives a situation under which completion happens to invert a central idempotent. 

\begin{lemma}\label{lemma:completionhomepi}
	Let $R$ be an $\EE_1$-ring, and $x\in\pi_0R$.
	Suppose there is some integers $n, m$ with $n > m$ and some $z \in
	\pi_0R$ so that $zx^n=x^m$. Also assume $x^{m'}$ is central for some $m'>0$.
	
	Then, $e:=z^{nm'}x^{n(n-m)m'}$ is a central idempotent in $\pi_*R$ so
	that $R\cong R[e^{-1}]\times R[(1-e)^{-1}]$ as an $\EE_1$-ring. 
	We also have $R^{\wedge}_x\cong R[(1-e)^{-1}], R[e^{-1}]\cong R[x^{-1}]$, so that $R \to R^{\wedge}_x$ is in particular a localization. Moreover, $\pi_*R^{\wedge}_x$ is the quotient of $\pi_*R$ by the two sided ideal generated by $x^m$.
\end{lemma}

\begin{proof}
	To see that $e=z^{nm'}x^{n(n-m)m'}$ is an idempotent, we have
	\begin{align*}
	(z^{nm'}x^{n(n-m)m'})^2=z^{2nm'}x^{2n(n-m)m'}=z^{nm'}x^{n(n-m)m'},
	\end{align*}
	where, the first equality uses that $x^{m'}$ is central and the second
	equality uses the equation $zx^n=x^m$ a total of $nm'$-times and the
	assumption that $n(n-m)m'\geq n$.
	
	We next show that $e$ is central. Indeed, for an arbitrary element $y \in \pi_*R$, we have 
	\begin{align*}
		ey&=z^{nm'}x^{n(n-m)m'}y \\
		&=z^{nm'}yx^{n(n-m)m'} \\
		&=z^{nm'}yz^{nm'}x^{2n(n-m)m'}\\
		&=z^{nm'}x^{2n(n-m)m'}yz^{nm'}\\
		&=x^{n(n-m)m'}yz^n\\
		&=ye.
	\end{align*}
	It follows that the localizations $R[e^{-1}],R[(1-e)^{-1}]$ are Ore localizations, so localizations can be computed as in \autoref{remark:ore-implies-colimit}.
	
	Since $R$ is the pullback
	$R[e^{-1}]\times_{R[e^{-1},(1-e)^{-1}]}R[1-e^{-1}]$, and
	$R[e^{-1},(1-e)^{-1}]=0$, we have $R \cong R[e^{-1}]\times R[(1-e)^{-1}]$.
	Since $e$ is sent to $1$ in $R[x^{-1}]$, and $x$ is invertible in
	$R[e^{-1}]$, we have $R[x^{-1}]=R[x^{-1},e^{-1}]=R[e^{-1}]$. Since the
	completion is the localization away from $R[x^{-1}]$, it follows that
	the completion is the other summand $R^{\wedge}_x \cong R[(1-e)^{-1}]$. To see that $x^m=0$ in the completion, we have 
	\begin{align*}
	(1-e)x^m=x^m-z^{nm'}x^{n(n-m)m'+m}=z^{nm'}x^{n(n-m)m'+m}-z^{nm'}x^{n(n-m)m'+m}=0.
	\end{align*}
	Because $x^m$ maps to $0$ in the component $R[(1-e)^{-1}] \simeq
	R^\wedge_x$, the kernel of surjective map $\pi_*R \to \pi_*R^{\wedge}_x$
	contains two sided ideal generated by $x^m$.
	From the above, this kernel is equal to 
	$\pi_*R[x^{-1}]$.
	Since $\pi_*R[x^{-1}]$ is contained in the two sided ideal generated by
	$x^m$ since any $y \in \pi_*R[x^{-1}]$ can be written as $yx^{-m} x^m$, 
	we obtain that the kernel is equal to the two sided ideal generated by
	$x^m$.
	\end{proof}

\begin{notation}\label{notation:completioncommute}
	Suppose that $R$ is an $\EE_1$-ring and $X: = \{x_1, \ldots,
	x_{|X|}\}$ is a finite set of elements
	in $\pi_0R$ each independently satisfying the conditions of \autoref{lemma:completionhomepi}.
	That is, for each $i$ with $1 \leq i \leq |X|$, there are some $n_i,
	m_i$ and $m_i'$ with $n_i > m_i$ and $z \in \pi_0 R$ so that $zx^{n_i} = x_{m_i}$
	and $x^{m_i'}$ is central.
	Then we use $R^{\wedge}_X$ to denote the iterated completion $(\cdots
	(R^{\wedge}_{x_1})\cdots)^{\wedge}_{x_{|X|}}$. By \autoref{lemma:completionhomepi}, this doesn't depend on the choice of ordering, since each completion is also a localization (see \autoref{example:localization-hom-epi}).
\end{notation}

\subsection{Tools for homological epimorphisms}
\label{subsection:tools}

\begin{lemma}\label{lemma:factorhomepi}
	Suppose we have maps $R \to S \to S'$ such that $R \to S$ is a hom. epi. Then $R \to S'$ is a hom. epi iff $S \to S'$ is a hom. epi.
\end{lemma}

\begin{proof}
	We use the description in \autoref{lemma:homepi}(1) of hom. epi. We have the composite of forgetful functors
	$$\Mod(S') \to \Mod(S) \to \Mod(R),$$
	where the second is fully faithful. It follows that the composite is fully faithful iff the first functor is.
\end{proof}
The following proposition is the key to relate the notion of a hom. epi to homological stability. Heuristically, it says that hom. epis to connective rings are determined by their $\pi_0$. See \cite[Theorem B]{hebestreit2024notehigherringtheory} for a generalization of this.
\begin{proposition}[Rigidity of homological epimorphisms]\label{proposition:rigidity}
	Consider a diagram of $\EE_1$-rings
	\begin{center}
		\begin{tikzcd}
			R	\ar[rr]\ar[dr] & & S'\\
			& \ar[ur] S,& 
		\end{tikzcd}
	\end{center}
	where $S \to S'$ is a map of connective $\EE_1$-rings that is an isomorphism on $\pi_0$, and where $R \to S$ and $R \to S'$ are homological epimorphisms. Then the map $S \to S'$ is an equivalence.
\end{proposition}

\begin{proof}
	Using \autoref{lemma:factorhomepi}, we learn that $S \to S'$ is a homological epimorphism. 	
	Letting $I$ be the fiber of $S \to S'$, we have $I \otimes_S S' = 0$. $I$ is the fiber of a
	map of connective $\EE_1$-rings, so it is $-1$-connective. We will
	inductively show on $i\geq-1$ that $\pi_iI$ is $0$. The inductive
	hypothesis implies that $I$ is $i$-connective, so and since $S \to S'$
	is an equivalence on $\pi_0$, we claim 
	\begin{align*}
		\pi_i(I) \cong \pi_i(I \otimes_S S) \cong
		\pi_i(I\otimes_S\pi_0(S)) \cong \pi_i(I\otimes_S\pi_0 (S'))
		\cong \pi_i(I\otimes_SS') \cong 0.
	\end{align*}
	Indeed, the second and fourth equivalences come from the fact that $I$ is
	$i$-connective and the maps $S' \to \pi_0S$ and $S \to \pi_0S$ are
	$1$-connective, so the map on tensor products is $i+1$-connective (see
	\cite[Lemma 3.1]{kk1}). The first equivalence is tautological, and the third and fifth equivalences follow by our
	assumptions. 
	Thus $\pi_i I=0$ for all $I$, so $I=0$, and hence $S \to S'$ is an isomorphism.
\end{proof}

The next lemma gives a criterion for being a hom. epi which is often easier to check in practice.
\begin{lemma}\label{lemma:homepipi0}
	Let $f:R \to S$ be a map of connective $\EE_1$-rings, and suppose that the map $\pi_0S\otimes_R\pi_0S \to \pi_0S\otimes_S\pi_0S$ is an equivalence. Then $f$ is a hom. epi.
\end{lemma}

\begin{proof}
	Consider the subcategory $C$ of pairs $(N,M)
	\in\Mod_{\Sp}(S^{\on{op}})\times\Mod_{\Sp}(S)$ such that $N\otimes_RM
	\cong N\otimes_SM$. Observe that $C$ is closed under colimits and desuspensions in each variable.
	
	Note that if $N,M$ are connective, then the pair $(N,M)$ is contained in
	$C$ if $\tau_{\leq n}N,\tau_{\leq n}M$ is contained in $C$ for all $n\in\NN$. This is because $\lim_n\tau_{\leq n}N\otimes_R\tau_{\leq n}M\cong N\otimes_RM$, and similarly for $S$.
	
	Thus to show that $f$ is a hom. epi., i.e the pair $(S,S)$ is in $C$, it
	suffices to show that $(\tau_{\leq n}S,\tau_{\leq n}S)$ is for each
	$n\in\NN$. However, $\tau_{\leq n}S$ is bounded in the the $t$-structure on $\Mod_{\Sp}(S)$ 
	and
	$\Mod_{\Sp}(S^{\on{op}})$, and $\pi_0S$ generates the hearts of these categories, so the result follows since we know by assumption that $(\pi_0S,\pi_0S) \in C$.
\end{proof}

The following lemma allows us to commute quotients by $\EE_2$-central elements with homological epimorphisms.
\begin{lemma}\label{lemma:e2centralhomepi}
	If $f:R \to S$ is a hom. epi, 
	and $x:R \to R$ is a bimodule map, then there is a natural isomorphism of $R$-modules $\cof(x)\otimes_R(S\otimes_RM) \cong S\otimes_R\cof(x)\otimes_RM$
\end{lemma}

\begin{proof}
	Consider the full subcategory of $R$-bimodules generated under colimits by $R$. 
	If $N$ is such a bimodule, then
	the functor $M\otimes_R(-)$ preserves the subcategory of $R$-modules that are $S$-modules, because it is true for the $R$-bimodule $R$ itself, and the subcategory is closed under colimits. Thus there is a natural transformation from $S\otimes_RN\otimes_RM \to N\otimes_R(S\otimes_RM)$ defined when $N$ is in this subcategory. To see that this map is an equivalence, we observe it is an equivalence for $N=R$, and the condition is closed under colimits.
\end{proof}

\subsection{Nilpotence of actions}
\label{subsection:nilpotence}
Here we record a few well known lemmas we use to understand the nilpotence of actions of $\EE_2$-central elements:

\begin{lemma}\label{lemma:squarecoflemma}
	Let $\cC$ be a stable $\infty$-category, and consider a diagram as below:
	\begin{center}
		\begin{tikzcd}
			X\ar[d,"f"]\ar[r,"0"]  & X'\ar[d,"f'"]\ar[r] & X''\ar[d,"f''"]\\
			Y\ar[r] & Y'\ar[r,"0"]& Y''.
		\end{tikzcd}
		\end{center}
	Then the composite of the associated maps on vertical cofibers $\cof(f) \to \cof(f') \to \cof(f'')$ is nullhomotopic. 
	
	In particular, if a natural endomorphism $g$ of the identity functor on $\cC$ acts by $0$ on $X$ and $Y$, then $g^2$ acts by $0$ on $\cof(f)$ for any map $X \to Y$.
\end{lemma}
\begin{proof}
	For the first claim, the composite $\cof(f) \to \cof(f') \to \Sigma X'$ is null since $X \to X'$ is null, so we obtain a factorization of $\cof(f) \to \cof(f')$ through $Y'$. Composing with the second horizontal map, since $Y' \to Y''$ is $0$, we obtain a nullhomotopy of the desired map.
	
	For the second claim, we apply the first claim when $X = X' = X'', Y =
	Y' = Y''$ and $f = f' = f''$ and the squares come from the natural endomorphism $g$.
\end{proof}

For the next lemma, we say that a 
monoidal stable $\infty$-category is {\em exactly monoidal} if the tensor product is
exact in each variable.
\begin{lemma}\label{lemma:squarenull}
	Let $\cC$ be an exactly
	monoidal stable $\infty$-category, let $f:\unit\to \Sigma^i\unit$ be a map in $\cC$, where $\unit$ is the monoidal unit.
	Then the map $\cof(f)\otimes f^{\otimes 2}$ is null. 
	
	In particular, if $R$ is an $\EE_1$-algebra, taking $\cC$ to be the category of $R$-bimodules, if $f:R \to \Sigma^iR$ a
	bimodule map, and $M \in \Mod_{\Sp}(R)$, then $f^{\otimes 2}$ acts by $0$ on $\cof f\otimes_RM$ as a map of $R$-modules.
\end{lemma}
\begin{proof}
	By identifying the category of $R$-bimodules with the category of colimit preserving endofunctors of $\Mod_{\Sp}(R)$, which is a monoidal stable $\infty$-category, the second statement follows from the first.
	
	We turn to proving the first statement. Let $\delta:\cof(f) \to \Sigma^{1}\unit$ be the boundary map in the cofiber sequence. The map $\delta\circ (\cof(f)\otimes f)$ is null since it agrees with $\cof(f) \xrightarrow{\delta} \Sigma^{1} \unit \xrightarrow{f} \Sigma^{1+i} \unit$. 
	It follows that $\cof(f)\otimes f$ factors through $\unit$. But then since $\unit \xrightarrow{f}\Sigma^i\unit \to \cof(f)$ agrees with $\unit \to \Sigma^{-i}\cof(f) \xrightarrow{f}\cof(f)$, it follows that $f^{\otimes2}$ vanishes on $\cof(f)$.
\end{proof}

\section{Computing stable homology}
\label{section:computing-stable-homology}
The goal of this section is to compute the stable homology of the Hurwitz spaces of interest.
First, in
\autoref{subsection:racks},
we give some background on racks, 
which is the natural setting for
our main stability result.
Next, in \autoref{subsection:main-stability} we state the main result,
\autoref{theorem:mainhom}, and explain how this implies our main results on
stable homology.
Then, in \autoref{subsection:centrality} we prove that
multiplication by certain components of Hurwitz spaces act centrally on homology.
In \autoref{subsection:comparison-map} we reduce the proof of
\autoref{theorem:mainhom} to showing that a certain map of pointed bar constructions is a homotopy equivalence. We prove this equivalence in
\autoref{subsection:equivalence}, relying on
a topological model for the relevant spaces provided by \autoref{appendix:topological-model}.

\subsection{Background on racks}
\label{subsection:racks}
Although we will ultimately only be interested in working with Hurwitz spaces where
the monodromy lives in a certain conjugacy class in a group, all of the
arguments go through just as easily when working with racks.
Hence, for this section, we will work with racks.
One concise reference for the basics of racks is
\cite{etingof2003rack}.
\begin{definition}
	\label{definition:rack}
	A {\em rack} is a set $c$ with a binary operation $\triangleright : c \times c
	\to c$
such that the following two properties hold:
\begin{enumerate}
	\item For every $x \in c$, the map $\phi_x : c \to c, y \mapsto x
		\triangleright y$ is a bijection.
	\item For any $x,y,z \in c$, $x \triangleright (y\triangleright z) =(x\triangleright y)\triangleright
		(x\triangleright z)$.
\end{enumerate}
\end{definition}

\begin{remark}\label{remark:rackdefinition}
	We may view the map $\triangleright : c \times c \to c$ as a map $\phi: c \to
	\hom(c,c), x \mapsto \phi_x$.
	In this setting, \autoref{definition:rack}(2) is equivalent to 
	the condition that $\phi$ is equivariant for the actions of $c$, where
	$z \in c$ acts on the source via $x \mapsto z \triangleright x$ and the
	action of $z \in c$ on the target
	is given by sending $\alpha(\bullet) \mapsto z \triangleright \alpha( (z
		\triangleright)^{-1}
	\bullet)$, for $\alpha \in \hom(c,c)$.
	Then, \autoref{definition:rack}(1) is the condition that these
	endomorphisms $\phi_x$ are automorphisms.
\end{remark}

\begin{example}
	\label{example:group-rack}
	The only type of rack we will need for our application to the
	Cohen--Lenstra heuristics and to prove
	\autoref{theorem:stable-homology-intro} is a rack which is a conjugacy
	class $c$ in a group $G$, in which case we take $x \triangleright y :=
	x^{-1}yx$,
	where, on the right hand side, $x^{-1}yx$ is viewed as multiplication in the group.
\end{example}

We now introduce some further notation we will use for racks.
\begin{notation}
	\label{notation:rack-iterate}
	For $c$ a rack and $g, g_1, \ldots, g_n \in c$, we will use the notation 
	\begin{align*}
		(g_1\cdots g_n) \triangleright g := (g_n \triangleright ( \cdots (g_2 \triangleright (g_1
		\triangleright g)) \cdots )).
	\end{align*}
	Note the reversal of order in the iteration of the $g_i$, which is
	compatible 
	with the convention from \autoref{example:group-rack}.
\end{notation}


\begin{notation}
	\label{notation:hurwitz-union}
	For $c$ a rack, we use $\Hur^c_n$ to denote the {\em pointed Hurwitz
	space} of degree $n$ for $c$, which is by definition the homotopy quotient
	$(c^n)_{hB_n}$ of $c^n$ by the action of $B_n$, where the positive generator of $B_2 \simeq \mathbb Z$ acts
	via $(g, h) \mapsto (h, h \triangleright g)$.
We use $\Hur^{c}$ to refer to the union of the Hurwitz spaces
$\coprod_n\Hur_n^{c}$, and let $\Conf$ refer to the union of the configuration
spaces $\coprod_n\Conf_n$, for $\Conf_n$ denoting the configuration space
parameterizing unordered sets of $n$ distinct points in $\mathbb C$.
\end{notation}
\begin{notation}
	\label{notation:component-rack}
We can identify the connected components of the pointed Hurwitz space
$\Hur^c_n$ with orbits of the $B_n$ action on $c^n$.
Under this identification, for
$x_1, \ldots, x_n \in c$ we use $[x_1] \cdots [x_n]$ to denote the
connected component of $\phur G n c {\mathbb C}$ corresponding to the
$B_n$ orbit of the tuple $(x_1, \ldots, x_n)$.
We will also use $\alpha_x$ as equivalent notation for $[x] \in \pi_0 \phur G n
c {\mathbb C}$.
See also \autoref{notation:rack}.
\end{notation}

%

The following definition is a generalization of the non-splitting condition 
to the setting of racks.
\begin{definition}\label{defn:nonsplitting}
	We say a rack $c$ is \textit{connected} if there is a single orbit under the operations $x\triangleright(-)$ for each $x \in c$.
	We say that a rack $c$ is \textit{non-splitting} if every nonempty subrack $c' \subset c$ is connected.
\end{definition}

\begin{remark}
	\label{remark:}
	It follows from the definitions that
	if $c$ comes from a conjugacy class in a group $G$ in the sense of
	\autoref{example:group-rack}, then $(G,c)$ is
non-splitting in the sense of \autoref{definition:nonsplitting} if $c$ is
a non-splitting rack.
\end{remark}

Our main result computes the stable homology of Hurwitz spaces associated
to non-splitting racks $c$.
See \autoref{subsection:nonsplitting-examples} for some examples of
non-splitting racks $c$ relevant
for the Cohen--Lenstra heuristics.

\subsection{The main stability result}
\label{subsection:main-stability}
We now set up some notation we will use in this section and state our main
result.
The map $\Hur^{c} \to \Conf$ is a map in $\Alg(\Spc)$.
To do this, we introduce a suitable comparison map:

\begin{notation}\label{construction:homologystructure}
	For a rack $c$, let $\Conf^{c}$ be the pullback $\Conf\times_{\pi_0\Conf}\pi_0\Hur^{c}$, 
	so that there is a map $\Hur^{c} \to \Conf^{c}$ in $\Alg(\Spc)$. 
	Let $A:=C_*(\Hur^{c};\QQ)$ and $A':=C_*(\Conf^{c};\QQ)$.
	Applying $C_*(-;\QQ)$ to the map $\Hur^{c} \to \Conf^{c}$,
	we obtain a map $v: A \to A'$ of $\EE_1$-algebras 
	in spectra, i.e in the $\infty$-category $\Alg(\Sp)$.
\end{notation}
The next result implies our main theorem on stable homology. After stating it,
we will explain why it implies \autoref{theorem:stable-homology-intro} from the
introduction. 
\begin{theorem}\label{theorem:mainhom}
	With notation as in \autoref{construction:homologystructure},
	suppose $c$ is a finite non-splitting rack.
	Let $U \in \pi_0A$ be $\sum \alpha_i$, where each $\alpha_i$ comes from a class in $\pi_0\Hur_n^c$ for $n>0$ which is central.
	Then, the map $v[U^{-1}]: A[U^{-1}] \to A'[U^{-1}]$ is an equivalence.
\end{theorem}

We prove this in \autoref{subsubsection:mainhom}.
From this, we easily deduce 
\autoref{theorem:stable-homology-intro}
from the introduction, computing the stable homology of our Hurwitz spaces.
\subsubsection{Proof of \autoref{theorem:stable-homology-intro} assuming
\autoref{theorem:mainhom}}
	\label{subsubsection:proof-stable-homology}

	First, note that there is an identification between the pointed Hurwitz space
	$\Hur_n^{G,c}$ and $\Hur_n^{c}$ as both can be described as the homotopy
	quotient $(c^n)_{hB_n}$, see \autoref{subsubsection:hurwitz-quotient}.
	Let us fix a value of $i\geq0$. 
	By \cite[Theorem 6.1]{EllenbergVW:cohenLenstra},
	there exist constants
	$I,J$, and $D$\footnote{In fact, one may take $D = 1$ by 
\cite[Proposition
A.3.1]{ellenbergL:homological-stability-for-generalized-hurwitz-spaces}.}
	so
	that for any $i \geq 0$ and $n > Ii + J$ and $U := \sum_{h
\in c} \alpha_h^{D\ord(h)}$,
$\pi_iA_n$ agrees with $\pi_iA[U^{-1}]_n$.

	Because of the definition of $\Conf^{c}$, it follows that $\pi_iA'_n$
	also agrees with $\pi_iA'[U^{-1}]_n$ as soon as $\pi_0A$ has stabilized.
	Thus it follows from \autoref{theorem:mainhom} that for any $i \geq 0$ and $n > Ii + J$, the map $\pi_iA_n \to \pi_iA'_n$
	is an isomorphism, since it agrees with the map induced by $v[U^{-1}]$
	on $\pi_i$. But
	this is exactly saying that $H_i(\Hur_n^{c};\QQ) \cong
	H_i(\Conf_n^{c};\QQ)$ for $n\gg0$, which is a reformulation of the
	statement that for each component $Z$ of $\Hur_n^{c}$, we have
$H_i(Z;\QQ) \cong
H_i(\Conf_n;\QQ)$, proving \autoref{theorem:stable-homology-intro}.
\qed

\subsection{Centrality}
\label{subsection:centrality}
In this subsection, we let $c$ be an arbitrary rack. The goal of this subsection is to show that if $x \in \pi_0\Hur^c$ is central, then $x$ is actually central in a much stronger sense. Namely, it is $\EE_2$-central in $\Hur^c$, meaning that multiplication by $x$ can be refined to a $\Hur^c$-bimodule map.

Informally, in the case of Hurwitz spaces associated
to groups, the reason for this stronger form of centrality is that there is a homotopy
$[\alpha] [\beta] \simeq [\beta] [\beta^{-1} \alpha \beta]$ obtained by changing the direction from which one
concatenates $G$-covers. Thus if $\beta$ centralizes $c$, this identifies left and
right multiplication by $\beta$. 
The above argument is not enough to make multiplication by $\beta$ into a bimodule map, but it can be made as such. We now argue this below and further generalize from groups to racks:
\begin{lemma}\label{lemma:centralityspace}
	Using notation as in \autoref{notation:completioncommute},
	we let $c$ be a rack.
	Let $\gamma\in \pi_0\Hur^{c}$ 
	be central. If we
	view $\Hur^{c}$ as an $\EE_1$-algebra in the $\infty$-category
	$\Spc_{/\Conf}$, then multiplication by $\gamma$ in $\Hur^{c}$ can be refined to the structure of an 
  $\Hur^{c}$-bimodule map.
\end{lemma}
\begin{proof}
	We can consider the $\EE_2$-monoidal $\infty$-category of local systems
	on $\Conf$, \linebreak$\Spc^{\Conf}$. This can be identified with the
	$\infty$-category of tuples $(X_i)_{i \in \NN}$, where $X_i\in \Spc^{BB_i}$ is a space with an action of the braid group on $i$-letters. Indeed this follows because $\Conf$ is the free $\EE_2$-algebras in spaces on a point, and is $1$-truncated, and $\coprod_n BB_n$ is the free $\EE_2$-monoidal category on an object \cite[Theorem 4]{joyal1986braided}, and is a groupoid, so they agree.

	We can view $\Hur^{c}$ as the colimit of the 
	multiplicative local system given by the tuple $(c^i)_{i \in \NN}$, 
	where the action of the braid group is determined by the fact that the
	twist sends a pair $(x,y) \in c\times c$ to $(y,y \triangleright x)$, 
	and the multiplication is given by concatenation. Note that since this is a local system of sets, we may view $\Hur^{c}$ as an object in the (ordinary) braided monoidal category $\Fun(\coprod_n BB_n,\Set)$. Here the tensor product of $(X_i)_{i \in \NN},(Y_i)_{i \in \NN}$ is given by $(\coprod_{i+j=n}\Ind^{B_n}_{B_i\times B_j}(X_i\times Y_j))_{n \in \NN}$, and the braiding is determined by the formula given in \cite[p7]{joyal1986braided}.
	
	Let $m$ be such that $\gamma \in \Hur_m^c$.
	Consider $m_!(*)$, where $m$ is considered as the inclusion of the base point $* \to BB_m$. Explicitly, $m_!(*)$ is the local system 
	which is $B_{m}$ in degree $m$ 
	and empty otherwise. There is a left module map $(c^i)_{i \in \NN} \otimes m_!(*)\to
	(c^i)_{i \in \NN}$, given by right multiplying by the class $\gamma$ on the identity element of $B_m$. Here we realize $(c^i)_{i \in \NN}\otimes m_!(*)$ as a bimodule via the $\EE_2$-monoidal structure, i.e using the twist map $(c^i)_{i \in \NN}\otimes m_!(*) \to m_!(*)\otimes (c^i)_{i \in \NN}$. The fact that this multiplication map is a bimodule map follows if the diagram
	
	\begin{center}
		\begin{tikzcd}
			(c^i)_{i \in \NN}	\otimes m_!(*)\ar{rr}{\bullet_r}
			\ar{dr}{\on{tw}} & &(c^i)_{i \in \NN} \\
			& \ar{ur}{\bullet_l}m_!(*)\otimes(c^i)_{i \in \NN} & 
		\end{tikzcd}
	\end{center}
	commutes, where $\bullet_l$ is left multiplication, $\bullet_r$ is right
	multiplication, and $\on{tw}$ is the twist map.
	Unraveling the definitions, this holds because the twist sends a tuple
	$(x,\gamma),x \in c^i$ to $(\gamma,x)$, which uses that $\gamma$ is central in $\pi_0\Hur^c$, so that $(\gamma\triangleright)$ acts trivially on $c$.
	
	Thus we have constructed a bimodule map of multiplicative local systems
	on $\Conf$ induced by multiplication by $\gamma$. By applying the $\EE_2$-monoidal functor $\colim:\Spc^{\Conf} \to \Spc_{/\Conf}$, (where the map to $\Conf$ is obtained via the colimit of the terminal local system) 
	we see that this bimodule map refines the multiplication by $\gamma$ map.
\end{proof}

\begin{corollary}
	\label{lemma:centrality}
		Let $\gamma\in \pi_0\Hur^{c}$ 
		be central. Then the associated class $[\gamma] \in \pi_0A$ is $\EE_2$-central.
\end{corollary}

\begin{proof}
	 This follows by applying the functor $C_*(-;\QQ)$ to the bimodule map produced in \autoref{lemma:centralityspace}.
\end{proof}

\subsection{Decomposition of the comparison map}
\label{subsection:comparison-map}

To check that the comparison map $v$ is an equivalence, we geometrically decompose the algebra $A$ based on central elements acting on it.
We accomplish this decomposition in \autoref{lemma:reductionstep}, which roughly
says it suffices to check $v[U^{-1}]$ is an equivalence after suitable 
localizations and completions.

As a first step toward this, we note that certain completions of localization of $A$ are well behaved in the sense of \autoref{lemma:completionhomepi}:

\begin{notation}
	\label{notation:rack}
	Let $c$ be a finite rack.
	Given an element $h \in c$, let us use $\alpha_h$ as alternate notation
	for the	element $[h]$ in $\pi_0\Hur_1^{c}$, as defined in
\autoref{notation:component-rack}. 
	For a subset $S \subset c$, we use $\alpha_S$ to denote the set
	$\{\alpha_s:s\in S\}$.
	We use $M[\alpha_S^{-1}]$ to denote the localization of some $A$-module $M$ at the
	set $\{\alpha_s,s\in S\}$ and we similarly use
	$M^\wedge_{\alpha_S}$ to denote the completion of $M$ at $\{\alpha_s : s \in
	S\}$ (when this doesn't depend on an order of $S$).
\end{notation}

The following lemma is a generalization of \cite[Proposition
3.4]{EllenbergVW:cohenLenstra} from groups to racks, whose proof requires no new
ideas.
\begin{lemma}\label{lemma:evwlemma}
	Let $c$ be a finite connected rack, and let $g \in c$. For every
	$N$ sufficiently large, any tuple of elements $(g_1,\ldots,g_N) \in c^N$
	generating the rack $c$ is equivalent under the braid group action to
	some tuple of the form $(g,g_1',\ldots,g_{N-1}')$ where $g_1',\ldots,g_{N-1}'$ generates $c$.
	Similarly, $(g_1,\ldots,g_N)$ is equivalent under the braid group action
	to some tuple of the form $(g_1'',\ldots,g_{N-1}'',g)$ where
	$g_1'',\ldots,g_{N-1}''$ generates $c$.
\end{lemma}

\begin{proof}
	We summarize the proof, leaving details to the reader, which can also be found in \cite[Proposition
	3.4]{EllenbergVW:cohenLenstra} in the case that $c$ is a conjugacy class in a
	group.
	We will only prove the first statement regarding equivalence to
	$(g,g_1',\ldots,g_{N-1}')$. The second statement regarding
	equivalence to $(g_1'',\ldots,g_{N-1}'',g)$ has an analogous proof.

	Choose $N$ large enough so that some $g' \in c$ appears $n+1$ times,
	where $n$ is such that $\phi_{g'}^n=\id$. Using the action of the braid
	group, we can move all of these copies of $g'$ to the left to see that
	$(g_1,\ldots,g_N)$ is equivalent under the braid group action to
	$(\underbrace{g',\ldots,g'}_{n+1},h_1,\ldots,h_{N-n-1})$, where $g',h_1, \ldots, h_{N-n-1}$
	still generate the rack. Because $\phi_{g'}^n=\id$, by using the element
	of the braid group that does a full twist of the last element around the
	rest, we see that $(\underbrace{g',\ldots,g'}_n,h) \in c^{n+1}$ is equivalent to
	$(\underbrace{h\triangleright g',\ldots,h\triangleright g'}_n,h) \in
	c^{n+1}$, where there are $n$ copies of $g'$. Using this along with the
	fact that $g',h_1,\ldots,h_{N-n-1}$ generate the rack and the rack is connected,
	shows that we can modify $(\underbrace{g',\ldots,g'}_{n+1},h_1,\ldots,h_{N-n-1})$ by an element of the braid group to change the first $n$ copies of $g'$ into $n$ copies of $g$, while keeping the fact that the rest of the elements generate the rack.
\end{proof}

\begin{lemma}\label{lemma:completionhomepihur}
	With notation as in \autoref{notation:rack}, let $c$ be a finite nonempty
	non-splitting rack
	and let $c' \subset c$ be a subrack. In
	$A[\alpha_{c'}^{-1}]$, any element $\alpha_{\beta}$ with $\beta \in c-c'$ satisfies the conditions of \autoref{lemma:completionhomepi}. 
	Thus for an $A$-module $M$, we can form
	$M[\alpha_{c'}^{-1}]^{\wedge}_{\alpha_{c-c'}}$ as in \autoref{notation:completioncommute}.
\end{lemma}
\begin{proof}
	By \autoref{lemma:centrality}, $\alpha_{\beta}^{\ord(\beta)}$ is $\EE_2$-central where $\ord(\beta)$ is the order of $\phi_{\beta}$ as an automorphism of $c$. It remains to prove that
	$y\alpha_{\beta}^2=\alpha_{\beta}$ for some $y \in \pi_*A[\alpha_{c'}^{-1}]$.
	To do this, choose $\gamma\in c'$. 
	Note that because $c$ is non-splitting, we have $\langle c', \beta
	\rangle$ is connected.
	Applying \autoref{lemma:evwlemma} to the subrack
	$\langle c', \beta \rangle \subset c$ generated by $c'$ and $\beta$,
	, 
	for $N$
	sufficiently large, 
	we can write $
\alpha_{x_1}
	\cdots \alpha_{x_N}\alpha_{\beta} =
	\alpha_\gamma^N\alpha_\beta$,
	where $x_1, \ldots, x_N \in c$ together generate the same
	subrack as $\gamma, \beta$.
	Applying \autoref{lemma:evwlemma} 
again, we can write
	$\alpha_{y_1} \cdots
	\alpha_{y_{N-1}}\alpha_{\beta} =\alpha_{x_1} \cdots \alpha_{x_N}$.
	Taking $y := \alpha_\gamma^{-N} \alpha_{y_1} \cdots \alpha_{y_{N-1}}$, we therefore have
	$y\alpha_\beta^2 = \alpha_\beta$. This verifies the conditions of
	\autoref{lemma:completionhomepi}.
	\end{proof}

	The following reduces one to considering stable homology classes on
	which $\alpha_{c'}$, for a subrack $c' \subset c$, act by isomorphisms, and the complementary set act nilpotently.
\begin{lemma}\label{lemma:reductionstep}
	Let $c$ be a finite non-splitting rack. Using notation from \autoref{notation:rack} and
	\autoref{construction:homologystructure}, and letting $v,U$ be as in \autoref{theorem:mainhom},
	the map $v[U^{-1}]$
	is an equivalence if, for every nonempty
	subrack $c' \subset c$, $(v[\alpha_{c'}^{-1}])^\wedge_{\alpha_{c-c'}}$ is an
	equivalence, where the completion is taken as an $A$-module.
\end{lemma}
\begin{proof}
	By repeatedly applying \autoref{lemma:invertandcomplete}, we learn that it
	is enough to show that for every subset $c'\subset c$,
	\begin{align}
		\label{equation:localize-and-quotient-map}
		(A/\alpha^{\ord(\beta_{|c-c'|})}_{\beta_{|c-c'|}})\otimes_A\cdots
		\otimes_A
		(A/\alpha^{\ord(\beta_1)}_{\beta_1})\otimes_Av[U^{-1}][\alpha_{c'}^{-1}] 
	\end{align}
is an equivalence, where $\coprod_{i=1}^{|c-c'|}\beta_i=c-c'$. Here, we have used \autoref{lemma:e2centralhomepi} to commute the localization with the quotienting,
and \autoref{lemma:centrality} to make sense of the iterated quotient. Namely,
because $\alpha_{\beta_i}^{\ord(\beta_{i})}$ is a bimodule map by
\autoref{lemma:centralityspace}, we can tensor with the $A-A$-bimodule $(A/\alpha^{\ord(\beta_1)}_{\beta_1})$ to make sense of the tensor product above as an $A$-module.

	Note that on the source and target of this map, for each $\beta \in c-c'$,
	$\alpha_{\beta}^{\ord(\beta)n_{\beta}}$ acts by 0 for some $n_{\beta}$. This
	follows from applying \autoref{lemma:squarenull} to see that after
	quotienting out by $\alpha_{\beta}^{\ord(\beta)}$,
	$\alpha_{\beta}^{2\ord(\beta)}$ acts by $0$, and using
	\autoref{lemma:squarecoflemma} to see that after each further quotient,
	the smallest power of $\alpha_{\beta}^{\ord(\beta)}$ that acts by $0$ gets multiplied by at most $2$.
	
	If $c'$ is empty, then we claim that both the source and target of the map 
	\eqref{equation:localize-and-quotient-map} 
	are $0$. 
	We claim that some power of $U$ acts by $0$ on the source and target, since a large
	enough power of $U$, $U^j$, is in the two sided ideal of $\pi_0A$ generated by
	$\alpha_{\beta}^{\ord(\beta)n_{\beta}}$. This is because by the
	pigeonhole principle, one can see that any element of
	$\pi_0\Hur^{G,c}_n$, written as a product of elements in $c$ for $n$
	large enough contains $\max_{\beta \in c}(\ord(\beta)n_{\beta})$ many copies of a single $\beta\in c$. By pulling these to the left using the braid group action, we see the element is divisible by $\alpha_{\beta}^{\ord_{\beta}n_{\beta}}$.
	Since $U^j$
	is in the ideal generated by the smallest powers of $\alpha_{\beta}^{\ord(\beta)}$ that act trivially, and hence it acts trivially.
	Since $U$ acts both as an isomorphism and nilpotently, the source and
	target of \eqref{equation:localize-and-quotient-map} with $c' =
	\emptyset$ are both $0$.
	
We next claim that 
the source and target of 
	\eqref{equation:localize-and-quotient-map}
	vanish unless $c' \subset c$ is a subrack.
	Indeed, observe that
	$\alpha_a\alpha_b=\alpha_b \alpha_{b \triangleright a}$. Hence, we learn that if $\alpha_a$ and $\alpha_b$ act
	invertibly, then so does $\alpha_{b \triangleright a}$. Thus the source and target of 
	\eqref{equation:localize-and-quotient-map}
	vanish unless $c'$ is closed under the operation $\triangleright$, i.e.,
	$c'$ is a subrack.

	Suppose the map
	\begin{align}
		\label{equation:localize-and-quotient-invert-u}
(A/\alpha^{\ord(\beta)}_{\beta_{|c-c'|}})\otimes_A\cdots \otimes_A
(A/\alpha^{\ord(\beta)}_{\beta_1})\otimes_Av[\alpha_{c'}^{-1}]
	\end{align}
	(which is similar to \eqref{equation:localize-and-quotient-map} but here
	we have not inverted $U$)
	is an equivalence for each subrack $c' \subset c$.
	Using \autoref{lemma:e2centralhomepi} to commute the
	localization and quotienting, we find that if
	\eqref{equation:localize-and-quotient-invert-u}
	is an equivalence for each
	subrack $c' \subset c$, \eqref{equation:localize-and-quotient-map} will also be an
	equivalence.
	Hence, it is enough to show
	\eqref{equation:localize-and-quotient-invert-u} is an equivalence
	for each subrack $c' \subset c$.
	Finally, we note that
	\eqref{equation:localize-and-quotient-invert-u} is an equivalence
	for a fixed $c' \subset c$ if and only if
	$(v[\alpha_{c'}^{-1}])^\wedge_{\alpha_{c-c'}}$ is an
	equivalence by
	using \autoref{lemma:completionhomepihur} and \autoref{lemma:invertandcomplete}.
	\end{proof}

\subsection{Showing the map is an equivalence}
\label{subsection:equivalence}

In this section we show that the maps in \autoref{lemma:reductionstep} are an
equivalence for each subrack $c'\subset c$. In the case that $c'=c$, the fact
that the map in \autoref{lemma:reductionstep} is an equivalence comes from the fact
that the homology of the group completion of $\Hur^{c}$ can be identified with
the rack homology of the rack $c$, which is computed in \cite{etingof2003rack}.
We verify this in \autoref{proposition:rackhomology}.
In general, we show that completion at $\alpha_{c-c'}$ (see \autoref{notation:rack}) corresponds to projecting in
pointed spaces to the components of $\Hur^{c'}$ at the
level of homology, allowing us to again apply \autoref{proposition:rackhomology} for the rack $c'$ in this case.
Identifying the completion in terms of $\Hur^{c'}$ is at the core of our proof, and is proven using a homological epimorphism argument whose main input is a topological one identifying two bar constructions in pointed spaces as homotopy equivalent.

\begin{proposition}\label{proposition:rackhomology}
	Use notation from \autoref{construction:homologystructure} and
	\autoref{notation:rack}, let $c$ be a finite connected rack.
	The map $v[\alpha_{c}^{-1}]$ is an equivalence.
\end{proposition}

\begin{remark}
	\label{remark:}
	\autoref{proposition:rackhomology} was essentially already known, and an
	alternate proof in the case that $c$ is a quandle with a single orbit
	under the action on itself
	proof can be found in \cite[\S
	6.5]{bianchi:deloopings-of-hurwitz-spaces}.\footnote{We note that the argument presented in 
\cite[\S 6.5]{bianchi:deloopings-of-hurwitz-spaces}
describing the homology of the group completion isn't quite correct, because it
is claimed there that a certain algebra $\mathcal A(\mathcal Q)$ is given by
$\mathbb Q[x_S| S \in \on{conj}(\mathcal Q_+)]/(x^2_S)$, but it should really be
$\mathbb Q[x_S| S \in \on{conj}(\mathcal Q_+)]/(x_S)^2$. In other words,
the algebra written there is the homology of a product of $2$-spheres, while it
should be the homology of a wedge of $2$-spheres. One can fairly easily correct
this error and obtain a correct statement. We note that in the case there is
only a single conjugacy class in $\mathcal Q_+$, which is the case relevant to
the present article, the argument given in
\cite[\S 6.5]{bianchi:deloopings-of-hurwitz-spaces}
is correct.
}
	
	The proof we give here is modeled off of the argument in \cite[Corollary
	5.4]{randal-williams:homology-of-hurwitz-spaces}. We note that the
	argument given there has a small gap, 
	where it is claimed on the first line of \cite[p.
	15]{randal-williams:homology-of-hurwitz-spaces} that $B\Hur^c_G \to
	B\Conf$ being a rational homology equivalence implies $\Omega B \Hur^c_G \to \Omega B
	\Conf$ is a rational homology equivalence on each path component.
	In general this does not hold, but it does hold when the spaces are \textit{weakly simple} (see \autoref{lemma:simplespacesloopspace} below), as was very helpfully pointed out to us by Oscar
	Randal-Williams.
	This observation allows us to fix the small gap.
\end{remark}

\begin{definition}\label{definition:weaklysimple}
	A connected space $X$ is \textit{weakly simple} if $\pi_1X$ acts trivially on $\pi_nX$ for $n>1$.
\end{definition}

\begin{lemma}\label{lemma:simplespacesloopspace}
	Suppose that $k$ is a localization of the integers and $X \to Y$ is a $k$-homology equivalence of weakly simple spaces that is also a $\pi_1$-equivalence. Then $\Omega X \to \Omega Y$ is also a $k$-homology equivalence.
\end{lemma}

\begin{proof}
	$\Omega X \to \Omega Y$ can be identified with the map $\pi_1X\times
	\Omega \tilde{X} \to \pi_1X \times \Omega \tilde{Y}$, 
	where
	$\tilde{X},\tilde{Y}$ are the universal covers of $X,Y$ respectively.
	For simply connected spaces, $k$-homology equivalences and $k$-homotopy
	equivalences are the same, so it is enough to show that
	$\tilde{X}\to\tilde{Y}$ is a $k$-homotopy equivalence. We will show
	inductively on $i\geq 1$ that it is a $k$-homotopy equivalence in degrees
	$\leq i$. 
	The base case $i = 1$ holds by assumption.
	For the inductive step, if we know the result in degrees $\leq
	i$, then we aim to show it holds in degree $i+1$. Consider the fiber sequences $\tau_{\geq i+1}X \to X \to \tau_{\leq i}X$ and $\tau_{\geq i+1}Y \to Y \to \tau_{\leq i}Y$. The $k$-homology of $\tau_{\leq i}X$ and $\tau_{\leq i}Y$ agree by the inductive hypothesis. By the Hurewicz theorem and our assumption that $X,Y$ are weakly simple, 
	the action of $\pi_1X$ on $H_{i+1}(\tau_{\geq i+1}X;k) = \pi_{i+1}X\otimes k$ and
	the action of $\pi_1Y$ on $H_{i+1}(\tau_{\geq i+1}Y;k) = \pi_{i+1}Y\otimes k$ are trivial. It follows by examination of the map of Serre spectral sequences of these fiber sequences and the fact that $X \to Y$ is a $k$-homology equivalence that the map $H_{i+1}(\tau_{\geq i+1}X;k) \to H_{i+1}(\tau_{\geq i+1}Y;k)$ is an equivalence. 
	Using the Hurewicz theorem to identify
$\pi_{i+1}(\tilde{X})\otimes k \simeq
H_{i+1}(\tau_{\geq i+1}X;k)$, and similarly for $Y$ in place of $X$,
	we have $\pi_{i+1}(\tilde{X})\otimes k \simeq
	H_{i+1}(\tau_{\geq i+1}X;k) \simeq 
H_{i+1}(\tau_{\geq i+1}Y;k)
	\simeq
	\pi_{i+1}(\tilde{Y})\otimes k$, completing the induction.
\end{proof}

The following lemma identifies the target of the comparison map after group completion:
\begin{lemma}\label{lemma:targetgroupcompletion}
	The map
	\begin{align}
		\label{equation:psi-definition}
		\psi:	\Omega B\Conf^{c} \to
		\Omega B\Conf \times_{\Omega B\pi_0\Conf}\Omega
		B\pi_0\Conf^{c}
	\end{align}
	is an equivalence. 
\end{lemma}

\begin{proof}
	We will first show that 
	Note that $\pi_1(\psi) : \pi_1(\Omega B\Conf^{c}) \to
	\pi_1(\Omega B\Conf \times_{\Omega B\pi_0\Conf}\Omega
	B\pi_0\Conf^{c})$
	is an equivalence. Therefore, in order to show $\psi$ is an equivalence,
	it suffices to show $\psi$
	is an equivalence on homology.
	By the group completion theorem, 
	\begin{align}
		\label{equation:group-completion-to-inversion}
		H_i(\Omega B\Conf^{c};\mathbb Z) \cong
		H_i(\Conf^c;\mathbb Z)[\alpha_c^{-1}].
	\end{align}
	Since $H_i(\Conf^c_n;\mathbb Z) \cong \oplus_{\pi_0
		\Hur^c_n} H_i(\conf_n;\mathbb Z)$,
	taking the sum over $n$ and inverting $\alpha_c$, 
	and using that $\pi_0
	\Hur^c[(\pi_0\Hur^c)^{-1}] \cong \Omega B \pi_0\Conf^c$, we obtain
	\begin{align}
		\label{equation:inversion-to-sum}
		H_i(\Conf^c;\mathbb Z)[\alpha_c^{-1}] \cong \oplus_{\Omega
			B\pi_0\Conf^c} \colim_n H_i(\conf_n;\mathbb Z).
	\end{align}
	Note here that the colimit in $n$ is given by the stabilization map
	$\alpha:\Conf_n \to \Conf_{n+1}$, since each $\alpha_x$ for $x \in c$ acts
	through the same map, which can be identified with $\alpha$ on each component
	of $\Conf^c$.
	Applying the group completion theorem again to $\Conf$, we learn that
	\begin{align}
		\label{equation:sum-to-completion}
		H_i(\Omega B\Conf) \cong 	H_i(\Conf;\ZZ)[\alpha^{-1}] \cong \oplus_{\Omega B \pi_0 \Conf}\colim_n H_i(\conf_n;\mathbb Z).
	\end{align}
	so combining
	\eqref{equation:group-completion-to-inversion},\eqref{equation:inversion-to-sum},
	and \eqref{equation:sum-to-completion}, we find
	\begin{align*}
		H_i(\Omega B\Conf^c;\ZZ) \cong \oplus_{\Omega B\pi_0\Conf^c} \colim_n H_i(\conf_n;\mathbb Z) 
		\cong 
		H_i(\Omega B\Conf \times_{\Omega B\pi_0\Conf}\Omega
		B\pi_0\Conf^{c};\mathbb Z).
	\end{align*}
	Hence, $\psi$ is an equivalence.
\end{proof}

The lemma below is the key ingredient in the proof of the proposition. Recall that $c/c$ is the set of components of the rack $c$, i.e the orbits of the action of $c$ on itself. There is a natural surjective map of racks $c \to c/c$, where $c/c$ is given the trivial rack structure.
\begin{lemma}\label{lemma:rackhomologyequivalence}
	For any finite rack $c$, the map 
	$$B\Hur^c \to B\Hur^{c/c}$$
	induced by $c \to c/c$ is a $\ZZ[\frac 1 {|G^0_c|}]$-homology equivalence.
\end{lemma}

\begin{proof}
		 This lemma follows as in the proof of \cite[Corollary
		 5.4]{randal-williams:homology-of-hurwitz-spaces}. Using
		 \autoref{proposition:modelquotdescription} for $X = c$, $M=N=*_+$,
		 we obtain an explicit topological model for $B\Hur^c$. By
		 filtering this space by the number of labels on the interior,
		 we get a filtration such that the associated cohomology
		 spectral sequence is the cochain complex computing rack
		 cohomology of $c$, see \cite{etingof2003rack}. It follows from \cite[Theorem 4.2]{etingof2003rack} that the map $c \to *$ induces an isomorphism on rack cohomology with $\ZZ[\frac 1 {G^0_c}]$-coefficients, which completes the proof.
\end{proof}

\begin{proof}[Proof of \autoref{proposition:rackhomology}]
	By the group completion theorem, $v[\alpha_{c}^{-1}]$ can be identified with
	the map $\Omega B\Hur^{c} \to \Omega B\Conf^{c}$ on homology. It is then enough to show that the map $\Omega B\Hur^{c} \to \Omega B\Conf^{c}$ is a rational homology equivalence.

	We next wish to reduce to showing 
$B \Hur^c \to B \Conf^c$ is a
rational homology equivalence.
It is shown in \cite[Proposition 5.2]{fennRS:the-rack-space}
that for any rack $X$, $BX$ is weakly simple. By \autoref{lemma:targetgroupcompletion}, $B\Conf^c$ is also weakly simple, since the action of $\pi_1$ on higher homotopy groups factors through $\pi_1B\Conf$ acting on the higher homotopy groups of $B\Conf$, and $B\Conf$ is the classifying space of a rack.

Applying \autoref{lemma:simplespacesloopspace} to the map
$B \Hur^c \to B \Conf^c$, we learn that if it is a
a rational homology equivalence, then
$\Omega B\Hur^{c} \to \Omega B\Conf^{c}$
is also a rational homology equivalence.
Here, we are also using that
the map $\pi_1 B \Hur^c \to \pi_1 B \Conf^c$ is an equivalence by construction of $\Conf^c$. 
	
We next claim that $B\Conf^{c} \to B\Conf$ is an equivalence on rational
homology.
	The map $\Omega B\pi_0\Hur^{c} \to \Omega B\pi_0\Conf$ is surjective
	with kernel a group we call $G$. 
	Here, $G$ is finite because of the fact that $|\pi_0\Conf^c_n| = |\pi_0\Hur^c_n|$
	stabilizes as $n$ grows, which is a consequence of
	\autoref{lemma:evwlemma} and the fact that a filtered colimit of finite
	sets along surjections is finite. It follows from this and the fact that $\psi$ from
	\eqref{equation:psi-definition} is an equivalence that the fiber of the map $B\Conf^{c}
	\to B\Conf$ is $BG$. 
	We then obtain our claim that $B\Conf^{c} \to B\Conf$ is an equivalence on rational homology
	using the Serre spectral sequence for this
	fiber sequence, since the rational homology of $G$ is just $\QQ$ in
	degree $0$ with trivial action of $B\Conf$.
	
	It remains then to show that the map $B\Hur^c \to B\Conf$ is an equivalence on rational homology, but this follows from applying \autoref{lemma:rackhomologyequivalence}, since $c/c=*$ since $c$ is connected, and $\Hur^*$ is $\Conf$.
\end{proof}

\begin{remark}
	In fact, the proof above (as well as this whole section) doesn't just work rationally, but also after inverting the order of $G$ showing up in the proof and the order of the `reduced structure group' of the rack $c$ \cite[Section 2]{etingof2003rack}.
	In the case $c$ comes from a conjugacy class in a group $\langle c
	\rangle$, 
	all primes dividing $|G|$ appearing in the proof (which is potentially
		different from $\langle c \rangle$ and all primes dividing the reduced
	structure group of the rack are contained in the set of primes dividing
	$\langle c\rangle$, by \cite[Theorem 2.5]{wood:an-algebraic-lifting-invariant}. So, we can compute stable homology after one inverts the order
	of $\langle c \rangle$.
\end{remark}

We continue to use notation from \autoref{notation:hurwitz-union},
\autoref{construction:homologystructure}, and \autoref{notation:rack}.
\begin{notation}
	\label{notation:hur-c-to-c-prime}
	For $c' \subset c$ a subrack,
there is a map of $\EE_1$-algebras $\overline{i}_{c'}^{c}:\Conf^{c'} \to \Conf^{c}$ that is
an inclusion of components.
Similarly, there is a map $\tilde{i}_{c'}^{c}:\Hur^{c'} \to \Hur^{c}$.
We have a section of the induced map $\Conf^{c'}_+ \to \Conf^{c}_+$ of
pointed $\EE_1$-algebras $\overline{r}_{c}^{c'}:\Conf^{c}_+\to \Conf^{c'}_+$ 
obtained by sending all components not in image of $\Conf^{c}$ to the base
point. Similarly there is a map $\widetilde{r}_{c}^{c'}:\Hur^{c}_+\to \Hur^{c'}_+$ obtained the same way.
Let $r_{c}^{c'}$ denote the associated $\EE_1$-algebra map $A \to C_*(\Hur^{c'})$ obtained by taking reduced chains.
In particular, this map of algebras sends all generators $\alpha_x$, for $x \in
c-c'$, to $0$.
\end{notation}

\autoref{proposition:propsubgp} is the key to proving the case when $c'$ is a proper
subrack of $c$. To state it, we need to introduce the following condition on the subrack:

\begin{definition}\label{defn:self-normalizing}
	For $c$ a rack and $c' \subset c$ a subset, we define the
	{\em normalizer} of $c$ in $c'$ to be
	$N_c(c') := \{x \in c | x\triangleright c' = c'\}$.
	We say that a subrack $c' \subset c$ is \textit{self-normalizing} if $c'
	= N_c(c')$.
\end{definition}

\begin{proposition}\label{proposition:propsubgp}
	Let $c$ be a rack.
	With notation as in \autoref{notation:hur-c-to-c-prime},
	let $c'$ be a self-normalizing subrack of $c$. Then the
	map $r_c^{c'}$ induces an equivalence
	$A[\alpha_{c'}^{-1}]^{\wedge}_{\alpha_{c-c'}} \to
C_*(\Hur^{c'};\QQ)[\alpha_{c'}^{-1}]$.
\end{proposition}
\begin{proof}
	We first note that since $\alpha_{c-c'}$ acts by $0$ on
	$C_*(\Hur^{c'};\QQ)[\alpha_{c'}^{-1}]$, 
	completion doesn't change it, so the map
	$A[\alpha_{c'}^{-1}]^{\wedge}_{\alpha_{c-c'}} \to
	C_*(\Hur^{c'};\QQ)[\alpha_{c'}^{-1}]$
in the proposition statement is obtained by taking $r_c^{c'}$, localizing at
$[\alpha_{c'}^{-1}]$ and completing at $\alpha_{c-c'}$.
	
	We will prove the proposition by applying \autoref{proposition:rigidity} 
	(rigidity of hom. epis)
	to the triangle
	\begin{center}
	\begin{tikzcd}
		A\ar[rr]\ar[dr] & &  C_*(\Hur^{c'};\QQ)[\alpha_{c'}^{-1}]\\
		& \ar[ur,"h"] A[\alpha_{c'}^{-1}]^{\wedge}_{\alpha_{c-c'}}.& 
	\end{tikzcd}
\end{center}
	The diagram is a triangle of $\EE_1$-algebras since the completion on the source is obtained by inverting a central idempotent by \autoref{lemma:completionhomepihur}.

	Note that all of the objects in the triangle are connective. By
	\autoref{lemma:completionhomepi},
	$\pi_0A[\alpha_{c'}^{-1}]^{\wedge}_{\alpha_{c-c'}}$ is the quotient of
	$\pi_0A[\alpha_{c'}^{-1}]^{\wedge}$ by the two sided
	ideal generated by $\alpha_{c-c'}$. 
	This ideal is exactly the classes coming from components not in $\Hur^{c'}$. Thus the map 
$h$
	is an equivalence on $\pi_0$.
	Moreover, $A \to A[\alpha_{c'}^{-1}]^{\wedge}_{\alpha_{c-c'}}$ is a hom. epi. by
	\autoref{lemma:completionhomepihur} and
	\autoref{example:localization-hom-epi}.

	It remains to see that the map $A \to
	C_*(\Hur^{c'};\QQ)[\alpha_{c'}^{-1}]$ is a hom. epi. Let $X =
	\pi_0\Hur^{c'}[\alpha_{c'}^{-1}]$, so that $C_*(X;\QQ) \cong
	\pi_0(C_*(\Hur^{c'};\QQ)[\alpha_{c'}^{-1}])$.
		By \autoref{lemma:homepipi0}, it is enough to check that the map 
		\begin{align*}
	C_*(X;\QQ)\otimes_AC_*(X;\QQ)
	\rightarrow C_*(X;\QQ)\otimes_{C_*(\Hur^{c'};\QQ)[\alpha_{c'}^{-1}]}C_*(X;\QQ)	
		\end{align*}
	 is an equivalence. Since $C_*(\Hur^{c'};\QQ) \to
	 C_*(\Hur^{c'};\QQ)[\alpha_{c'}^{-1}]$ is a hom. epi, we may replace the target
	 with  $C_*(X;\QQ)\otimes_{C_*(\Hur^{c'};\QQ)}C_*(X;\QQ)$. Then this
	 comparison is now induced from taking reduced chains of a map at the
 level of spaces. In what follows, the tensor product is the two sided bar
 construction of pointed spaces and we use notation from
 \autoref{subsubsection:spaces}:
	\begin{align}
		\label{equation:hom-epi-spaces}
	 X_+\otimes_{\Hur^{c}_+}X_+ \to X_+\otimes_{\Hur^{c'}_+}X_+.
	\end{align}
	So, it is enough to verify the map \eqref{equation:hom-epi-spaces} is a
	homology
	equivalence, which follows from
	\autoref{proposition:tensorproducthomotopy} below,
	using that $c'$ is self-normalization so $N_c(c') = c'$.
\end{proof}

\begin{figure}
	\includegraphics[scale=.6]{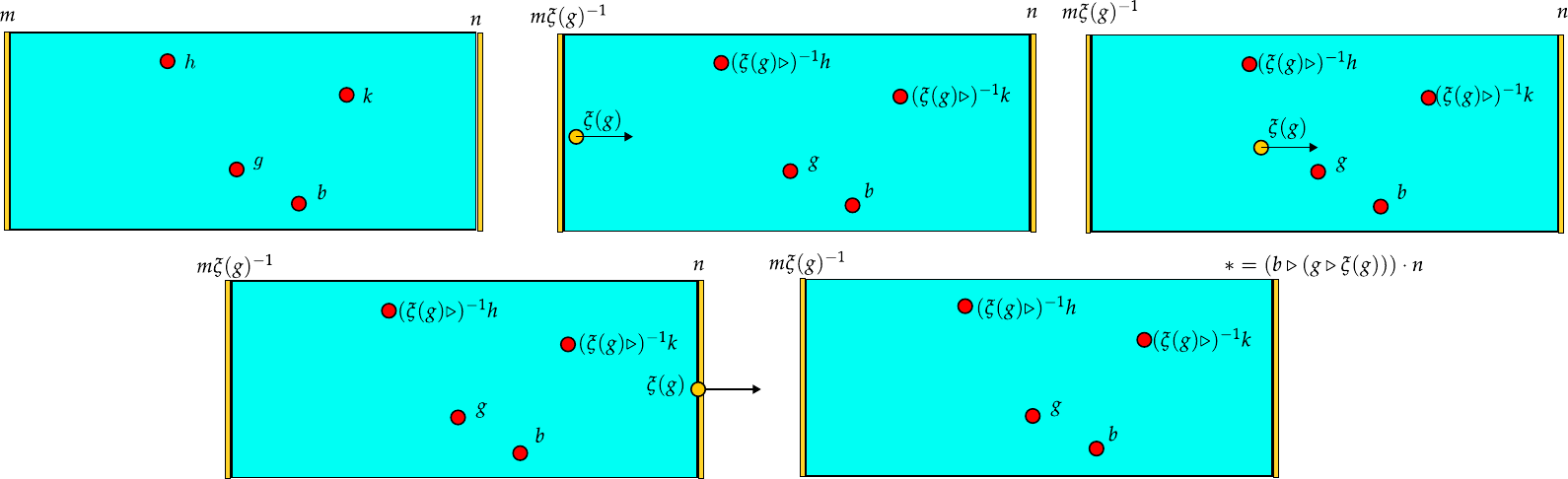}
	\caption{This is a pictorial description of the homotopy from
		\autoref{proposition:tensorproducthomotopy}.
We take $b \in c'$ but $g \notin c'$.
We pass the element $\xi(g)$ from the left to the right. 
When it meets the right boundary, it acts on the boundary by
$b \triangleright (g\triangleright \xi(g)))$. 
By assumption $g\triangleright \xi(g) \notin c'$.
Since $b \in c'$, 
$b \triangleright (g\triangleright \xi(g)))$
is also not in $c'$.
Therefore, the right boundary becomes the base point, and this whole
configuration is then identified
with the base point.}
\label{figure:nullhomotopy}
\end{figure}

The final hurdle before proving our main result is the 
following key result, which was
needed for the proof of \autoref{proposition:propsubgp}.

\begin{proposition}\label{proposition:tensorproducthomotopy}
	Let $c$ be a rack, $c' \subset c$ be a subrack with normalizer $N_c(c')$,  
	and $X_+$ be $\pi_0\Hur^{c'}[\alpha_{c'}^{-1}]_+$, viewed as a $\Hur^{c'}_+$-bimodule. 
	Then the map
\begin{align}
	\label{equation:c-to-prime-equivalence}
	X_+\otimes_{\Hur^{c}_+}X_+ \to X_+\otimes_{\Hur^{N_c(c')}_+}X_+ 
\end{align}
	is a homology equivalence.
\end{proposition}
\begin{remark}
	\label{remark:}
	In fact, one can show that the map
	\eqref{equation:c-to-prime-equivalence} is not only a homology
	equivalence but even an equivalence, by verifying that the fundamental groupoids of the two spaces are isomorphic.
	However, we will not need this stronger statement, so we do not prove it.
\end{remark}
\begin{remark}
	\label{remark:}
	We note that the idea to define the filtration $F_\bullet$
in the following proof was communicated to us by Andrea Bianchi.
This allowed us to remove an additional condition 
we had previously imposed on the subrack $c'$, and thereby simplify
the following proof.
We would like to thank Andrea Bianchi for this very helpful observation.
\end{remark}
\begin{proof}
	The map of interest $X_+\otimes_{\Hur^{c}_+}X_+ \to
	X_+\otimes_{\Hur^{N_c(c')}_+}X_+ $ has a section induced from the inclusion
	of racks $N_c(c') \to c$. It thus suffices to show that this section induces
	a homology equivalence.
	
	We use the notation $\overline{Q}^*_\epsilon[M,\Hur_+^{N_c(c')},N]$, as	defined in \autoref{definition:quotientmodel} and \autoref{lemma:modelquotdescriptionpointed}, which we describe below.
	By \autoref{lemma:modelquotdescriptionpointed}, in order to prove this section is a
	homology equivalence it suffices to show the inclusion
	$$\iota_{\epsilon}:\overline{Q}^*_\epsilon[X_+,\Hur_+^{N_c(c')},X_+] \to
	\overline{Q}^*_\epsilon[X_+,\Hur_+^{c},X_+]$$ is an ind-weak homology equivalence (see \autoref{definition:indweakeq}) as $\epsilon$ approaches $0$ where $0<\epsilon<1$.
	
	We now recall the description of these spaces, which follows from \autoref{remark:eqrelation}.
	First, ${Q}_{\epsilon}[X_+,\Hur^{c}_+,X_+]$ is the space whose points
	consist of tuples
	$(m,n,x,l:x \to c)$ with the following properties:
	$m \in X_+, n \in X_+$, $x$ is a finite 
	subset $x$ of
	$[0,1]\times [\epsilon,1-\epsilon]$ such that if $\pi_2$ is the projection onto the
	second coordinate, then $\pi_2$ is injective on $x$, and the points in
	$\pi_2(x)$ are distance at least $\epsilon$ from each other,
	and $l:x \to c$ is a map of sets, which we colloquially refer to as a labeling function.
	
	Viewing $X_+$ as having a left and right action of $\Hur^c$, 	
	$\overline{Q}^*_{\epsilon}[X_+,\Hur^{c}_+,X_+]$
	is the quotient of 
${Q}_{\epsilon}[X_+,\Hur^{c}_+,X_+]$
	by the following three relations:
		\begin{itemize}
			\item[$(i)$] 			Suppose $x = \{p_1, \ldots, p_r,\ldots, p_k\}$ with $\pi_2(p_1) >
			\cdots > \pi_2(p_k)$. If $\pi_1(p_r) = 0$, so $p_r$ lies
			on the left boundary, the data ($m,n,x,l:x \to c)$ is
			equivalent to $(m', n, x' := \{p_1, \ldots, p_{r-1},  p_{r+1},
			\ldots, p_k\}, l': x'\to c)$, where $m' = ml(p_r)$,
			$l'(p_i) := l(p_r)\triangleright l(p_i)$ for $1
			\leq i < r$, $l'(p_i) := l(p_i)$ for $i > r$.

			\item[$(ii)$] Suppose $x = \{p_1, \ldots, p_r,\ldots, p_k\}$ with
			$\pi_2(p_1) >
			\cdots > \pi_2(p_k)$. If $\pi_1(p_r) = 1$, so $p_r$ lies
			on the right boundary, the data ($m,n,x,l:x \to c)$ is
			equivalent to $(m, n', x' := \{p_1, \ldots, p_{r-1},  p_{r+1},
			\ldots, p_k\}, l': x'\to c)$, where $n' = ((l(p_{r+1})\cdots l(p_{k}))\triangleright l(p_r))n$, $l'(p_i) := l(p_i)$ for $i \neq r$.
			\item[$(iii)$] All points where either $n$ or $m$ are the base point are identified to the base point.
		\end{itemize}
		
		The space $\overline{Q}^*_{\epsilon}[X_+,\Hur^{N_c(c')}_+,X_+]$ is
		defined similarly, except the the labelling function $l$ must
		land in $N_c(c')\subset c$.

	At this point, it may be useful to refer to the notion of two spaces being
	ind-weakly equivalent, defined in 
\autoref{definition:indweakeq}.
			Let $S^{c,c'}_{\epsilon}$ denote the quotient of the inclusion $\overline{Q}^*_{\epsilon}[X_+,\Hur^{N_c(c')}_+,X_+]\to  \overline{Q}^*_{\epsilon}[X_+,\Hur^{c}_+,X_+]$.

	To show that the
	$\iota_{\epsilon}$ is an ind-weak homology equivalence, using that the
	inclusion has the homotopy extension property,  (\autoref{lemma:cofibration}), it suffices to show the
	following:
	\begin{itemize}
		\item[$\star'$] 
			The space $S^{c,c'}_{\epsilon}$ is ind-weakly homology equivalent to a point.
	\end{itemize}

Note that any point in $S^{c,c'}_{\epsilon}$ that is not the base point can be represented by a point in $\overline{Q}^*_{\epsilon}[X_+,\Hur^{c}_+,X_+]$ where some $l(p_i)$ is not in $N_c(c')$.

It will be easier to prove the desired homology equivalence at the level of the
associated graded of a certain filtration, which we define next.
We define the filtration $F_\bullet S^{c,c'}_{\epsilon}$ on
$S^{c,c'}_{\epsilon}$,
where, for $j \geq 0$, $F_j S^{c,c'}_{\epsilon}$ 
is the subset of $S^{c,c'}_{\epsilon}$
consisting of those points where one of $l(p_{k-j}), \ldots, l(p_k)$ lies in
$N_c(c') - c$, together with the base point.
We set $F_{-1}S^{c,c'}_{\epsilon}$ to be the basepoint, and use $G_j S^{c,c'}_{\epsilon}$ to denote the associated graded of the filtration $F_j S^{c,c'}_{\epsilon}/ F_{j-1} S^{c,c'}_{\epsilon}$, obtained by identifying $F_{j-1} S^{c,c'}_{\epsilon}$ with the base point.

 It is easy to see that this is an exhaustive filtration by sub CW-complexes, so
 that it suffices to prove the following associated graded version of $\star'$:

\begin{itemize}
	\item[$\star$] For each $j\geq0$, $G_jS^{c,c'}_{\epsilon}$ is ind-weakly equivalent to a point.
\end{itemize}

	Fixing $j\geq0$ and $\epsilon>0,$ 
we will find a smaller $\epsilon'$ such that the map
$G_jS^{c,c'}_{\epsilon} \to G_jS^{c,c'}_{\epsilon'}$ is nullhomotopic.

We may choose a function
$\xi:c-N_c(c') \to c'$ such that 
$g\triangleright \xi(g) \notin c'$ for each $g \in c-N_c(c')$, since $g\triangleright c' \neq c'$
from the definition of the normalizer of $c'$.

We use ${Q}_{\epsilon}^{c,c'}$ as notation to shorten ${Q}_{\epsilon}[X_+,\Hur^{c}_+,X_+]$.
Let $\theta:{Q}_{\epsilon}^{c,c'}\to S_{\epsilon}^{c,c'}$ be the projection map.
We define a filtration on $Q^{c,c'}_{\epsilon}$ by $F_\bullet Q^{c,c'}_{\epsilon} :=
\theta^{-1}(F_\bullet S^{c,c'}_{\epsilon})$ and use $G_{\bullet}Q^{c,c'}_{\epsilon}$ to denote the associated graded.

Fix a point $y \in Q^{c,c'}_{\epsilon}$ given by the data \begin{itemize}
	\item[$(\dagger)$] $(m,n,x = \{p_1, \ldots, p_k\},l:x \to c)$ with $\pi_2(p_1)> \cdots > \pi_2(p_k)$. 
\end{itemize} 

Suppose $y$ as above lies in $F_jQ_{\epsilon}^{c,c'}$. Define $v$ to be the largest number
so that $l(p_v) \notin N_c(c')$. By definition of the filtration $F_j Q_{\epsilon}^{c,c'}$, we have $v\geq k-j$.
	Let $\epsilon' := \frac {\epsilon}{2}$.
	We now construct a continuous homotopy $H : F_jQ^{c,c'}_\epsilon  \times I \to
	G_j S^{c,c'}_{\epsilon'}$ 
	as follows:
	If either $m,n$ is the base point, the point $y$ is equivalent to the
	base point under relation $(iii)$, and we choose the constant homotopy
	at the base point. So now assume this is not the case. If $v>k-j$, then
	the point is in $F_{j-1}Q_{\epsilon}^{c,c'}$, so goes to the base point in the
	associated graded $G_j Q_{\epsilon}^{c,c'}$. Thus, we may choose the constant homotopy on such points as well. So now we may assume as well that $v=k-j$.
	
	Using relation $(i)$, we can
	change the label of $m$ to $m \xi(l(p_v))^{-1}$, add a point $s$ at $(0,
	\pi_2(p_v)+\frac {\epsilon} {2})$ labeled
	$\xi(l(p_v))$, 
	and replace the label $l(p_u)$ by $(\xi(l(p_v) \triangleright)^{-1}
		l(p_u)$ for $1 \leq u < v$
	to get an equivalent point of $G_jS_{\epsilon'}^{c,c'}$. 

	Now, define the homotopy $H(y,t)$ by continuously changing the location of the point $s$ to $(t,
	\pi_2(p_v)+\frac {\epsilon} {2})$ at time $t$.
	At $t=0$, $H$ is given by the composite $F_j Q_{\epsilon}^{c,c'}\rightarrow G_j
	S_{\epsilon}^{c,c'} \rightarrow G_j
	S_{\epsilon'}^{c,c'}$. Thus it will suffice to show:
\begin{itemize}
	\item[(1)] At $t=1$, $H$ is the constant map to the base point.
	\item[(2)] $H$ descends to a continuous map
		$G_j S^{c,c'}_\epsilon \times I \to
		G_j S^{c,c'}_{\epsilon'}$.
\end{itemize}
Indeed, if these are shown, then $H$ descends to a nullhomotopy of the inclusion
$G_j S^{c,c'}_{\epsilon} \to G_j
S^{c,c'}_{\epsilon'}$. 

We first check $(1)$ 
straightforwardly. For the point $y$, at $t=1$, the coordinates of $s$ are
$(1,\pi_2(p_j)+\frac {\epsilon } {2}))$, and the points with second
projection smaller than it in decreasing order are $p_v,\ldots, p_k$. Using
relation $(ii)$, 
(which we used to define $\overline{Q}^*_{\epsilon}[X_+,\Hur^{c}_+,X_+]$ as a quotient of
$Q_{\epsilon}[X_+,\Hur^{c}_+,X_+]$,) this is equivalent to a point where $s$ is removed, and $n$
changes to $n' :=[(l(p_v)\cdots l(p_k))\triangleright l(s)] \cdot n$. But
$l(p_v)\triangleright l(s) = l(p_v) \triangleright \xi(l(p_v)) \notin c'$ by
definition of $\xi$. Since the action of $l(p_u)$ for $u>v$ preserves the
property of not
being in $c'$, it follows that $((l(p_v)\cdots l(p_k))\triangleright l(s))
\notin c'$. Hence, $n'$ is the base point. Thus, by relation $(iii)$, this point
is the base point of $G_j S^{c,c'}_{\epsilon'}$, showing $(1)$.

	It remains to show $(2)$ for each $j\geq0$. It is easy to see that the equivalence relation induced by $(i),(ii),(iii)$ when restricted to the subspace $F_jQ_{\epsilon}^{c,c'}$ is the same as the one induced by the map $F_jQ_{\epsilon}^{c,c'} \to \overline{Q}^*_{\epsilon}[X_+,\Hur^{c}_+,X_+]$.
	It then suffices to show that $H$ is compatible with $(i),(ii),(iii)$ as well as the relation sending $F_{j-1}Q_{\epsilon}^{c,c'}$ to the base point.
	
	By construction,
	$H$ sends $F_{j-1}Q_{\epsilon}^{c,c'}$ to the constant homotopy at the base point, and is also compatible with the equivalence relation
	$(iii)$.

	We next check $H$ is compatible with the equivalence
	relation $(ii)$.
	Indeed, the value $l(s)$ remains unchanged under the
	identification from $(ii)$, unless $(ii)$ is applied to the point $p_v$. 
	In this case, the homotopy is the constant homotopy at the base point.
	Hence, the value of $l(s)$ is compatible with $(ii)$.

	It remains to check that $H$ is compatible with the equivalence relation $(i)$. 
	We first deal with the case that the point $p_r$ 
	being removed in $(i)$ has a label $l(p_r) \notin c'$. In this case the homotopy is the constant homotopy at the base point, 
	because for any time $t$, the point $p_r$ remains on the left hand side
	and one can use relations $(i)$ and $(iii)$ to identify it with the base
	point.
	Therefore, the nullhomotopy is compatible with
	the relation in this case.

	It only remains to consider then the case that
	the point $p_r$ being removed in $(i)$ has $l(p_r) \in c'$. 
	Note that we may assume $r \neq v$ because $l(p_v) \notin c'$.
	Because we are working on the subquotient
	$G_j S^{c,c'}_{\epsilon'}$, if we apply
	$(i)$ for some point $p_r$ with $r > v$, the point becomes identified
	with the base point, and so $H$ is the constant homotopy on the base point. (This is the reason for introducing the
	filtration $F_j$.)
	Hence, it only remains to deal with the case that $r < v$ at a point
	$y$ presented as $(m,n,x,l:x \to
	c)$, as in $(\dagger)$ above.

	We now verify relation 
	$H$ is compatible with the equivalence relation $(i)$ when $r < v$ and
	$l(p_r) \in c'$.
	If we first apply the nullhomotopy before applying $(i)$, $m$ changes to
	$m \cdot [\xi(l(p_v))^{-1}]$, we add a point $s$ at 
$(t, \pi_2(p_v) + \frac{\varepsilon}{2})$
	labeled by $\xi(l(p_v))$,
	and the points $p_u$ for $1 \leq u < v$ have labels which change from
	$l(p_u)$ to
	$(\xi(l(p_v)) \triangleright )^{-1} l(p_u)$.
	If we now apply $(i)$, $m$ then becomes $m \cdot [\xi(l(p_v))]^{-1} \cdot
	[(\xi(l(p_v))
	\triangleright )^{-1} l(p_r)]$, and 
	the label on $p_w$ for $1 \leq w < r$ becomes $((\xi(l(p_v))
	\triangleright )^{-1} l(p_r)) \triangleright ((\xi(l(p_v))
	\triangleright )^{-1} l(p_w))$.

	If instead we apply $(i)$ before applying the nullhomotopy, $m$ changes
	to $m\cdot [l(p_r)]$ and the points $p_w$ with $1 \leq w < r$ have label
	which changes from $l(p_w)$ to
	$l(p_r) \triangleright l(p_w)$.
	When we next apply the nullhomotopy, 
	we introduce a point $s$ at $(t, \pi_2(p_v) + \frac{\varepsilon}{2})$ labeled by $\xi(l(p_v))$,
	we change $m
	\cdot [l(p_r)]$ to $m \cdot [l(p_r)] \cdot [\xi(l(p_v))]^{-1}$,
	we change the label on 
$p_w$ with $1 \leq w < r$ from $l(p_r) \triangleright l(p_w)$ to
$(\xi(l(p_v)) \triangleright)^{-1} (l(p_r) \triangleright l(p_w))$,
and we change the label on $p_u$ with $r < u < v$ from $l(p_u)$ to
$(\xi(l(p_v)) \triangleright )^{-1} l(p_u)$.

	To conclude, we only need observe that the two points obtained above
	agree we need to verify the two relations
	\begin{align}
		\label{equation:m-relation}
		m \cdot [\xi(l(p_v))]^{-1} \cdot [(\xi(l(p_v))
		\triangleright )^{-1} l(p_r)] &=
		m \cdot [l(p_r)] \cdot [\xi(l(p_v))]^{-1} \\
		\label{equation:label-relation}
		((\xi(l(p_v))
	\triangleright )^{-1} l(p_r)) \triangleright ((\xi(l(p_v))
	\triangleright )^{-1} l(p_w))
	&= (\xi(l(p_v)) \triangleright)^{-1} (l(p_r) \triangleright
	l(p_w)).
	\end{align}
	First, recall the defining relation of Hurwitz space from
	\autoref{notation:hurwitz-union} implies that the positive generator of the
	braid group identifies $[g] \cdot [h] = [h] \cdot [h
	\triangleright g]$. 
	Therefore, taking $h = \xi(l(p_v))$, and using that 
	$[\xi(l(p_v))]$ acts invertibly on $M$, the action on $M$ factors
	through a group where the relation
	$[\xi(l(p_v))]^{-1} \cdot [g]  = [\xi(l(p_v))
	\triangleright g]\cdot [\xi(l(p_v))]^{-1}$
	is satisfied for each $g \in c$.
	Hence, \eqref{equation:m-relation} holds by taking
	$g = (\xi(l(p_v))\triangleright )^{-1} l(p_r)$.
Finally, \eqref{equation:label-relation} is equivalent to property $(2)$ in the
definition of rack \autoref{definition:rack}.
\end{proof}

Finally, we verify that in the non-splitting case, every non-empty subrack is self-normalizing:
\begin{lemma}\label{lemma:nullability}
	Let $c$ be a non-splitting rack, and $c'$ a nonempty subrack. Then $c'$ is self-normalizing.
\end{lemma}

\begin{proof}
	Let $S$ be the set of elements $y \in c$ such that $y\triangleright c'=
	c'$. Then, $S\subset c$ is a subrack of $c$ containing $c' \subset S$ as a
	subrack.
	Further, $c'$ is 
	closed under the operations $x\triangleright$ for each $x \in S$. 
	The non-splitting hypothesis applied to $S$ implies $S$ is connected,
	but we also find that $c' \subset S$ is a connected subrack. It follows
	that $S=c'$.
\end{proof}

\subsubsection{Proof of \autoref{theorem:mainhom}}
\label{subsubsection:mainhom}
	By \autoref{lemma:reductionstep}, it is enough to show that for each nonempty
	subrack $c' \subset c$, $v[\alpha_{c'}^{-1}]^{\wedge}_{\alpha_{c-c'}}$ is
	an equivalence. When $c' = c$, the result is
	\autoref{proposition:rackhomology}. When $c'\subset c$ is a proper
	subrack, we have a commutative square
	\begin{center}
	\begin{tikzcd}
		A[\alpha_{c'}^{-1}]^{\wedge}_{\alpha_{c-c'}} \ar[r]\ar[d]
	&C_*(\Hur^{c'};\QQ)[\alpha_{c'}^{-1}] \ar[d]\\
	A'[\alpha_{c'}^{-1}]^{\wedge}_{\alpha_{c-c'}}\ar[r] &
C_*(\Conf^{c'};\QQ)[\alpha_{c'}^{-1}].
	\end{tikzcd}
		\end{center}

	We wish to show the left vertical map is an equivalence, and so we will
	show the other three maps in the diagram are equivalences.
	The upper horizontal map is an equivalence by
	\autoref{proposition:propsubgp}, since the subrack is self-normalizing by \autoref{lemma:nullability}. The lower horizontal map is an
	equivalence because any homology class in $A'[\alpha_{c'}^{-1}]$ supported
	on a component that does not come from $\Conf^{c'}$
	is divisible by $\alpha_{c-c'}$, and the completion is projection onto
	the summand where $\alpha_{c-c'}$ acts by $0$ by \autoref{lemma:completionhomepi}.
	Finally, the right vertical map is an equivalence by
	\autoref{proposition:rackhomology}, where $c$ is replaced with $c'$.
	\qed

\section{Deducing the moments of the Cohen--Lenstra heuristics}
\label{section:cohen-lenstra}

In this section, we explain how we deduce 
\autoref{theorem:moments-roots-of-unity}, and hence
\autoref{theorem:moments}, from our computation of the stable homology of
Hurwitz spaces.
The structure of this section is as follows.
We first give some examples of nonsplitting conjugacy classes relevant to the
Cohen--Lenstra heuristics in \autoref{subsection:nonsplitting-examples}.
We then prove \autoref{lemma:formalization-count}, which gives a
criterion for computing the finite field point counts of a space in terms of its
stable homology.
In \autoref{proposition:hurwitz-space-point-count}, we apply this lemma to
Hurwitz spaces, and obtain a point count in terms of the number of geometrically
irreducible components of that Hurwitz space.
In \autoref{proposition:component-count}, we count the components of the
relevant Hurwitz spaces.
In \autoref{subsection:proof-moments}, we deduce
\autoref{theorem:moments} from the finite field point counts of these Hurwitz
spaces.
Finally, in \autoref{subsection:non-abelian},
we compute the moments predicted by the non-abelian Cohen--Lenstra heuristics
\autoref{theorem:non-abelian-moments-roots-of-unity}.

\subsection{Examples of non-splitting conjugacy classes}
\label{subsection:nonsplitting-examples}
An important example of a non-splitting $(G,c)$ for our applications to the
Cohen--Lenstra heuristics is the following.
\begin{example}
	\label{example:generalized-dihedral-nonsplitting}
	Let $H$ be a finite abelian group of odd order.
	Write $G = H \rtimes \mathbb Z/2
\mathbb Z$, with $H$ a finite abelian group of odd order and the nontrivial
element of $\mathbb Z/2 \mathbb Z$ acting on
$H$ via inversion.
	Let $c \subset G$ denote the conjugacy class of order $2$ elements, which is 
the preimage of the nontrivial element of $\mathbb Z/2 \mathbb Z$ under the
surjection $G \to G/H \simeq \mathbb Z/2 \mathbb Z$.
That such $(G,c)$ is non-splitting follows from Sylow's theorems, since elements
of order $2$ in any subgroup $H \subset G$ generate a Sylow subgroup, and hence are all conjugate.
This was also observed in \cite[Lemma 3.2]{EllenbergVW:cohenLenstra}.
\end{example}

The next example will not be used in this paper, but we mention it to give
another example relevant to a generalization of the Cohen--Lenstra heuristics.
\begin{example}
	\label{example:cubic-cohen-lenstra}
	Let $H$ be a finite abelian group of order prime to $6$ with an $S_3$
	action. Let $G := H
	\rtimes S_3$.
	Let $c \subset G$ denote the conjugacy class of elements which map to
	transpositions in $S_3$.
	A similar argument to
	\autoref{example:generalized-dihedral-nonsplitting} by considering
	$2$-Sylow subgroups shows that $(G,c)$ is again non-splitting.
\end{example}
\begin{remark}
	\label{remark:}
	The verification of \autoref{theorem:stable-homology-intro} for $(G,c)$
	as in \autoref{example:cubic-cohen-lenstra} 
	is closely connected to the Cohen--Lenstra--Martinet heuristics for
	torsion in class groups of cubic fields of
	squarefree discriminant, see
	\cite[\S6.1]{ellenbergVWhomologicalII}.
\end{remark}

\subsection{A preliminary lemma on point counting}
The next lemma is quite standard, and its proofs essentially appears in many places, such as 
\cite[Theorem 8.8]{EllenbergVW:cohenLenstra} and also
\cite[Theorem
9.2.4]{ellenbergL:homological-stability-for-generalized-hurwitz-spaces}.
The next lemma is a formal consequence of the
Grothendieck--Lefschetz trace formula and Deligne's bounds on eigenvalues of
Frobenius.

\begin{remark}
	\label{remark:}
	We prove the following lemma in the more general setting of stacks, but we will
only apply it in the context of schemes.
\end{remark}
In what follows, if $x$ is a complex number, we use $\| x \|$ to denote its
absolute value, while if $X$ is a set, we use $|X|$ to denote its cardinality.
For $\mathscr X$ a Deligne--Mumford stack, we follow the standard convention that 
$| \mathscr X(\mathbb F_q)|$ denotes $\sum_{x \in \mathscr X(\mathbb F_q)}
\frac{1}{|\aut(x)|}$, so points are weighted inversely proportionally to their
automorphism groups.
Of course, when $\mathscr X$ is a scheme, this agrees with the usual notion of
$|\mathscr X(\mathbb F_q)|$.

\begin{lemma}
	\label{lemma:formalization-count}
	Fix a prime power $q$ and a 
	sequence $\{Y_n\}_{n \in S}$ of nonempty smooth Deligne--Mumford
	stacks over $\spec \mathbb F_q$ with
	$Y_n$ of pure dimension $n$, where $n$ traverses over an infinite
	sequence of increasing integers $S \subset \mathbb Z_{\geq 0}$. Suppose that the following cohomological conditions are satisfied:
	\begin{enumerate}
		\item[(1) Stability:]
		There are real numbers $I$ and $J$, only depending on the sequence
		$\{Y_n\}_{n \in S}$ with the following property.
		Suppose there is some fixed $\ell$ and a collection of complex
		numbers
		$t_{i}$ so that
%
		whenever $n > I i + J$, 
		$\on{tr}(\frob_q^{-1} | H^i(Y_{n,\overline{\mathbb F}_q};
		\mathbb Q_{\ell})) = t_i$. 
				\item[(2) Exp. Bound:] There are constants $C,C'$ independent of $n$ so that for any $i$,
			$\dim H^i(Y_{n, \overline{\mathbb F}_q}; \mathbb
		Q_{\ell}) \leq C' C^{i}$.
			\end{enumerate}
	Then, if $q > C^2$,
	\begin{align*}
		\left \|  \frac{\left|Y_n (\mathbb F_q)\right|}{q^n} -
		\sum_{i=0}^\infty (-1)^i t_i 
		\right \| \leq 
		\frac{2C'}{1- \frac{C}{\sqrt{q}}}
		\left(\frac{C}{\sqrt{q}}\right)^{\frac{n-J}{I}}.
	\end{align*}
\end{lemma}
\begin{proof}
	To ease notation, let $X_n := Y_n \times_{\spec \mathbb F_q} \spec \overline {\mathbb
		F}_q$.
	Note that our stability assumption holds whenever $n > I i + J$, or
	equivalently when $i < \frac{n - J}{I}$.
	Using our stability assumption (1)
	and the Grothendieck--Lefschetz trace formula
	\cite[Theorem 3.1.2]{behrend:lefschetz-trace-formula},
	we have
	\begin{equation}
		\label{equation:manipulate-point-count}
		\begin{aligned}
			\frac{\left|Y_n (\mathbb F_q)\right|}{q^n} =& \sum_{i=0}^{2n}
			(-1)^i \tr \left(
			\frob_q^{-1} | H^i(X_n; \mathbb Q_{\ell}) \right)
			\\
			=& \sum_{0 \leq i < \frac{n-J}{I}}
			(-1)^i t_i 
			+\sum_{i \geq \frac{n-J}{I}} (-1)^i\tr \left(
				\frob_q^{-1} | H^i(X_n; \mathbb Q_{\ell}).
			\right)
		\end{aligned}
	\end{equation}
	Subtracting $\sum_{i=0}^\infty (-1)^i t_i$ 
	from both
	sides of \eqref{equation:manipulate-point-count} and taking absolute
	values, we obtain
	\begin{equation}
		\label{equation:difference-bound}
		\begin{aligned}
			&\left \|\frac{\left|Y_n (\mathbb F_q)\right|}{q^n} - \sum_{i=0}^\infty
			(-1)^i t_i 
			\right\| \\
			&\leq \sum_{i \geq \frac{n-J}{I}} \left\| t_i 
			\right\| + 
\sum_{i \geq \frac{n-J}{I}} (-1)^i \left \| \tr \left(
				\frob_q^{-1} | H^i(X_n; \mathbb Q_{\ell})
			\right) \right \|
		\end{aligned}
	\end{equation}
	
	We claim that $\sum_{i=0}^\infty (-1)^i t_i $
	converges absolutely as a
	function of $q$ for $\sqrt{q} > C$.
	Indeed, using Sun's generalization of Deligne's bounds for algebraic stacks
	\cite[Theorem 1.4]{sun:l-series-of-artin-stacks} and our exponential
	bound assumption $(2)$ we
	obtain 
	$t_i \leq \frac{C'C^{i}}{q^{i/2}}$.
	It follows that 
	\begin{align*}
		\sum_{i=0}^\infty \| (-1)^i t_i
		\| \leq C'
		\sum_{i=0}^\infty \left (\frac{C}{\sqrt{q}} \right)^i = C'
		\frac{1}{1-\frac{C}{\sqrt q}}.
	\end{align*}
	We similarly find that  
	\begin{align}
		\label{equation:trace-bound}
		\sum_{i= \frac{n-J}{I}}^\infty \| (-1)^i t_i
		\| \leq C'\sum_{i= \frac{n-J}{I}}^\infty
		\left(\frac{C}{\sqrt q} \right)^i = C' \left(\frac{C}{\sqrt
			q}\right)^{\frac{n-J}{I}} \frac{1}{1-\frac{C}{\sqrt q}}.
	\end{align}
	
	This bounds the first term on the right hand side of \eqref{equation:difference-bound}. It remains to bound the second
	term.
	Again, using Sun's generalizations of Deligne's bounds and our
	exponential bound assumption
	$(2)$, we similarly obtain
	$\|\tr (\frob_q^{-1} | H^i (X_n; \mathbb Q_{\ell})) \| \leq
	\frac{C'C^{2n-j}}{q^{i/2}}$.
	Hence, so long as $C < \sqrt{q}$, we can bound
	\begin{align*}
\sum_{i \geq \frac{n-J}{I}} (-1)^i \left \| \tr \left(
				\frob_q^{-1} | H^i(X_n; \mathbb Q_{\ell})
			\right) \right \|
		&\leq C' \sum_{i \geq \frac{n-J}{I}}
		\left(\frac{C}{\sqrt q}\right)^{i}
		\\
		&= C' \left(
		\frac{C}{\sqrt{q}}\right)^{\frac{n-J}{I}}
		\frac{1}{1- \frac{C}{\sqrt{q}}}.\qedhere
	\end{align*}
\end{proof}

\subsection{Applying the lemma to count points on Hurwitz spaces}

We next show that the conditions of \autoref{lemma:formalization-count} are
satisfied for Hurwitz spaces using our computation of the stable homology.
We
use \autoref{lemma:formalization-count} to obtain a count of the finite field
valued points of relevant Hurwitz spaces in terms of the number of their
geometrically irreducible components.

\begin{proposition}
	\label{proposition:hurwitz-space-point-count}
	Suppose $G$ is a finite group, $c \subset G$ is a conjugacy class closed
	under invertible powering, and
	$(G,c)$ is non-splitting.
	Let $q$ be an prime power satisfying $\gcd(q, |G|) = 1$.
If $Y_n \subset \cphur G n c {\mathbb F_q}$ is a union of connected components which has $r$
		geometrically irreducible components for each $n$ sufficiently
		large, 
		there exist constants $C, I, J$, all independent of
		$n$ and $q$, but possibly depending on $r$, so that, as long as
		$q > C^2$,
	\begin{align}
		\label{equation:subset-components-govern-count}
		\left \|  \frac{\left|Y_n(\mathbb
		F_q)\right|}{q^n} -
		r \left(1 -
		\frac{1}{q}\right)
		\right \| \leq 
		\frac{2C}{1- \frac{C}{\sqrt{q}}}
		\left(\frac{C}{\sqrt{q}}\right)^{\frac{n-J}{I}}.
	\end{align}
\end{proposition}

An important special case of the above is when we take 
$Y_n =\boundarycphur G n c {\mathbb F_q} g$ for some $g \in G$.
It may be useful to recall
the notation 
$\boundarycphur G n c B g$ for components of 
$\cphur G n c B$ whose boundary monodromy is $g$, as defined in
\autoref{definition:boundary-monodromy}

\begin{proof}
	This follows from \autoref{lemma:formalization-count}, taking the
	constant $C'$ there to be the same as $C$, once
	we verify its hypotheses.

	To simplify matters slightly, we will assume all irreducible components of $Y_n$
	are geometrically irreducible, as components which are not geometrically
	irreducible have no $\mathbb F_q$ points, by smoothness of the $Y_n$.
	Let $B$ denote the spectrum of a henselian dvr with residue field
	$\mathbb F_q$ and generic point of characteristic $0$. 
	In this case, there is a scheme $Z_n$ which is a union of $r$
	irreducible components of $\boundarycphur G n c B g$, so that $Y_n = Z_n
	\times_{B} \spec \mathbb F_q$.
	We let $Y_{n, \mathbb C} := Z_n \times_B \spec \mathbb C$ and
	$Y_{n, \overline{\mathbb F}_q} := Y_n \times_{\spec \mathbb F_q} \spec
 \overline{\mathbb F}_q.$

	We first verify the stability assumption from
	\autoref{lemma:formalization-count}(1).
	To do this, we start by constructing numbers $t_i$, which we will
	realize as the trace of $\frob_q^{-1}$ acting on certain vector spaces.
	For $B$ a scheme, we use $\conf_{n, B}$ denote the configuration space
	parameterizing $n$ unordered points on $\mathbb A^1_{B}$. (Previously,
	we had used $\conf_n$ as notation for this when $B = \spec \mathbb C$.)
	Choose $\ell$ prime to $|G|$ and $q$.
	Define $t_i := \tr(\on{Frob}_q^{-1}| H^i(\conf_{n,\overline{\mathbb F}_q};\mathbb
	Q_\ell))^{\oplus r}$ as a
	vector space with Frobenius action,
	where $r$ is the number of geometrically irreducible components of $Y_n$
for sufficiently
	large $n$.
	Note that $t_i$ is independent of $n$ and $q$ once $n > 2$ 
	and is given by the trace of the identity matrix of
	dimension $r$ if $i = 0$, the trace of $q$ times the identity matrix of dimension $r$
	when $i = 1$, and $0$ if $i > 1$.
	In particular, we have
	$\sum_{i=0}^\infty (-1)^i t_i
	= r(1 - 1/q)$.

	To verify 
	\autoref{lemma:formalization-count}(1),
	we next need to show the $t_i$
	agrees with the trace of Frobenius on $H^i(Y_{n, \overline{\mathbb
	F}_q}; \mathbb Q_\ell)$ for $n > i I + J$.
	There is a map 
	\begin{align}
		\label{equation:iso-c}
		H^i(\conf_{n,\mathbb C};\mathbb Q_\ell)^{\oplus r} \to H^i(Y_{n,
			\mathbb C};\mathbb Q_\ell),
	\end{align}
which is an isomorphism for
	$n> iI + J$
	by \autoref{theorem:stable-homology-intro}.
	
	Note that configuration space has a normal crossings compactification by
	\cite[Corollary
	B.1.4]{ellenbergL:homological-stability-for-generalized-hurwitz-spaces}.
	It follows from \cite[Proposition
	7.7]{EllenbergVW:cohenLenstra} 
	that for $\ell$ prime to $|G|$ and $q$ and $j \geq 1$
	that
	there is an isomorphism 
\begin{align}
	\label{equation:}
	H^i(Y_{n,\overline{\mathbb F}_q}, \mathbb
	Z/\ell^j \mathbb) \simeq
	H^i(Y_{n,\mathbb C}, \mathbb Z/\ell^j \mathbb Z).
	\end{align}
	Technically speaking, 
\cite[Proposition
7.7]{EllenbergVW:cohenLenstra} assumed $j = 1$, but the same proof works for $j
> 1$.
Hence taking the limit over $j$ and tensoring with $\mathbb Q_\ell$, we obtain
	\begin{align}
		\label{equation:full-iso-hur}
		H^i(Y_{n,\overline{\mathbb F}_q}; \mathbb
	Q_\ell) \simeq
	H^i(Y_{n,\mathbb C}; \mathbb Q_\ell)
	\end{align}
	when $\ell$ is prime to $|G|$ and $q$.
	
	Therefore, \eqref{equation:full-iso-hur}
	restricts to an isomorphism
	\begin{align}
		\label{equation:iso-hur}
		H^i(Y_{n,\overline{\mathbb F}_q}; \mathbb
	Q_\ell) \simeq
	H^i(Y_{n,\mathbb C}; \mathbb Q_\ell).
	\end{align}
	An analogous argument then shows 
	\begin{align}
		\label{equation:iso-conf}
	H^i(\conf_{n,\overline{\mathbb F}_q};\mathbb Q_\ell)^{\oplus r} \simeq
	H^i(\conf_{n,\mathbb C};\mathbb Q_\ell)^{\oplus r}.
	\end{align}
	Combining \eqref{equation:iso-c}, \eqref{equation:iso-hur}, and \eqref{equation:iso-conf}, we obtain an 
	isomorphism
$U_i : 
H^i(\conf_{n,{\overline{\mathbb F}_q}};\mathbb
Q_\ell)^{\oplus r} \to H^i(Y_{n,\overline{\mathbb F}_q} ;\mathbb Q_\ell)$.
This isomorphism $U_i$ is moreover Frobenius equivariant because it is induced by a map of
schemes $Y_n \to \coprod_{i=1}^r \conf_{n,\mathbb F_q}$ over $\mathbb F_q$.
	This implies 
	$t_i = \tr(\frob_q|H^i(Y_{n, \overline{\mathbb
	F}_q}; \mathbb Q_\ell))$.

	Finally, we verify the exponential bound assumption,
	\autoref{lemma:formalization-count}(2).
	This condition with $C = C'$
	follows from an argument similar to \cite[Proposition
	7.8]{EllenbergVW:cohenLenstra}.
	(It would exactly follow from \cite[Proposition
		7.8]{EllenbergVW:cohenLenstra}, except the statement of that
	result assumes $\ell > n$, which we don't want to assume. The reason
they have that restriction is that they argue by passing to a cover where they
order the $n$ branch points and so they need to invert $n!$. 
However, we can avoid passing to such a cover and directly compare the
cohomology between in characteristic $0$ and positive characteristic, so we do not need to invert $n!$.)
	Recall
	\eqref{equation:full-iso-hur} holds,
	as we showed above.
After identifying \'etale cohomology of $Y_{n, \mathbb C}$ with singular cohomology,
the result then follows from \cite[Proposition 2.5]{EllenbergVW:cohenLenstra}.

	We note that we can take the value of $C$ to be independent of $q$ and
	$\ell$ prime to $|G|$ and $q$
	satisfying our constraints because $\dim H^i(Y_n, \mathbb Q_\ell) \leq  \dim H^i(\boundarycphur G n c
	{\mathbb C}g; \mathbb Q_\ell)$ is independent of $q$ and $\ell$.
\end{proof}

\begin{remark}\label{remark:explicitconstant}
	The value of the constant $C^2$ in
	\autoref{proposition:hurwitz-space-point-count}
	(which is also the value of the constant $C$ in
		\autoref{theorem:moments} and
	\autoref{theorem:moments-roots-of-unity}) can be
	explicitly computed in terms of the point at which the homology of a certain
	Hurwitz space stabilizes.
	The computation of this constant does not use our work, but entirely follows
	from the proof of \cite{EllenbergVW:cohenLenstra}.
	For explicitly determining the constant, it was also useful for us to consult
	\cite{randal-williams:homology-of-hurwitz-spaces} and
	\cite{ellenbergL:homological-stability-for-generalized-hurwitz-spaces}.
	It may be possible to reduce the value of the constant $C^2$,
	so what we describe below
	is really an upper bound for $C^2$.
	
	Using \cite[Proposition 4.4, Proposition 8.1, Theorem 7.1, and Corollary
	7.4]{randal-williams:homology-of-hurwitz-spaces},
	we find that the stabilization operator $U$ 
	induces an isomorphism
	$H_d(
	\cphurc G n c) \to H_d(\cphurc G {n+\deg U} c)$ once 
	$n \geq (N_0+2)(2d+1) + N_0$; here $N_0$ is the 
	is the constant
	defined in \cite[Proposition 4.4]{randal-williams:homology-of-hurwitz-spaces},
	describing the $n$ after which the image of $U$ on
	$H_0(\on{Hur}_n^c)$ stabilizes.
	Then, tracing through the proof of 
	\cite[Theorem 4.1.3, Proposition
	4.3.3, Theorem 9.2.4]{ellenbergL:homological-stability-for-generalized-hurwitz-spaces},
	we find it is sufficient to take 
	$q > 4((2 |c|)^{3N_0 +\deg U}+1)^4$.
	
	For example, in the particular case that $c$ is the conjugacy class of
	transpositions in $S_3$, $|c| = 3$, we may take $\deg U = 2$, 
	and one may explicitly compute $N_0 = 5$, so we get
	that we can take $q > 4 \cdot (6^{17}+1)^4$ in
	\autoref{proposition:hurwitz-space-point-count} for $G=S_3$ and in 
	\autoref{theorem:moments-roots-of-unity} for $H=\ZZ/3\ZZ$.
\end{remark}

\autoref{proposition:hurwitz-space-point-count} lets us count the points
of relevant Hurwitz spaces once we compute their number of geometrically
irreducible components. We do so in the next proposition.
It will be useful to recall the notion of being split completely over $\infty$
from
\autoref{definition:split-completely}.

\begin{proposition}
	\label{proposition:component-count}
	Suppose $G, H$, and $c$ are as in
\autoref{example:generalized-dihedral-nonsplitting}.
	Let $q$ be an odd prime power with $\gcd(q,|H|) = 1$. Take $h := \gcd(|H|, q-1)$, and fix $g
	\in G$.
		For $n$ sufficiently large, there are $|\wedge^2 H[h]|$ many
	geometrically irreducible components 
	of $\boundarycphur G n c {\mathbb F_q} g$,
	parameterizing $H$-covers of hyperelliptic
	curves which are split completely over $\infty$. 
	There are constants constant $C, I, J$, independent of $n$ and $q$, and
	depending only on $H$, so that for $q$ satisfying $\gcd( |H|, q) = 1$,
	$q > C^2$, 
	and any $g \in G$,
	\begin{align}
		\label{equation:hurwitz-point-count-from-components}
		\left \|  \frac{\left|\boundarycphur G n c {\mathbb F_q} g(\mathbb
		F_q)\right|}{q^n} -
		|\wedge^2 H[h]| \left(1 -
		\frac{1}{q}\right)
		\right \| \leq 
		\frac{2C}{1- \frac{C}{\sqrt{q}}}
		\left(\frac{C}{\sqrt{q}}\right)^{\frac{n-J}{I}}.
	\end{align}
	\end{proposition}
\begin{proof}
Once we verify the claim about the number of geometrically irreducible
components,
\eqref{equation:hurwitz-point-count-from-components}
follows from
\autoref{proposition:hurwitz-space-point-count}.
The remainder of the proof is essentially explained in 
\cite[Theorem
3.1]{sawinW:conjectures-for-distributions-containing-roots-of-unity},
and we now spell out the details.
In the case that $n$ is even, 
\cite[Lemma 10.3]{liuWZB:a-predicted-distribution} yields a
bijection between components of
$\boundarycphur G n c {\overline{\mathbb F}_q} g$ and components of
$\boundarycphur G n c {\mathbb C} g$.
Then, \cite[Theorem 12.1(2)]{liuWZB:a-predicted-distribution} and
\cite[Corollary 12.6]{liuWZB:a-predicted-distribution} yield that the
number of irreducible components of
$\boundarycphur G n c {\mathbb C} g$ can be identified bijectively with a torsor
for the set of elements of a certain finite group $H_2(G, c)$
which are equal to their $q$th power.
There is an identification $H_2(G, c) \simeq (\wedge^2 H)^{\mathbb Z/2 \mathbb
Z}$
as is explained in 
the course of the proof of 
\cite[Theorem
3.1]{sawinW:conjectures-for-distributions-containing-roots-of-unity}.
Moreover, in our case, $\mathbb Z/2 \mathbb Z$ acts on $H$ by inversion, and
hence it acts trivially on $\wedge^2 H$, and so 
$(\wedge^2 H)^{\mathbb Z/2 \mathbb Z} \simeq \wedge^2 H$.
In other words, the number of geometrically
irreducible components is precisely
the $q-1$ torsion subgroup of $\wedge^2 H$,
$\wedge^2 H[q-1]$. Since $h := \gcd(|H| , q-1)$, this is also equal to $\wedge^2
H[h]$, and hence there are $\left|\wedge^2 H[h]\right|$ geometrically irreducible components
of $\boundarycphur G n c {{\mathbb F}_q} g$.

We conclude by explaining why the above argument still works when $n$ is odd.
The setting of 
\cite{sawinW:conjectures-for-distributions-containing-roots-of-unity}
and
\cite{liuWZB:a-predicted-distribution} is
slightly different from ours because they make the additional assumption
that all $G$-covers are unramified over $\infty \in \mathbb P^1$.
However, with our definition of Hurwitz spaces,
\autoref{definition:pointed-hurwitz-space},
the straightforward generalizations of 
all of the results cited in the above paragraph easily carry through.
The main minor modification needed is that instead of working with marked
covers $X \to \mathbb P^1_S$ as in
\cite{wood:an-algebraic-lifting-invariant}, one should instead work
with marked covers $X \to \mathscr P_S$, where $\mathscr P_S$ is a
root stack with an order $2$ root over the $\infty$ section, $\infty: S \to
\mathbb P^1_S$.
Additionally, in \cite[Theorem 12.4]{liuWZB:a-predicted-distribution},
one no longer requires that the inertia elements $g_i$ need multiply to
$\id$, since the corresponding cover $X \to \mathbb P^1$ will be
branched at $\infty$.
With these minor modifications in place, the rest of the argument goes
through for $n$ odd and completes the proof.

At this point, the constant $C$ may depend on the values of $h$ and $g$, but
since there are at most $|G|^2$ many choices for the pair $(h,g)$, we can
replace $C$ with the constant from this finite set of values which maximizes the
right hand side of \eqref{equation:hurwitz-point-count-from-components}.
This concludes the proof.
\end{proof}

\subsection{Proof of \autoref{theorem:moments-roots-of-unity}}
\label{subsection:proof-moments}

We conclude by proving our main result, \autoref{theorem:moments-roots-of-unity}, which determines the $H$-moment of the
class group of quadratic extensions of $\mathbb F_q(t)$ for $q$ sufficiently
large relative to $H$.
It will be useful to recall our notation for the class group of a function field
extension of $\mathbb F_q(t)$ and for the set $\mathcal{MH}_{n,q}$ of monic
smooth hyperelliptic curves over $\mathbb F_q$,
defined prior to \autoref{theorem:moments}.
The proof of \autoref{theorem:moments-roots-of-unity} essentially follows immediately from
\autoref{proposition:component-count} and a little bit of class field theory.
In fact, we prove a slight refinement of \autoref{theorem:moments-roots-of-unity}
which moreover quantifies the error term in the asymptotic count of torsion in class groups
of quadratic fields.

\begin{theorem}
	\label{theorem:moments-with-error}
	Suppose $H$ is a finite abelian group of odd order. 
	Fix $q$ an odd prime power with 
	$\gcd( |H|, q) = 1$ 
	and let 
	$h := \gcd(|H|, q - 1)$.
	There is are integer constants $I$, $J$, and $C$ depending only on $H$ so that,
	if $q > C^2$,
	\begin{align}
		\label{equation:repeated-hurwitz-point-count-from-components}
		\left \|  \frac{\sum_{K \in \mathcal{MH}_{n,q}}
		\left|\surj(\cl(\mathscr O_K), H)\right|} {\sum_{K \in \mathcal{MH}_{n,q}} 1} -
		\alpha_n \left|\wedge^2 H[h] \right| 
		\right \| \leq 
		\left(\frac{1}{1-\frac{1}{q}} \right)
		\frac{2C}{1- \frac{C}{\sqrt{q}}}
		\left(\frac{C}{\sqrt{q}}\right)^{\frac{n-J}{I}},
	\end{align}
	where $\alpha_n = 1$ if $n$ is odd and $\alpha_n = \frac{1}{|H|}$ if $n$
	is even.
\end{theorem}
\begin{proof}
Note that the class group
$\cl(\mathscr O_K)$ is identified with the Galois group of the 
maximal abelian unramified extension
of $K$, completely split completely over $\infty \in \mathbb P^1_{\mathbb F_q}$, as follows from class field theory
\cite[15.6]{hayes:a-brief-introduction-to-drinfield-modules}.
(Also see the paragraph following 
\cite[15.6]{hayes:a-brief-introduction-to-drinfield-modules}
to show that what is called $H_A$ there
is in fact the maximal unramified abelian extension split completely over
$\infty$.)
Recall $G = H \rtimes (\mathbb Z/2 \mathbb Z)$.

We first deal with the case $n$ is even.
When $n$ is even, by \autoref{lemma:hyperelliptics-split-over-infinity},
every $K \in
\mathcal{MH}_{n,q}$ corresponds bijectively to a smooth proper curve $X$
split completely over $\infty$.
For $n$ even,
we can then identify $\sum_{K \in \mathcal{MH}_{n,q}} \left|\surj(\cl(\mathscr
O_K), H)\right|= \frac{ \left|\boundarycphur G n c {\mathbb F_q} {\id} (\mathbb F_q)
\right|}{\# H}$
by combining
\cite[Lemma 9.3 and Lemma 10.2]{liuWZB:a-predicted-distribution}.
Note that the denominator $\sum_{K \in \mathcal{MH}_{n,q}} 1 = q^n - q^{n-1}$
for $n > 1$,
as it is identified with the $\mathbb F_q$ points of unordered configuration
space of $n$ points in $\mathbb A^1_{\mathbb F_q}$.
The case that $n$ is even now follows from
\autoref{proposition:component-count} upon dividing both sides of
\eqref{equation:hurwitz-point-count-from-components} by $1
-\frac{1}{q}$.

It remains to deal with the case that $n$ is odd.
Recall notation for components with specified boundary monodromy from
\autoref{definition:boundary-monodromy}.
The case that $n$ is odd is similar to the even case, except that
we 
now claim
$\sum_{K \in \mathcal{MH}_{n,q}} \left|\surj(\cl(\mathscr O_K), H) \right|=
\frac{ \left|\cphur G n c {\mathbb F_q} (\mathbb F_q) \right|}{\# H},$
with no condition imposed on the boundary monodromy.
Indeed, any $G$
cover corresponding to a point of
$\cphur G n c {\mathbb F_q}$ has inertia type $c$ at $\infty$, and so the the $G$
extension corresponds to an $H$ extension $L/K$ followed by a quadratic
extension $K/\mathbb F_q(t)$, where the quadratic extension is ramified over $\infty$
and the $H$ extension is split completely over $\infty$.
Here we are using that the quadratic extension $K/\mathbb F_q(t)$ here is split
completely over
$\infty$ in the stacky sense of \autoref{definition:split-completely}, and we
are using \autoref{lemma:hyperelliptics-split-over-infinity} to identify 
$\mathcal {MH}_{n,q} \simeq \phur {\mathbb Z/2 \mathbb Z} n {\{1\}} {\mathbb
F_q} (\mathbb F_q)$.
Hence, via the identification 
$\phur {\mathbb Z/2 \mathbb Z} n {\{1\}} {\mathbb
F_q}$ with the configuration space of $n$ unordered points in $\mathbb
A^1_{\mathbb F_q}$, we have
$\sum_{K \in \mathcal{MH}_{n,q}} 1 = q^n - q^{n-1}$.
Altogether, there are $|H|$ possible values for the boundary
monodromy at $\infty$ when $n$ is odd, which is why the $H$-moment in
the case that $n$ is odd is $|H|$ times the moment in the case that $n$
is even. Thus, the result
again follows by summing the result of
\autoref{proposition:component-count} over those $|H|$ elements $g
\in c = G - H$.
\end{proof}

\subsection{Non-abelian Cohen--Lenstra moments}
\label{subsection:non-abelian}

We next prove a theorem implying
\autoref{theorem:non-abelian-moments-roots-of-unity}, the non-abelian
Cohen--Lenstra moments with roots of unity.

\begin{theorem}
	\label{theorem:non-abelian-moments-with-error}
	Using notation from \autoref{notation:non-abelian},
	suppose $H$ is an admissible group with a $\mathbb Z/2 \mathbb Z$ action.
	Fix $q$ an odd prime power with 
	$\gcd( |H|, q) = 1$.
Let $\delta: \hat{\mathbb Z}(1)_{(2q)'} \to H_2(H \rtimes \mathbb Z/2 \mathbb Z, \mathbb
Z)_{(2q)'}$ be a group homomorphism with $\on{ord}(\im \delta) \mid q-1.$
	There is are integer constants $I$, $J$, and $C$ depending only on $H$ so that,
	if $q > C^2$ and $n$ is even,
		\begin{equation}
	\begin{aligned}
		\label{equation:non-abelian-cl}
		&\left \|  \frac{\sum_{K \in \mathcal{MH}_{n,q}}
			\left| \{ \pi \in \surj_{\mathbb Z/2 \mathbb
			Z}\left(G^\sharp_\emptyset(K), H \right): \pi_* \circ
	\omega_{K^\sharp/K} = \delta \}\right|} {\sum_{K \in \mathcal{MH}_{n,q}} 1} -
	\frac{1}{[H: H^{\mathbb Z/2 \mathbb Z}]} \right \| \\
	&\leq 
		\left(\frac{1}{1-\frac{1}{q}} \right)
		\frac{2C}{1- \frac{C}{\sqrt{q}}}
		\left(\frac{C}{\sqrt{q}}\right)^{\frac{n-J}{I}},
	\end{aligned}
	\end{equation}
where $[H: H^{\mathbb Z/2 \mathbb Z}]$ denotes the index of the $\mathbb Z/2
\mathbb Z$ invariants of $H$ in $H$.
\end{theorem}
\begin{proof}
	Let $G := H \rtimes \mathbb Z/2 \mathbb Z$ and let 
	$c \subset G$ denote the conjugacy class of order $2$ elements.
	Note that $(G,c)$ is nonsplitting by Sylow's theorem, see
	\cite[Lemma 3.2]{EllenbergVW:cohenLenstra}.

	We claim there is a unique 
	geometrically irreducible component of
	$\boundarycphur G n c {\mathbb F_q} {\id}$,
	whose $\omega$-invariant is $\delta$ precisely when $\on{ord}(\im
	\delta) \mid q -1$ and $n$ is even, where the $\omega$-invariant of a
	component was defined (and shown to be well defined) in
	\cite[Proposition
	4.3]{liu:non-abelian-cohen-lenstra-in-the-presence-of-roots}.
	To prove this claim, we note that
	$\cphur {\mathbb Z/2 \mathbb Z} {n} {\iota} {\mathbb
	F_q}$ has a unique geometrically irreducible component, where $\iota$ is
	the nontrivial element of $\mathbb Z/2 \mathbb Z$.
	The claim then follows from \cite[Lemma
	4.4]{liu:non-abelian-cohen-lenstra-in-the-presence-of-roots},
\cite[Lemma 4.5]{liu:non-abelian-cohen-lenstra-in-the-presence-of-roots},
and \cite[Proposition 12.7]{liuWZB:a-predicted-distribution}.
Here, we are also using that there is a unique component of the Hurwitz scheme
parameterizing quadratic extensions of $\mathbb F_q(t)$
of reduced discriminant $q^n$ precisely when $n$ is even, as follows from Riemann Hurwitz, and
so the $M_a$ appearing in 
\cite[Lemma 4.4]{liu:non-abelian-cohen-lenstra-in-the-presence-of-roots}
can all be taken to be $2$ and the $E_a$ can be taken to be the singleton $0$,
using the equality in \cite[Lemma
4.5]{liu:non-abelian-cohen-lenstra-in-the-presence-of-roots}.

	The result now follows from
	\eqref{equation:subset-components-govern-count} with $r =1$
	and \cite[Lemma
	4.6]{liu:non-abelian-cohen-lenstra-in-the-presence-of-roots}, which
	accounts for the factor of $[H: H^{\mathbb Z/2 \mathbb Z}]$.
\end{proof}

\begin{remark}
	\label{remark:}
	We now explain why we make the restrictions we do in
	\autoref{theorem:non-abelian-moments-with-error}, even though
	conjectures on non-abelian Cohen--Lenstra heuristics have been posed more
	generally.
	Before explaining our restrictions, we first explain which aspects of these conjectures have greater
	generality.
	First,	
	\cite[Conjecture 5.1]{wood:nonabelian-cohen-lenstra-moments}
	is stated more generally for imaginary quadratic fields
	and for certain subgroups of $H \wr S_2$ instead of just for $H \rtimes S_2$.
	Moreover, 
	\cite[Conjecture 1.2]{liu:non-abelian-cohen-lenstra-in-the-presence-of-roots}
	(generalizing the closely related \cite[Conjecture
	1.3]{liuWZB:a-predicted-distribution})
	is stated more generally for $\Gamma$ extensions $K/\mathbb F_q(t)$,
	where $\Gamma$ is an arbitrary finite group,
	instead of just $\mathbb Z/2 \mathbb Z$ extensions.

	The reason we restrict to the case that $H$ has odd order and limit
	ourselves to counting Galois $H \rtimes \mathbb Z/2 \mathbb Z$ extensions of $\mathbb F_q(t)$
	is that in these cases, we know the conjugacy class of order $2$
	elements in $H \rtimes \mathbb Z/2 \mathbb Z$ is
	non-splitting. Hence, in this case, we are able to apply
	\autoref{theorem:stable-homology-intro}.

	We now explain why we restrict to the field extension $K/\mathbb F_q(t)$
	being a $\mathbb Z/2 \mathbb Z$ extension, while
	\cite[Conjecture
	1.2]{liu:non-abelian-cohen-lenstra-in-the-presence-of-roots} and
	\cite[Conjecture
	1.3]{liuWZB:a-predicted-distribution}
	are stated
	for $\Gamma$ extensions $K/ \mathbb F_q(t)$, for $\Gamma$ an
	arbitrary finite group.
	For $H$ a finite admissible $\Gamma$ group, let $G: = H \rtimes \Gamma$ and
	let $c \subset G$ be the set of elements of $G - \id$ which have the same order as
	their image in $\Gamma$.
	The reason that we restrict to $\Gamma = \mathbb Z/2 \mathbb Z$ is
	because this is the only case for
	which $(G,c)$ is non-splitting, and
	hence this is the only case to which
	\autoref{theorem:stable-homology-intro} will apply.

	Let us now explain why we must have $\Gamma = \mathbb Z/2 \mathbb Z$ for
	$(G,c)$ to be non-splitting.
	In order for $(G,c)$ 
	to be non-splitting, it
	must also be the case that $(\Gamma, \Gamma - \id)$ is non-splitting.
	In particular, $\Gamma - \id$ must consist of a single conjugacy class.
	If $\Gamma - \id$ is a single conjugacy class, its order must divide
	$|\Gamma|$, so $|\Gamma| - 1$ must divide $|\Gamma|$. Supposing $|\Gamma|
	> 1$, this can only happen when $|\Gamma| = 2$. 
\end{remark}
\begin{remark}
	\label{remark:imaginary-quadratic}
	One may wonder why we restrict to quadratic fields $K/ \mathbb
	F_q(t)$ unramified over infinity in \autoref{theorem:moments-with-error}, 
	as opposed to also dealing with the
	case that these extensions are ramified over infinity.
	It seems likely that the same statement could be made for ramified
	extensions, following the proof of
	\autoref{theorem:moments-with-error} in the case $n$ is odd.
	Another approach to the imaginary quadratic case could be by adapting
	the argument in the proof of
	\cite[Theorem 1.2]{wood:nonabelian-cohen-lenstra-moments}.
	We believe it would be interesting to address the imaginary quadratic case.
	This would also verify many additional cases of 
	\cite[Conjecture 5.1]{wood:nonabelian-cohen-lenstra-moments},	
	and also many cases of a function field version of
	\cite[Conjecture 1.6]{bostonW:non-abelian-cohen-lenstra-heuristics}.
\end{remark}

\appendix
\section{A topological model for the Hurwitz space bar construction}
\label{appendix:topological-model}

The goal of this appendix is to provide a good topological model for the two-sided bar construction
$M\otimes_{\Hur^{G,c}_+}N$, where $M$ is a pointed set with a right
$\Hur^{G,c}$-action and $N$ is a pointed set with a left $\Hur^{G,c}$ action. The end result that is used in the main body of the paper is \autoref{lemma:modelquotdescriptionpointed}.

\subsection{Background on Hurwitz spaces of braided sets}
A natural setting in which our constructions go through is that of braided sets,
sometimes also called ``sets equipped with an invertible solution of the
Yang--Baxter equation,'' see  \cite{drinfeld2006some,lu2000set,soloviev2000non}. We recall the definition:

\begin{definition}
	A {\em braided set} is a set $X$ equipped with an automorphism $\sigma:
	X\times X \to X\times X$ satisfying the braid relation $\sigma_{1}
	\sigma_{2}\sigma_{1} = \sigma_{2}\sigma_{1}\sigma_{2}$, where
	$\sigma_{i}$ is the automorphism of $X\times X\times X$ obtained by 
	applying $\sigma$ to factors $i$ and $i+1$ and the identity on the other factor.
\end{definition}

\begin{construction}[Hurwitz spaces for braided sets]\label{construction:hurwitzbraidedset}
	Let $X$ be a braided set. Then there is a left action of the braid group on
	$X^n$, where the $i$th standard generator $\sigma_i$ of the braid group acts by
	applying $\sigma$ to the $i$ and $i+1$ factors of $X^n$ and acts as the
	identity on all other factors. The collection $(X^n)_{n \in \mathbb
	Z_{\geq 0}}$ can be viewed as an $\EE_1$-algebra in the category of local system on $\Conf= \cup_{n\geq0}BB_n$, where the multiplication is given by concatenation. We define $\Hur^X$ to be the $\EE_1$-algebra in $\Spc$ given by the \'etale cover of $\Conf$ arising from this local system, and use $\Hur^X_n$ to denote the part lying over $\Conf_n = BB_n$. This construction is natural in maps of braided sets, and there is a natural $\EE_1$-algebra map $\Hur^{X} \to \Conf$. Explicitly, the underlying space of $\Hur^X_n$ is the homotopy quotient $(X^n)_{hB_n}$.
\end{construction}
Throughout the appendix, we use $X$ to refer to a braided set.

\begin{example}
	The reader should keep in mind the following example: If $G$ is a group and $c$ is a union of conjugacy classes, then $c$ forms a braided set, where the twist operation $\sigma$ sends $(x,y)$ to $(y,y^{-1}xy)$.
	Then the Hurwitz space $\cup_n \Hur^{G,c}_{n,\mathbb C}$ as defined in
	\autoref{subsubsection:hurwitz-quotient} agrees with $\Hur^{X}$ as defined in \autoref{construction:hurwitzbraidedset}.
	More generally, if $c$ is a rack, the twist operation on $\Hur^c$ is
	given by sending $(x,y) \mapsto (y, y \triangleright x)$.
\end{example}

The next remark will not be used in what follows, but it may help explain
why the category of braided sets is a natural setting to work in.
\begin{remark}
The category of $\EE_1$-algebras in local systems on $\Conf$ with values in sets
can be identified with the category of lax-monoidal functors from the free
$\EE_2$-monoidal category on a point into the category of sets. For a braided
set $X$, viewing $\Hur^X$ as an object in this category via \autoref{construction:hurwitzbraidedset}, gives an identification of the category of braided sets with the full subcategory consisting of monoidal functors.
\end{remark}

\subsection{A monoid model of Hurwitz spaces}

We begin by defining topological spaces that model $\Hur^{X}$ and $\Conf$:
First, we define a topological monoid, $\mathrm{conf}$, which is a model for the
$\EE_1$-algebra structure on $\Conf$.

\begin{notation}\label{definition:conftop}
	Define the topological space $\mathrm{conf}$
	as the space whose points are given by pairs $(t,x)$ where, $t \in
	\mathbb R_{\geq0}$, and $x$ is a (possibly empty) configuration of finitely many distinct unordered
	points in 
	$(0,t)\times (0,1)$, not intersecting $[0,\frac 1 2) \times [0,1]$ or
	$(t-\frac 1 2,t]\times [0,1]$. The multiplication is given by
	concatenation of configurations and adding the elements $t$. We use
	$\mathrm{conf}$ to denote this topological monoid, and $\mathrm{conf}_n$ to denote the component with configuration of $n$ points. There is a map of
	topological monoids $t:\mathrm{conf} \to \RR_{\geq0}$ given by $(t,x)
	\mapsto t$. There is also a subset $\mathrm{ord} \subset \mathrm{conf}$
	consisting of configurations such that the second coordinate of each
	point in $x$ is $\frac 1 2$. 	
	For each $n$, the intersection
	$\mathrm{ord} \cap \mathrm{conf}_n$ is contractible. We will use the
	inclusion $\mathrm{ord}_n :=\mathrm{ord}\cap \mathrm{conf}_n\subset \mathrm{conf}_n$ to view each component of $\mathrm{conf}$ as being equipped with a `base point', i.e a fixed contractible subspace. We also use $\mathrm{confbig}$ to denote the variant of this topological monoid where we remove the condition that configurations not intersect $[0,\frac 1 2)\times [0,1]$ and $(t-\frac 1 2,t]\times [0,1]$, and use $\mathrm{bigord}$ and $\mathrm{bigord}_n$ to denote the corresponding subspaces. 
	There is an inclusion of topological monoids $\mathrm{conf}\to \mathrm{confbig}$ which is clearly a homotopy equivalence.
\end{notation}

\begin{notation}\label{notation:braidgroup}
	We identify $\pi_1(\mathrm{conf}_n,\mathrm{ord}_n)$ with the braid group
	such that the elementary twist $\sigma_{i}$ (twisting strands $i$ and
	$i+1$) 
	gets sent to the homotopy class of paths from $\mathrm{ord}_n$ that swaps the $i$th and $i+1$th element of the configuration via rotating them clockwise around each other by $\pi$.
\end{notation}

\begin{convention}\label{convention:pi1}
	 By convention, the composition in the fundamental group/groupoid of a space is such that given two paths $\gamma$, $\gamma'$, we have $[\gamma][\gamma']$ is the class of the path which is first $\gamma'$ and then $\gamma$ (assuming they are composable). This convention is the one that is compatible with the orientation of the interval in that we view $\gamma'$ as a morphism from its starting point to its end point.
\end{convention}

We next define a topological monoid $\operatorname{hur}^{X}$ which models $\Hur^{X}$ as an $\EE_1$-algebra.
\begin{notation}\label{definition:hurtop}
	Since $\Hur^{X}$, viewed as an $\EE_1$-algebra, 
	is obtained as a multiplicative 
	finite cover of $\Conf$, we can
	construct a model
	$\operatorname{hur}^{X}$ for $\Hur^{X}$
	as the corresponding multiplicative finite cover of $\mathrm{conf}$ in
	topological spaces. Namely, we first use the base points 
	constructed in \autoref{definition:conftop}
	to define the monoid
	map $\NN \to \mathrm{conf}$, which allows us to identify
	the fundamental group of $\mathrm{conf}$ with $B_n$, using the
	convention of \autoref{notation:braidgroup}. Then,
	$\operatorname{hur}^X_n$, as a cover of $\mathrm{conf}_n$, can be constructed as
	the quotient of $X^n\times \widetilde{\mathrm{conf}}_n$ by the braid
	group $B_n$, where $\widetilde{\mathrm{conf}}_n$ is the universal cover
	of $\mathrm{conf}_n$;
	More explicitly, a point of
$\operatorname{hur}^X_n$
	is given by a $B_n$-equivalence class of data $(x,t,\gamma,\alpha_i)$, where
	$(x,t) \in \mathrm{conf}_n$, $\gamma$ is a homotopy class of paths from this
	point to $\mathrm{ord}$, and $\alpha_i$, for $1 \leq i \leq n$, is an
	$n$ tuple of points in $X$, where $n=|x|$. Here the braid group action
	is such that $j$th standard generator of $B_n$, $\sigma_j$, sends $(x,t,\gamma,\alpha_i)$ to $(x,t,\sigma_j\gamma,\sigma_j(\alpha_i))$ using \autoref{convention:pi1} and \autoref{notation:braidgroup} to get a left action of $B_n$ on such $\gamma$, and using the left action of the braid group on $X^n$ from \autoref{construction:hurwitzbraidedset}.
	
	To make $\operatorname{hur}^{X}:=\cup_{n\geq0}\operatorname{hur}^{X}_n$
	into a topological monoid, we use the multiplication of $\mathrm{conf}$, concatenate paths, and concatenate tuples of points in $X$.
	There is a map of topological monoids $\operatorname{hur}^{X} \to
	\mathrm{conf}$. Composing the time coordinate $t:
	\mathrm{conf}\to \mathbb R_{\geq 0}$ with this map, we obtain a
	composite
	map, (which we also call $t$,) $t:
	\operatorname{hur}^{X} \to \mathbb R_{\geq 0}$. This is a map of
	topological monoids. We also define a variant
	$\operatorname{hurbig}^{X}$ of $\operatorname{hur}^X$ by doing the same construction with $\mathrm{confbig}$ instead of $\mathrm{conf}$. 
	We note 
$\operatorname{hurbig}^{X}$
	is homotopy equivalent to $\operatorname{hur}^X$.
\end{notation}

The key construction we use for Hurwitz spaces for braided sets is the following
cutting construction.
The idea for this construction is that if we start with a configuration in
$[0,t'] \times [0,1]$ which never contains points in the vertical line $t
\times[0,1]$, we can cut it into two configurations by cutting along $t
\times [0,1]$.
\begin{construction}[Cutting]\label{construction:cutting}
	Let $t \in [0,\infty)$, and consider the open set $U(t)$ of $\operatorname{hurbig}^X$ consisting of  $(x,t',\gamma,\alpha_i)$ such that $t'>t$ and no points of the configuration $x$ lie on $t\times [0,1]$. Then there is a continuous map $U(t)\to \operatorname{hurbig}^X\times \operatorname{hurbig}^X$ which we call `cutting at $t$', given as follows: We first cut the configuration on $[0,t']$ into two configurations $x',x''$ on $[0,t]$ and $[0,t'-t]\cong [t,t']$. Choose a homotopy class of paths $\phi$ from $x$ to $\mathrm{ord}$ with a representative that doesn't cross $t\times [0,1]$. 
	Using the braid group action, the tuple $(x,t',\gamma,\alpha_i)$ is
	equivalent to $(x,t',\phi,\beta_i)$ for some $\beta_i$ obtained by
	acting on the tuple $(x_i)$ via the element of the braid group given by
	the class of $\phi \gamma^{-1}$. We can now similarly cut $\phi$ into
	homotopy classes of paths $\phi',\phi''$ from the configurations
	$x',x''$ to the respective base points, and cut the sequence $\beta_i$
	into $\alpha_j',\alpha_k''$ by unconcatenating it into two sequences corresponding to the number of points in the configuration in $[0,t]\times [0,1]$ and $[t,t']\times [0,1]$ respectively. 
	
	Then the cutting map sends $(x,t',\gamma,\alpha_i)$ to
	the pair of tuples
	$(x',t,\phi',\alpha'_j)$ and $(x'',t'-t,\phi'',\alpha''_k)$. Since the
	concatenation map $X^a\times X^{b} \to X^{a+b}$ is $B_a\times B_b$-equivariant, this doesn't depend on the choice of $\phi$.
\end{construction}

\subsection{A scanning map}
Next, we produce well behaved topological models for the source and target of
\autoref{proposition:tensorproducthomotopy}. The following definition will be
the key ingredient in constructing the models.
The reader may wish to consult \autoref{remark:b-hur-topology}
for help understanding \autoref{definition:twosidedbartop}.

\begin{notation}\label{definition:twosidedbartop}
	Let $M$ be a set with a right action of
$\Hur^{X}$ and $N$ be a set with a left action of $\Hur^X$.
	Consider the topological
	space $B[M,\Hur^{X},N]$
	consisting of points which are of the form
	\begin{align}
		\label{equation:b-point}
		(a,b, y)
	\end{align}
	where $a \in M, b \in N$ and 
	$y = (x,1,\gamma,x_i) \in \operatorname{hurbig}^X$, following notation
	from \autoref{definition:hurtop}. The topology on $B[M,\Hur^{X},N]$ has a basis given as follows. Consider the following data:
	
	\begin{enumerate}[(a)]
		\item Numbers $d_1,d_2\in (0,1)$ with $d_1 < d_2$.
		\item A finite collection of pairwise disjoint open balls
			$U_1,\ldots,U_n$ in contained in the interior of
		$[d_1,d_2]\times [0,1]$.
		\item A homotopy class of paths $\phi$ from the configuration of
			the centers of the
			balls $U_i$, viewed as an element of
			$\mathrm{confbig}_n$, to the base point.
		\item Elements $\alpha_1,\ldots,\alpha_n \in X$.
		\item Elements $a' \in M, b'\in N$.
	\end{enumerate}

	We next define subsets $\mathfrak B(d_1,d_2,U_i,\phi,\alpha_i,a',b')$
	which form a basis of the topology on $B[M,\Hur^{X},N]$. 
	A point of the form \eqref{equation:b-point} lies in
	$\mathfrak B(d_1,d_2,U_i,\phi,\alpha_i,a',b')$ 
	if the following conditions hold.
	\begin{enumerate}
		\item None of the points in $x$ lie in $[d_1,d_2]\times [0,1]- \cup_1^n U_i$, and there is a unique point from $x$ in each $U_i$.
		\item Cutting the element of $\operatorname{hurbig}^X$ twice to
			restrict it to the interval $[d_1,d_2]$ yields a point
			$y' \in \operatorname{hurbig}^X$ (see
			\autoref{construction:cutting}). Then, using the homotopy
			class of $\phi$, the corresponding tuple of elements in
			$X$ associated to $y'$ is $\alpha_1,\ldots,\alpha_n$.
		\item[(3)] Define $y_1 \in \operatorname{hurbig}^X$ to be the
		element of $\operatorname{hur}^{X}$ obtained
		by
		cutting and restricting to the interval $[0,d_1]$, and let $y_2$ be the element similarly obtained by using the interval $[d_2,1]$.
		We then require that $ay_1 =a'$ and $y_2b=b'$.
	\end{enumerate}
\end{notation}
\begin{remark}
	\label{remark:b-hur-topology}
	The topology of \autoref{definition:twosidedbartop} is such that points of
	the configuration $x$ are allowed to collide with the boundary
	$\{0,1\}\times [0,1]$, at which point they disappear and  multiply the
	elements $m,n$ using the left action of $\Hur^X$ on $M$ and the right
	action on $N$. Condition (3) guarantees that when a collection of points converges to the boundary, it multiplies $m$ or $n$ to the correct value.
\end{remark}

We next construct a homotopy equivalence showing that the topological space $B[M,\Hur^X,N]$ models the two-sided bar construction
$M\otimes_{\operatorname{hur}^{X}}N$, which we define next. Our proof is an adaptation of the proof of \cite[Theorem 5.2.3]{bergstromDPW:hyperelliptic-curves-the-scanning-map}. Indeed, the case $M=N=X=*$ corresponds to the result there.

\begin{notation}
	\label{definition:two-sided-bar}
	Let $H$ be a topological monoid
	and let $M$ be a right module for $H$ and let $N$ be
	a left module for $H$.
	View $\Delta^n$ as the collection of tuples $(y_0,y_1,\ldots,y_n)$
	with $\sum_{i=0}^n y_i=1$. 
	We will describe elements of
	$M \times N \times \Delta^n \times H^n$ via the notation
	\begin{align}
		\label{equation:bar-data}
		(a, b, (y_0,\ldots, y_n), (x_1, \ldots, x_n)).
	\end{align}
	We define the {\em two-sided bar construction}, 
	$M\otimes_{H}N$, to be the quotient of $\coprod_{n \geq 0} M \times N
	\times \Delta^n \times H^n$,
	by the following
	equivalence relations: 
	\begin{enumerate}
		\item $(a, b, (0, y_1, \ldots, y_n), (x_1, x_2,\ldots, x_n))
		\sim (ax_1, b, (y_1,\ldots, y_n), (x_2, \ldots,
		x_n))$,
		\item $(a, b, (y_0,\ldots, y_{n-1},0), (x_1, \ldots, x_{n-1},x_n)) \sim (a,
		x_n b, (y_0,\ldots, y_{n-1}), (x_1, \ldots,
		x_{n-1}))$,
		\item $(a, b, (y_0,\ldots, y_{i-1},0,y_{i+1},\ldots, y_n), (x_1,
		\ldots, x_{i-1},x_i,x_{i+1},\ldots, x_n))$
		
		$\sim
		(a, b, (y_0,\ldots,y_{i-1},y_{i+1},\ldots, y_n), (x_1,
		\ldots, x_{i-1},x_i x_{i+1}, \ldots, x_n))$,
		\item $(a, b, (y_0,\ldots, y_{i-1},y_i,y_{i+1},\ldots, y_n), (x_1,
		\ldots, x_{i-1},e,x_{i+1},\ldots, x_n))$
		
		$\sim
		(a, b, (y_0,\ldots,y_{i-1}+y_i,y_{i+1},\ldots, y_n), (x_1,
		\ldots, x_{i-1},x_{i+1}, \ldots, x_n))$,
		where $e$ denotes the identity of the monoid $H$.
	\end{enumerate}
\end{notation}

\begin{lemma}\label{lemma:maincomparison}
	Let $M$ be a set with a right action of $\operatorname{hur}^{X}$ and let
	$N$ be a set with a left action of $\operatorname{hur}^X$. Then there is a weak homotopy equivalence $\sigma: M\otimes_{\operatorname{hur}^{X}}N\to B[M,\Hur^{X},N]$, natural in $X,M,N$.
\end{lemma}
\begin{proof}
	Recall that $M\otimes_{\operatorname{hur}^{X}}N$ denotes the two-sided bar
	construction, as defined in \autoref{definition:two-sided-bar}.
	In particular, since $\operatorname{hur}^{X}$ models $\Hur^{X}$ as an $\EE_1$-algebra, $M
	\otimes_{\operatorname{hur}^{X}} N$ models the two-sided bar construction $M\otimes_{\Hur^{X}}N$.
	
	We produce a map $\sigma: M\otimes_{\operatorname{hur}^{X}}N \to B[M,\Hur^{X},N]$, as follows: 
	Given data as in \eqref{equation:bar-data} for $H = \operatorname{hur}^{X}$, and using notation from
	\autoref{definition:hurtop},
	let $t_i := t(x_i) \in \RR_{\geq 0}$ 
	denote the time associated to $x_i \in
	\operatorname{hur}^{X}$. 
	We multiply $x_1,\ldots,x_n$ to give an element $x \in
	\operatorname{hur}^{X}$ with $t := t(x) =\sum_{i=1}^n t_i$. First, extend the interval on which $x$ is defined from $[0,t]\times [0,1]$ to $[-\frac 1 2,t+\frac 1 2]\times [0,1]$.
	We then define $t' := \sum_{i=1}^n
	y_{i} \cdot (\sum_{j=1}^it_j)$. 
	At this point, it will be useful to recall the cutting construction from
	\autoref{construction:cutting}.
	Choose $\epsilon>0$ small enough so that there are no points of the
	configuration $x$ in $(t'-\frac 1 2,t'-\frac 1 2+\epsilon]\times [0,1]$ and $[t'+\frac 1 2-\epsilon,t'+\frac 1 2)\times [0,1]$. We can then cut $x$ into three pieces, $w,x'',z \in \operatorname{hurbig}^X$ coming from the intervals $[-\frac 1 2,t'-\frac 1 2+\epsilon]$, $[t'-\frac 1 2+\epsilon,t'+\frac 1 2-\epsilon], [t'+\frac 1 2-\epsilon,t+\frac 1 2]$.
 Extend $x''$ to be of length $1$ by extending the length of the interval on
 both ends by $\epsilon$, and let $x'$ denote the resulting element of $\operatorname{hurbig}^X$. Note that $x'$ doesn't depend on $\epsilon$, and neither does the class of $w,z$ in $\pi_0\Hur^X$.
	
	The then define the map $\sigma$ to sends data as in
	\eqref{equation:bar-data} to the point $(aw,zb,x')$,
	in $B[M,\Hur^{X},N]$, with notation as in \eqref{equation:b-point}. To
	see this is well defined, one must check that this glues along the
	identifications (1)-(4) in the two sided bar construction
	given in
	\autoref{definition:two-sided-bar}.
	We omit this verification.
	
	To verify that $\sigma$ is continuous, it suffices to do check
	continuity on $M\times N\times \Delta^n\times (\operatorname{hur}^X)^n$.
	We now carry out this straightforward verification.
	We choose a basic open set of the form $\mathfrak
	B(d_1,d_2,U_i,\phi,\alpha_i,a',b')$.
	We wish to verify its preimage under $\sigma$ is also open. If $(a, b,
	(y_0,\ldots, y_n), (x_1, \ldots, x_n))$ is a point in the preimage, then
	in a small neighborhood, $(1)$ and $(2)$ of
	\autoref{definition:twosidedbartop} are clearly still satisfied. The path
	components of $w$ and $z$ are also unchanged in a small neighborhood, so
	condition $(3)$ is satisfied as well.

	Having shown $\sigma$ is continuous, we next claim that $\sigma$ is surjective on path components. 
	Every point in $B[M,\Hur^{X},N]$ has a path to a point where the
	configuration on $[0,1]\times [0,1]$ is empty, by linearly pushing the
	configuration towards $0\times [0,1]$ on the boundary, and multiplying
	once points collide with the boundary. A point with an empty
	configuration is determined by $(a,b) \in M\times N$. This point is the
	image of the point $(a,b,(1),()) \in M\otimes_{
	\operatorname{hur}^{X}}N$, showing the claimed surjectivity on path components.
	We view these points as above with empty configuration as giving a copy of $M\times N \subset M\otimes_{\operatorname{hur}^{X}}N$. 

	Recall we are aiming to show $\sigma$ is a homotopy equivalence. To do
	this, we claim it suffices to show the following two statements.
	\begin{enumerate}
		\item[(i)] For every $v\geq1$, given a map of pairs
		$f:(D^{v},S^{v-1}) \to (B[M,\operatorname{hur}^{X},N],
		M\times N)$, there is a lift, up to homotopy, to a map of pairs $g:(D^{v},S^{v-1}) \to (M\otimes_{ \operatorname{hur}^{X}}N,M\times N)$
		such that $\sigma \circ g$ is homotopic to $f$.
		\item[(ii)] For every $v\geq 1$ and a map of pairs $g:(S^v,*)
			\to (M\otimes_{ \operatorname{hur}^{X}}N,M\times N)$
			such that $\sigma \circ g$ is nullhomotopic, then $g$ is nullhomotopic.
	\end{enumerate}	

	We now explain why $(i)$ and $(ii)$ above imply $\sigma$ is a homotopy
	equivalence. First, $(i)$ in the case $v=1$, shows in particular that the
	map $\sigma$ on $\pi_0$ is injective, since it shows that the equivalence relation describing $\pi_0$ as a quotient of $M\times N$ is the same for the source and target.
	Since we showed above that $\sigma$ is a surjection on $\pi_0$, it will
	follow that $\sigma$ is an isomorphism on $\pi_0$.
	Then $(i)$ also shows that on each component, the map on $\pi_v$ is
	surjective for each $v \geq 1$, and $(ii)$ shows $\pi_v$ is injective.

	We turn to proving $(i)$, and assume that we have a map $f$ as given in
	the statement of $(i)$.
	We set up notation to define the map $g$ described in $(i)$.
	Using compactness, we can find a finite open
	cover $D^v = \cup_{\lambda \in \Lambda}U_{\lambda}$ such that there is
	an $r>0$ and for each $\lambda \in \Lambda$, there is a $t_{\lambda} \in
	(0,1)$ such that for all $u \in U_{\lambda}$, the configuration associated
	to $f(u)$ has no points in $(t_{\lambda}-r,t_{\lambda}+r)\times [0,1]$.
	Choose a partition of unity $w_{\lambda}$ subordinate to the cover $U_{\lambda}$. Let
	$u$ be a point in $U_{\lambda_0}\cap\cdots\cap U_{\lambda_n}$,
	and not in $U_{\lambda}$ for $\lambda \in \Lambda - \{\lambda_0, \ldots,
	\lambda_n\}$.
	Note that for every point $u \in D^v$ satisfies the above for a unique
	choice of $\{\lambda_0, \ldots, \lambda_n\} \subset \Lambda$ 
	by taking $\lambda_0,\ldots, \lambda_n$ to be the subset of $\Lambda$
	such that $u\in U_{\lambda_i}$ for every such $i$.
	Assume that
	$t_{\lambda_0}<t_{\lambda_1}<\cdots<t_{\lambda_n}$ (which can always be
	achieved by decreasing $r$ and making all $t_{\lambda}$ distinct). Then
	we consider the $n$-tuple of elements $x_1,\ldots,x_n \in \operatorname{hur}^{X}$, where $x_i$ is obtained from the element of $\operatorname{hur}$ associated to $f(u)$ by cutting along 
	$t_{\lambda_{i-1}} \times[0,1]$ and $t_{\lambda_i} \times[0,1]$, 
	taking the configuration with points lying in
	$[t_{\lambda_{i-1}},t_{\lambda_{i}}] \times[0,1]$, and rescaling the
	intervals by $\frac 1 {2r}$ so that the time function $t$ applied to
	this new configuration is $(t_{\lambda_i}-t_{\lambda_{i-1}})\frac 1 {2r}$.
	Define $u'$ and $u'' \in \operatorname{hurbig}^{X}$
	to be the elements 
	obtained respectively by cutting to restrict to the intervals
	$[0,t_{\lambda_{0}}]$ and $[t_{\lambda_n},1]$ respectively. 
	We define the map $g:D^{v} \to M\otimes_{\operatorname{hur}^{X}}N$ by sending $u$ 
	to the tuple $$(au',u''b,(w_{\lambda_{0}(u)},\ldots,w_{\lambda_{n}(u)}),
	(x_1,\ldots,x_n)).$$

	We will show there is a homotopy between $\sigma \circ g$ and  $f$,
which will verify $(i)$.
Choose the open cover $\{U_\lambda\}_{\lambda \in \Lambda}$ of $D^v$ so that image of each component the boundary
	$S^{v-1}$ is contained in exactly one $U_{\lambda_i}$. 
	(When $v > 1$, the boundary is connected, so this is the same as saying
	the boundary $S^{v-1}$ is contained in exactly one
	$U_{\lambda_i}$.)
	Then, $g$ is a map of
	pairs because each component of the boundary is sent to a single point,
	as $M\times N$ is a discrete set.

	We now check continuity of $g$ at a point $u \in D^v$.
	Let $j,k$ be such that $w_{\lambda_i}(u)=0$ for $i <j$ and $i>k$ but not
	for $i=j$ or $i=k$. Any neighborhood of $g(u)$, contains an open set
	$W$ consisting of data of the form $(a',b',(y_0,\ldots,y_n),
	(z_1,\ldots,z_n))$ where $|y_i-w_{\lambda_i}(u)|<\epsilon$ for some
	$\epsilon>0$ small enough such that this guarantees that $y_j,y_k\neq0$,
	and such that 
	\begin{itemize}
		\item[($\star$)] $a'z_1\cdots z_{j-1}=au'x_1\cdots x_{j-1}$ and
		$z_{k+1}\cdots z_nb' = x_{k+1}\cdots x_nu''b$.
	\end{itemize}
	In order to show $g$ is continuous at $u$, it is then enough that for any such open set $W$, $g^{-1}(W)$ contains an open neighborhood 
	of $u$ in $D^v$.

	Choose a basic open set 
	$\mathfrak B(d_1,d_2,U_i,\phi,\alpha_i,a',b')$ containing $f(u)$ and a
	small open neighborhood $V$ containing $u$ with $V \subset
	f^{-1}(\mathfrak B(d_1,d_2,U_i,\phi,\alpha_i,a',b')) \cap 
	U_{\lambda_0}\cap \cdots \cap U_{\lambda_n}$.
	We may shrink 
	$\mathfrak B(d_1,d_2,U_i,\phi,\alpha_i,a',b')$,
	to assume $d_1\leq t_{\lambda_0},d_2\geq t_{\lambda_n}$,
	where, here, shrinking this neighborhood may involve increasing the
	number of $U_\lambda$ so that there is one such open containing each
	point of the configuration $x$, appearing in the data for $f(u)$, in
	$(d_1, d_2) \times[0,1]$. 
	By further shrinking $V$, we may assume the functions $w_\lambda$ for
	$\lambda \in \Lambda -\{\lambda_0, \ldots, \lambda_n\}$ vanish on $V$.
	Therefore, using identifications $(1)-(3)$ of the bar construction, for the purposes of proving continuity,
	we may assume our open cover consists solely of
	$U_{\lambda_0}, \ldots, U_{\lambda_n}$, i.e, we can ignore the other
	$U_\lambda$.
	At this point, condition 
(3) of \autoref{definition:twosidedbartop} enables us to further shrink $V$ so that
($\star$) holds for all $g(v)$ with $v \in V$.
We have therefore found an open set in $D^v$ containing $u$ and contained in
$g^{-1}(W)$ for $W$ an open set containing $g(u)$. This proves that $g$ is
continuous.
	
	For a point $z \in D^{v+1}$, $\sigma(g(z))$ is obtained from $f(z)$
	by zooming in on the configuration on the interval around $\sum_{i=0}^n w_{\lambda_i}t_{\lambda_i}$ of radius $r$, rescaling it to be in the interval $[0,1]$, and then multiplying the element of $\operatorname{hurbig}$ obtained by cutting to the left of this interval with $a$ and multiplying the element of $\operatorname{hurbig}$ obtained by cutting to the right with $b$. There is an evident continuous homotopy from $f$ to $\sigma\circ g$, given by linearly rescaling around this interval, and letting points act on the boundary as they reach it. This gives us a homotopy between $f$ and $\sigma \circ g$, which is a homotopy of pairs because of our assumption about $g$ applied to the boundary sphere (each component of which is sent to a single point). 
	Thus we have verified $(i)$.
	
	To prove $(ii)$, given a map of pairs $g:(S^v,*) \to
	(M\otimes_{\operatorname{hur}^{X}}N,M\times N)$ and a nullhomotopy of
	$\sigma \circ g$,
	the same construction used above for $(i)$ lifts the nullhomotopy of
	$\sigma \circ g$ to a pointed map $h:D^{v+1} \to
	M\otimes_{\operatorname{hur}^{X}}N$. It remains to show that $h|_{S^v}$
	is homotopic to $g$ as a pointed map. Here we follow the last paragraph
	of the proof of \cite[Theorem
	5.2.3]{bergstromDPW:hyperelliptic-curves-the-scanning-map}. Namely,
	given $u \in S^v$, suppose $u$ is in $U_{\lambda_0}\cap\cdots\cap
	U_{\lambda_n}$ and not in $U_{\lambda}$ for $\lambda \in
	\Lambda - \{\lambda_0, \ldots, \lambda_n\}$. Associated to $g(u)$ is a
	sequence of $\kappa$
	elements in $\operatorname{hur}^{X}$, for some integer $\kappa$. We can first linearly stretch
	these $\kappa$ elements each by a factor of $\frac 1 {2r}$, and then use
	the identification (3) of the 2-sided bar construction
	\autoref{definition:two-sided-bar}
to make $n$ cuts in
the places specified by $t_{\lambda_i}$ in $\sigma \circ g$. Let $\tilde{g}$ be
the function of $u$ obtained this way, so that $\tilde{g}$ is homotopic to $g$.
Then, $\tilde{g}$ agrees with $h|_{S^v}$ except for the weights defining a point in $\Delta^n$. Linearly changing these weights then gives a continuous homotopy between $\tilde{g}$ and $h|_{S^v}$. Thus $g$ is homotopic to $h|_{S^v}$.
\end{proof}
\subsection{A simple topological model}
We will carry out a scanning argument to obtain a description of the homotopy type of
$B[M,\Hur^{X},N]$ in terms of explicit quotient spaces, which we define next.
This identification is verified in
the unpointed case in
\autoref{lemma:scanningunpointed},
and in the pointed case in
\autoref{lemma:modelquotdescriptionpointed}.
In the next definition, we have two somewhat complicated relations $(i)$ and
$(ii)$.
We must impose these relations so that
the proof of
\autoref{proposition:modelquotdescription}
goes through, the key idea being depicted in
\autoref{figure:boundary}.

\begin{definition}\label{definition:quotientmodel}
	Let $M$ be a set with a right action of $\Hur^X$ and let $N$ be a set with a left action of $\Hur^{X}$. Let $0<\epsilon<1$. We define  $Q_{\epsilon}[M,\Hur^{X},N]$ 
	to be the set consisting of triples $(m,n,x,l: x \to X)$ with $m \in M,
	n \in N,$
	$x = \{p_1, \ldots, p_k\}$ a finite subset of $[0,1]\times [\epsilon,1-\epsilon]$
	(with first projection 
$\pi_1: [0,1] \times [0,1] \to [0,1]$ and
	second projection $\pi_2: [0,1] \times [0,1] \to [0,1]$)
	such that for each $i\neq j$, the distance between $\pi_2(p_i)$ and
	$\pi_2(p_j)$ is at least $\epsilon$, and $l:x \to X$ is a
	map of sets.
	
	We define $\overline{Q}_{\epsilon}[M,\Hur^{X},N]$ to be the quotient space of $Q_{\epsilon}[M,\Hur^{X},N]$
	under the equivalence relation generated by the following two
	types of equivalences.
	
	\begin{itemize}
		\item[$(i)$] Suppose $x = \{p_1, \ldots, p_r,\ldots, p_k\}$ with $\pi_2(p_1) >
			\cdots > \pi_2(p_k)$. If $\pi_1(p_r) = 0$, so $p_r$ lies
			on the left boundary, the data ($m,n,x,l:x \to X)$ is
			equivalent to $(m', n, x' := \{p_1, \ldots, p_{r-1},  p_{r+1},
			\ldots, p_k\}, l': x'\to X)$, defined as follows:
			write
			$(y_1, \ldots, y_r) :=
			\sigma_1\sigma_2\cdots\sigma_{r-1}(l(p_1),
			\ldots, l(p_r)).$
			Then $m' = my_1$, and $l'(p_i) := y_i$ for $1 \leq i < r$, $l'(p_i) := l(p_i)$ for $i > r$.
		\item[$(ii)$] Suppose $x = \{p_1, \ldots, p_r,\ldots, p_k\}$ with
			$\pi_2(p_1) >
			\cdots > \pi_2(p_k)$. If $\pi_1(p_r) = 1$, so $p_r$ lies
			on the right boundary, the data ($m,n,x,l:x \to X)$ is
			equivalent to $(m, n', x' := \{p_1, \ldots, p_{r-1},  p_{r+1},
			\ldots, p_k\}, l': x'\to X)$, defined as follows:
			write
			$(y_r, \ldots, y_k) :=
			\sigma_{k-r-1}\cdots \sigma_{1}(l(p_r),
			\ldots, l(p_k)).$
			Then $n' = y_kn$, and $l'(p_i) := l(p_i)$ for $i < r$,
			$l'(p_i) = y_{i-1}$ for $i> r$.
	\end{itemize}
\end{definition}

\begin{remark}\label{remark:eqrelation}
	In the case that the braided set $X$ comes from a rack $c$, the equivalences
	in \autoref{definition:quotientmodel} can be described more simply as
	follows.
	\begin{itemize}
		\item[$(i)$] 			Suppose $x = \{p_1, \ldots, p_r,\ldots, p_k\}$ with $\pi_2(p_1) >
		\cdots > \pi_2(p_k)$. If $\pi_1(p_r) = 0$, so $p_r$ lies
		on the left boundary, the data ($m,n,x,l:x \to X)$ is
		equivalent to $(m', n, x' := \{p_1, \ldots, p_{r-1},  p_{r+1},
		\ldots, p_k\}, l': x'\to X)$, where $m' = ml(p_r)$, $l'(p_i) :=
		l(p_r)\triangleright l(p_i)$ for $1 \leq i < r$, $l'(p_i) := l(p_i)$ for $i > r$.

\item[$(ii)$] Suppose $x = \{p_1, \ldots, p_r,\ldots, p_k\}$ with
$\pi_2(p_1) >
\cdots > \pi_2(p_k)$. If $\pi_1(p_r) = 1$, so $p_r$ lies
on the right boundary, the data ($m,n,x,l:x \to X)$ is
equivalent to $(m, n', x' := \{p_1, \ldots, p_{r-1},  p_{r+1},
\ldots, p_k\}, l': x'\to X)$, where $n' = ((l(p_{r+1})\cdots l(p_{k}))\triangleright l(p_r))n$, $l'(p_i) := l(p_i)$ for $i \neq r$.
	\end{itemize}
\end{remark}

A useful construction we will use to study the space
$B[M,\operatorname{hur}^{X},N]$ is a flow that pushes configurations outwards,
which we define next.
\begin{construction}\label{construction:flow}
	Choose a smooth function $f:[0,1] \to \RR$ such that $f([0,\frac 1 2))
	<0$ and $f(1-x)=-f(x)$. In particular, $f(1/2) = 0$. Consider the vector field $f(x)\frac{\partial}{\partial x}$ on $[0,1]\times [0,1]$, where $x$ is the first coordinate. We can then form a continuous flow $\phi_t:B[M,\Hur^{X},N]\times [0,\infty) \to B[M,\Hur^{X},N]$ by moving points in a configuration along the flow, and having them act on the elements $(m,n)$ in $M\times N$ when they reach the boundary. We can define such a flow (which we give the same name) on $\overline{Q}_{\epsilon}[M,\Hur^X,N]$ via the same procedure.
\end{construction}

The following proposition is key in relating
$\overline{Q}_{\epsilon}[M,\Hur^X,N]$ to $B[M,\Hur^X,N]$. This proposition characterizes the
image of the natural map from the first to the second. 

\begin{proposition}\label{proposition:modelquotdescription}
	Let $M$ be a set with a right action of $\Hur^X$ and $N$ be a set with a
	left action of $\Hur^{X}$. 
	Let
	$\pi_2: [0,1] \times[0,1] \to[0,1]$ be the second projection.
	There is a natural continuous injection $\overline{Q}_{\epsilon}[M,\Hur^X,N] \to B[M,\Hur^X,N]$ whose image consists of those points 
	$(a,b,y)$ where $y := (x := \{p_1, \ldots, p_k\}, 1,\gamma ,x_i)$ as in
	\autoref{definition:twosidedbartop} with the property that $\epsilon
	\leq \pi_i(p_j) \leq
	1-\epsilon$ for $1 \leq j \leq k$ and any two distinct
	elements in $\{\pi_2(p_1),
	\cdots, \pi_2(p_k)\}$ differ by at least $\epsilon$.

	Moreover, a continuous map $K \to B[M,\Hur^X,N]$ lifts to $\overline{Q}_{\epsilon}[M,\Hur^X,N]$ if its image in the image of finitely many connected components of $Q_{\epsilon}[M,\Hur^X,N]$.
\end{proposition}

\begin{figure}
	\includegraphics[scale=.3]{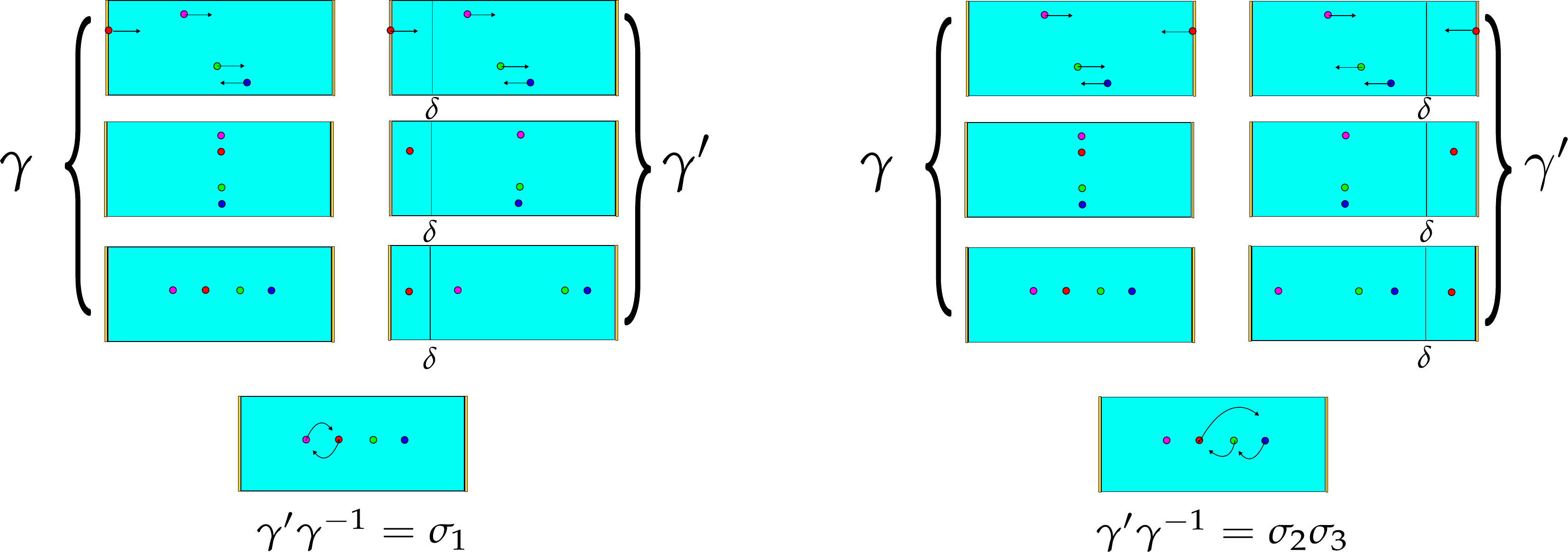}
\caption{This is a picture depicting the compatibility of the map $\tilde{g}$ in the
	proof of \autoref{proposition:modelquotdescription} with
	relations $(i)$ and $(ii)$ from \autoref{definition:quotientmodel}.
	The first two columns depict the paths $\gamma$ and $\gamma'$ the points
	take during the map $\tilde{g}$ under relation $(i)$, and the composite $\gamma'
	\gamma^{-1}$ is depicted underneath.
	The last two columns depict the paths $\gamma$ and $\gamma'$ which the
	points take under relation $(ii)$, and the composite $\gamma'
	\gamma^{-1}$ is again depicted underneath.
	In the diagrams, there are $4$ points, $p_1$ which is purple, $p_2$
	which is red, $p_3$ which is green, and $p_4$ which is blue.
	So $r = 2$ and $k = 4$ in the notation of
	\autoref{definition:quotientmodel}.
	The point $p_2$ is on the boundary. When $p_2$ is on the left boundary, the
	path is $\sigma_1 \cdots \sigma_{r-1} = \sigma_1$ since $r = 2$. 
	When $p_2$ is on the right boundary, the path is $\sigma_r \cdots
	\sigma_{k-1} = \sigma_2 \sigma_3$.
}
\label{figure:boundary}
\end{figure}

\begin{proof}
		To shorten notation, in this proof only, we will use $Q$ to denote
	$Q_{\epsilon}[M,\Hur^{X},N]$ and $\overline{Q}$ to denote $\overline{Q}_{\epsilon}[M,\Hur^{X},N]$.
	We will produce a continuous map 
	$$Q \to B[M,\Hur^{X},N]$$ which induces a map $\overline{Q} \to B[M,
	\Hur^X, N]$ satisfying the claimed properties.
	
	We define the interior $Q^{\circ}$ to be the dense subset of $Q$ on
	which the configurations are contained in $(0,1)\times [\epsilon,1-\epsilon]$. In other
	words, a point of $Q^{\circ}$ is given by $m \in M, n \in N$,
	such a configuration $x \subset (0,1) \times [\epsilon,1-\epsilon]$, and a
	function $l: x \to X$. We define a continuous function $g:Q^{\circ}\to
	B[M,\Hur^{X},N]$ as we describe next. 
	
	We first continuously choose a homotopy
	class of paths $\gamma$ from the configuration $x$ to $\mathrm{bigord}$ as
	follows: given a configuration $x$, because the projection $\pi_2|_x : x
	\to [0,1]$ to the second
	coordinate is injective, we can linearly move all points so that their
	first coordinate is $\frac 1 2$. Then, we can rotate the configuration in
	$[0,1] \times [0,1]$ counterclockwise by $\frac \pi 2$
	so that the second coordinate of all the points is
	$\frac 1 2$, so that the path ends in $\mathrm{bigord}$. This ends our description of $\gamma$.

	Let $x := \{p_1, \ldots, p_k\}$ with $\pi_2(p_1) > \cdots > \pi_2(p_k)$
	and define $x_i := l(p_i)$. 
	The map $g$ then sends the above point $(m,n,x,l:x \to X) \in Q$ to the
	point $(m,n,y) \in B[M,\Hur^{X},N]$, where $y \in
	\operatorname{hurbig}^{X}$ is given by $(x,1,\gamma,x_i)$. We omit the
	straightforward verification that
	this map bijects $Q^{\circ}$ onto the image stated in the proposition statement.
	
	We claim the above map $Q^\circ \to B[M,\Hur^X,N]$ continuously extends
to a map $Q \to B[M,\Hur^X,N]$. To see this,
	we can produce a continuous inward flow $\varphi_{-t}$ on $Q$ by pushing
	points inwards along the vector field $-f(x)\frac{\partial}{\partial x}$
	as in \autoref{construction:flow}. For positive $t$, this flow lands inside
	$Q^{\circ}$, so we can continuously extend the function $g$ to a map $\tilde{g}:Q \to B[M,\Hur^{X},N]$ 
	via the formula $\phi_t \circ g\circ \varphi_{-t}$ for any $t>0$, with
	$\phi_t$ as in \autoref{construction:flow}. It is easy to see that the
	underlying configuration associated to the map $\tilde{g}$ is similar to
	the configuration of the original point, with the only difference being
	that all of the points on the boundary of $[0,1] \times [0,1]$ are removed.
	
	Since the image of $g$ was closed, and $Q^\circ$ is dense in $Q$, every
	element of $Q$ is sent to the same point as some element of $Q^\circ$.
	To show that the equivalence relation induced by $\tilde{g}$ is exactly
	determined by $(i)$ and $(ii)$ from \autoref{definition:quotientmodel}, 
	it is then enough to show that $(i)$ and
	$(ii)$ are satisfied, because one can repeatedly use $(i)$ and $(ii)$ to
	show that every point is equivalent to a point in the image of $Q^\circ$, and the map $g$ is a bijection, so this accounts for the entire equivalence relation.
	
	To show relations $(i)$ and $(ii)$, we can reduce to the case where there is exactly one point on the boundary, since if we prove the relations in this case, by continuity, the general case will follow from continuity by perturbing the other points on the boundary and having them approach the boundary.
	
	We now establish relation $(ii)$ from
	\autoref{definition:quotientmodel}
	in this case, as $(i)$ follows from a
	similar argument. Suppose that we have a point $w \in Q$ where $y$ is
	the only point in the configuration associated to $w$ on the boundary of
	$[0,1] \times [0,1]$. Moreover, assume $y$ lies on $1\times [0,1]$. We
	need to compute the limit of $\phi_s\circ g\circ \varphi_{-t}(w)$ as $s$
	approaches $t$ from below. For the configuration associated to point
	$\varphi_{-t}(w)$, we can choose some $\delta<1$ so that the point
	$y'$ to which $y$ flows in
	in $\varphi_{-t}(w)$ has $\pi_1(y') >\delta$ (where
	$\pi_1(y')$ denotes the first coordinate of $y'$) and is the only point of the
	configuration with first 
	coordinate $\geq \delta$. Let $[\gamma]$ be
	the homotopy class of paths from the configuration in $\varphi_{-t}(w)$ to
	$\mathrm{bigord}$ that is used in defining the map $g$. We can choose a
	different homotopy class of paths by first cutting up $[0,1]\times
	[0,1]$ into $[0,\delta]\times [0,1]$ and $[\delta,1]\times[0,1]$, doing
	the construction as in the map $g$ on each piece\footnote{Really, the rotation used in the construction should be rescaled since $[0,\delta]\times [0,1]$ and $[\delta,1]\times [0,1]$ are not squares.}, 
	and concatenating to
	get a new homotopy class of paths $[\gamma']$. One
	checks that
	the class $[\gamma'][\gamma]^{-1}$ in the braid group is given by
	$\sigma_{k-1}\cdots\sigma_i$, where there are $k$ points of the configuration, and
	where $y$ has the $i$th highest second
	coordinate of all elements in the configuration. We have illustrated checking this in \autoref{figure:boundary}.	Thus $g\circ
	\varphi_{-t}(w)$ is equivalent to the point where the path $[\gamma']$ is
	used instead of $[\gamma]$, and the element of $X^k$ is changed
	according to the action of $\sigma_{k-1}\cdots\sigma_i$. Letting $s$ approach $t$, we see that this is exactly the image of the point as described in relation $(ii)$.
	
	We turn to proving the last statement of the proposition. Let
	$\overline{Q}'$ denote a subspace of $\overline{Q}$ that is the image of
	finitely many components of $Q$ so that the image of $K$ factors through
	$\overline{Q}'$. Each component of $Q$ is compact, and so $\overline{Q}'$ is compact, since it is a quotient of finitely many such components.
	Since
	$B[M,\Hur^X,N]$ is Hausdorff, 
	and a continuous bijection from a compact space to a Hausdorff space is
	a homeomorphism,
	the inclusion of $\overline{Q}' \to \overline{Q}$ is a homeomorphism onto its image. 
	Thus, the unique lift of the map from $K$ to $\overline{Q}'$ (and hence to $\overline{Q}$) is continuous.
\end{proof}

For technical reasons, we will want to work with inductive systems of
topological spaces $(X_{\epsilon})_{\epsilon > 0}$, which we view as presenting the colimit of their homotopy types. For example, we will consider $\overline{Q}_{\epsilon}[M,\Hur^{X},N]$ as an ind-system as $\epsilon$ approaches $0$.

\begin{definition}\label{definition:indweakeq}
	Let $L:\Top \to \Spc$ and $L:\Top_* \to \Spc_*$ denote the functors of $\infty$-categories sending a (pointed) topological space to its weak homotopy type.
	Given a (pointed) map $f:X \to X'$ in $\Ind(\Top)$ or $\Ind(\Top_*)$,
	we say that it is an {\em ind-weak equivalence} if $\colim Lf$ is an equivalence.
	Similarly, we say that $f$ is an {\em ind-weak homology equivalence} if $\colim Lf$
	is a homology equivalence.
\end{definition}

The following is an ind-variant of a well known lemma about weak homotopy
equivalences. In what follows, we use the smash product $\wedge$, the definition
of which is in \autoref{subsubsection:spaces}.

\begin{lemma}\label{lemma:compacthtyeq}
	Let $P$ be a filtered poset, and let $X'_\bullet \hookrightarrow X_\bullet, \bullet \in P$ be an injection of $P$-indexed diagrams of path connected pointed topological spaces (not necessarily with the subspace topology) such that for every pointed inclusion of compact spaces $K' \hookrightarrow K$, and commutative diagram 
	\begin{equation}
		\label{equation:k-inclusion}
		\begin{tikzcd}
			K'\ar[hook]{r}{i}
			\ar{d}{\psi} &K \ar[d]\\
			X'_{\alpha}\ar[r,hook] & X_{\alpha}
		\end{tikzcd}
	\end{equation}
	there is a $\beta>\alpha$ and a pointed homotopy $H:K\wedge [0,1]_+ \to
	X_\beta$ with the following two properties.
	\begin{enumerate}
		\item 	The restriction of $H$ along $i \wedge \id$ induces a pointed homotopy
	$K'\wedge[0,1]_+ \to X'_{\beta}$ whose restriction to the point $0 \in
	[0,1]_+$ is the composite $K' \xrightarrow{\psi} X'_\alpha \to X'_\beta$
	where $\psi$ is the map from \eqref{equation:k-inclusion}.
\item
	The restriction of $H$ to the point $1 \in [0,1]_+$ defines a map $K \to
	X_{\beta}$, which we assume factors continuously through $X'_{\beta}$.
	\end{enumerate}
	Then $X'_{\bullet} \to X_\bullet$ is an ind-weak homotopy equivalence.
\end{lemma}
\begin{proof}
	Since $\pi_i(-) = \pi_{i-1}(\Omega -)$ and homotopy sets and based loop spaces
	commute with filtered colimits in $\Spc$, it suffices by induction on
	$i$ to show that $\Omega X'_{\bullet} \hookrightarrow{\Omega
	X_{\bullet}}$ satisfies the same hypotheses $(1)$ and $(2)$ as in the
	statement that $X'_\bullet \to X_\bullet$ satisfies
(except that $\Omega X'_\bullet$ and $\Omega X_\bullet$ need not be
	connected), and moreover that $\Omega X'_{\bullet} \hookrightarrow{\Omega
	X_{\bullet}}$ is an ind-equivalence on $\pi_0$.
	
	To see that the map
$\Omega X'_{\bullet} \hookrightarrow{\Omega
	X_{\bullet}}$
	satisfies $(1)$ and $(2)$, note that giving a pointed map $K \to \Omega
	X_{\alpha}$ is the same as giving a map $K\wedge S^1 \to X_{\alpha}$,
	and similarly for $K'$. 
	Here, we implicitly choose a fixed point in $S^1$ and view $S^1$ as a
	pointed space to make sense of the wedge product.
	Applying the hypotheses to $K\wedge S^1$ in
	place of $K$, 
	we get an pointed map $K\wedge S^1 \wedge [0,1]_+\to X_{\beta}$, 
	which is the same as a map $K\wedge[0,1]_+\to \Omega X_{\beta}$
	satisfying the desired properties $(1)$ and $(2)$.
	
	We finish the proof by showing that show that $\colim_{\alpha \in P}\pi_0\Omega X'_{\alpha} \to \colim_{\alpha \in P}\pi_0\Omega X_{\alpha}$ is a bijection. 
	We first check it is surjective.
	Given a point $* \in \Omega X_{\alpha}$, taking $K$ to be the compact
	space $*_+$ and $K' \subset K$ the base point,
	hypotheses $(2)$ 
	applied to the pointed map $*_+ \to \Omega
	X_{\alpha}$ shows surjectivity. For injectivity, if $[0,1]_+ \to \Omega
	X_{\alpha}$ is a path between two points with the end points factoring
	through $\Omega X'_{\alpha}$, then applying hypothesis $(1)$ 
	to this map, we get a map $[0,1]_+ \wedge [0,1]_+ = ([0,1]\times [0,1])_+\to \Omega X_{\beta}$, 
	and the part of the boundary $\{0,1\}\times [0,1]\cup [0,1]\times \{1\}$
	yields a path in $\Omega X'_{\beta}$ between the two points.
\end{proof}

The following scanning argument is a ind-version of a quite standard scanning argument (see for example \cite[Lemma
5.2]{randal-williams:homology-of-hurwitz-spaces} or \cite[Proposition 5.2.7]{bergstromDPW:hyperelliptic-curves-the-scanning-map}),
and is illustrated in \autoref{figure:scanning}.

\begin{figure}
	\includegraphics[scale=.5]{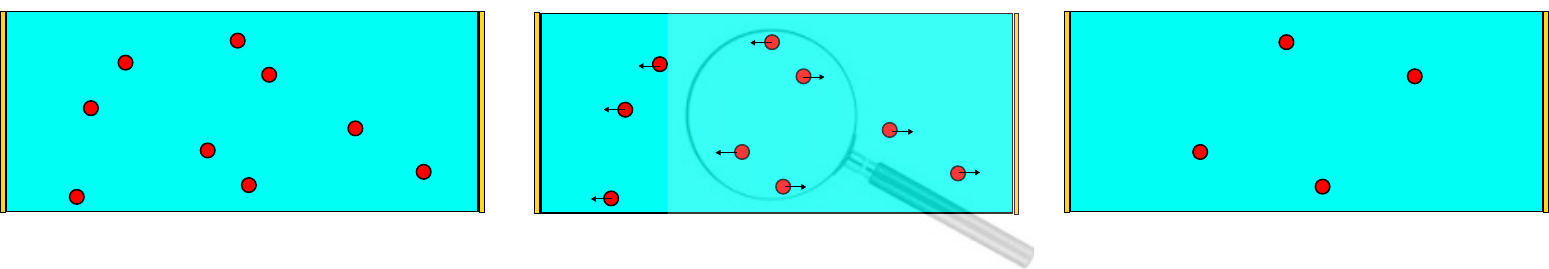}
\caption{This is a picture of the scanning argument of \autoref{lemma:scanningunpointed}.
	In the diagrams, there vertical distance at least $\varepsilon$ between any
	two red points,
	and
	all red points have second coordinate lie in 
	$[\epsilon, 1-\epsilon]$.
}
\label{figure:scanning}
\end{figure}

\begin{lemma}\label{lemma:scanningunpointed}
	The inclusions $\overline{Q}_\epsilon[M,\Hur^{X},N] \to B[M,\Hur^{X},N]$
	of  \autoref{proposition:modelquotdescription} indexed over the poset of
	real numbers $1>\epsilon>0$ are an ind-weak equivalence.
\end{lemma}

\begin{proof}
	It is easy to see that this map is always an equivalence on path
	components for each $\frac 1 2 \geq \epsilon>0$, and that every element is in the
	component of a point given by $(m,n) \in M\times N$ and an empty
	configuration of points.
	It suffices to prove the statement for each connected path component, so we
	now fix a path component $Z$, 
	restrict to $\epsilon\leq\frac 1 2$, and fix some $(m,n)$ on this
	component with empty configuration as the base point for all
	$\epsilon$. We will now apply \autoref{lemma:compacthtyeq} to the path
	component $Z$ of this base point. 
	Consider a diagram 	
	\begin{center}
		\begin{tikzcd}
			K'\ar[r,hook]\ar[d] &K \ar[d,"f"]\\
			\overline{Q}_\epsilon[M,\Hur^{X},N]\ar[r,hook] & B[M,\Hur^{X},N]
		\end{tikzcd}
	\end{center} with $K,K'$ compact. Choose a $\delta,\epsilon>0$ such that
	for all points in the image of $f : K \to B[M, \Hur^X, N]$, the part of the configuration in
	$[\frac 1 2-\delta,\frac 1 2+\delta]\times [0,1]$ projects injectively
	under the second projection and lands in $[\epsilon,
	1-\epsilon]$, and each two points in the
	image of this projection have distance at least $\epsilon$ from each
	other.
	Since the flow in \autoref{construction:flow} sends $\frac 1 2 \pm \delta$
	to the boundary in a finite amount of time, we can use the flow to
	produce a continuous homotopy starting from $f$ and ending with a map
	landing in the image of $\overline{Q}_{\epsilon}[M,\Hur^{X},N]$. 

	We conclude by verifying the two conditions of
	\autoref{lemma:compacthtyeq}.
	First, we verify condition (2) from \autoref{lemma:compacthtyeq}.
	Indeed, since $K$ is compact, at time $t=1$ of the homotopy, the map satisfies the criterion of the last sentence of \autoref{proposition:modelquotdescription}, so factors through $\overline{Q}_{\epsilon}[M,\Hur^{X},N]$.
	Finally, we verify condition $(1)$ of \autoref{lemma:compacthtyeq}
	by using the flow on $\overline{Q}_{\epsilon}[M,\Hur^{X},N]$. 
	Indeed, the homotopy also
	restricts on $K'$ to a homotopy of pointed maps to
	$\overline{Q}_{\epsilon}[M,\Hur^{X},N]$. This verifies the conditions of
	\autoref{lemma:compacthtyeq}, as desired.
	\end{proof}

We turn to proving a pointed version of \autoref{lemma:scanningunpointed}. As a first step, we prove a lemma which shows that certain pushouts of our topological models are homotopy pushouts.

\begin{lemma}\label{lemma:cofibration}
	Let $M$ be a right $\Hur^{X}$ module and $N$ be a left $\Hur^{X}$
	module. 
	Suppose that $Y \subset M\times N$ is a subset closed under the right
	and left actions of $\Hur^{X}$, and that $X' \to X$ is an inclusion of braided sets.
	Let
	$A\subset \overline{Q}_\epsilon[M,\Hur^{X},N]$ denote the subset of points in the image of points in $Q_{\epsilon}[M,\Hur^{X'},N]$ whose
	projection to $M\times N$ lies in $Y$.
	Then $A\subset \overline{Q}_\epsilon[M,\Hur^{X},N]$
	has the homotopy extension property.
\end{lemma}

\begin{proof}
	 It is easy to see that for each $1>\epsilon>0$, $\overline{Q}_{\epsilon}[M,\Hur^X,N]$ admits the structure of a CW-complex of dimension $<\frac 2 \epsilon$.
	 This is because each component of $Q_{\epsilon}[M,\Hur^X,N]$ is a union of discs, and the gluing relations to form $\overline{Q}_{\epsilon}[M,\Hur^X,N]$ are along subcomplexes. It is finite dimensional since $\epsilon$ puts a bound on the number of points in the configurations that are allowed.
	 For a suitably chosen
CW structure on $\overline{Q}_{\epsilon}[M,\Hur^X,N]$, 
the inclusion $A\subset \overline{Q}_\epsilon[M,\Hur^{X},N]$
is that of a subcomplex, showing the lemma.
\end{proof}

We are finally prepared to prove a pointed version of \autoref{lemma:scanningunpointed}, which is the main result of this section:
\begin{theorem}\label{lemma:modelquotdescriptionpointed}
	Let $M$ be a pointed set with a right $\Hur^X$ action and $N$ be a
	pointed set with a left $\Hur^{X}$ action. There is a natural
	identification of $M\otimes_{\Hur^{X}_+}N$ with the ind-weak homotopy
	type of the quotient of the pointed spaces
	$$\overline{Q}^*_\epsilon[M,\Hur_+^{X},N]:=
	\overline{Q}_\epsilon[M,\Hur^{X},N]/S_{M,X,N},$$ as $\epsilon$ approaches $0$ (for $0<\epsilon<1$) where $S_{M,X,N} \subset \overline{Q}_\epsilon[M,\Hur^{X},N]$ is the subspace
	consisting of all points whose projection to either $M$ or $N$ is the
	base point.
\end{theorem}

\begin{proof}
		The pointed bar construction $M\otimes_{\Hur^{X}_+}N$ is the total cofiber of the square
		\begin{equation}
			\label{equation:bar-square}
				\begin{tikzcd}
						*\otimes_{\Hur^{X}}*\ar[r]\ar[d] &M\otimes_{\Hur^{X}}* \ar[d]\\
						*\otimes_{\Hur^{X}}N	\ar[r] & M\otimes_{\Hur^{X}}N.
					\end{tikzcd}
			\end{equation}
		In other words, $M\otimes_{\Hur^{X}_+}N$ is 
		the cofiber of the map from the pushout of the diagram
		
		$$M\otimes_{\Hur^{X}}*\leftarrow*\otimes_{\Hur^{X}}*\rightarrow*\otimes_{\Hur^{X}}N	$$
		to $M\otimes_{\Hur^{X}}N$.
		
		By \autoref{lemma:maincomparison} and \autoref{lemma:scanningunpointed},
		$M'\otimes_{\Hur^{X}}N'$ can be modeled as the colimit over
		$\epsilon>0$ of the homotopy types of $Q_{\epsilon}[M',\Hur^{X},N']$
		for any unpointed right module $M'$ and unpointed left module $N'$. It then follows by applying
		\autoref{lemma:cofibration}, that 
		the maps
		$$\overline{Q}_{\epsilon}[*,\Hur^{X},*] \to \overline{Q}_{\epsilon}[*,\Hur^{X},N]$$
		$$\overline{Q}_{\epsilon}[*,\Hur^{X},*] \to \overline{Q}_{\epsilon}[M,\Hur^{X},*]$$
		are cofibrations, i.e., they both have the homotopy extension
		property.
		In turn, this implies the map
		$$\overline{Q}_{\epsilon}[*,\Hur^{X},N]\cup_{\overline{Q}_{\epsilon}[*,\Hur^{X},*]}\overline{Q}_{\epsilon}[M,\Hur^{X},*] \to \overline{Q}_{\epsilon}[M,\Hur^{X},N]$$
		is a cofibration, where the $\cup$ notation denotes the pushout.
		Finally, this implies
		the total cofiber of the square
		\begin{equation}
			\label{equation:q-epsilon-square}
				\begin{tikzcd}
						\overline{Q}_{\epsilon}[*,\Hur^{X},*]\ar[r]\ar[d] &\overline{Q}_{\epsilon}[M,\Hur^{X},*] \ar[d]\\
						\overline{Q}_{\epsilon}[*,\Hur^{X},N]	\ar[r] & \overline{Q}_{\epsilon}[M,\Hur^{X},N]
					\end{tikzcd}
			\end{equation}
		can be modeled as the quotient of
		$\overline{Q}_\epsilon[M,\Hur^{X},N]$ by $S_{M,X,N}$.
		The total cofiber of \eqref{equation:bar-square} is the colimit over $\epsilon$
		of the total cofibers of the squares
		\eqref{equation:q-epsilon-square} for $\varepsilon > 0$,
		since the total cofiber is a colimit, so commutes with the
		filtered colimit over $\varepsilon$. Thus
		we obtain the
		desired result. 
\end{proof}

\bibliographystyle{alpha}
\bibliography{./bibliography}

\end{document}